\documentclass[preprint]{imsart}
\usepackage{stmaryrd}
\usepackage{multirow,mathtools,latexsym, enumerate, enumitem, amsmath, amsthm, amssymb, pdftricks,tikz,framed,color,enumitem,enumerate,mleftright}

\usetikzlibrary{arrows.meta, decorations.markings,math} %needed tikz libraries

\pdfoutput=1

\usepackage{soul}
\startlocaldefs

\def\x{{x}}
\def\y{{y}}
\def\Y_#1{y_{\!#1}}

\def\R{\mathbb{R}}
\def\Z{\mathbb{Z}}

\def\TV{\mathrm{TV}}

\newcommand{\lb}[1]{{\lfloor #1 \rfloor}}
\newcommand{\ub}[1]{{\lceil #1 \rceil}}
\endlocaldefs
\newtheorem{theorem}{Theorem}[section]
\newtheorem{corollary}[theorem]{Corollary}
\newtheorem{lemma}[theorem]{Lemma}

\newtheorem{remark}[theorem]{Remark}

\newtheorem{assumption}{Assumption}[section]

\usepackage{tikz}
\usepackage{romanbar}

\theoremstyle{definition}

\newtheorem{example}{Example}[section]

\theoremstyle{remark}

\definecolor{darkred}{rgb}{0.9,0.1,0.1}

\newcommand{\rn}[1]{\Romanbar{#1}}

\newcommand{\n}{\mathbb{N}}

\renewcommand{\tilde}{\widetilde} 
\newcommand{\norm}[1]{\left\lvert #1 \right\rvert} %euclidean norm

\begin{document} 

\title{Coupling and Convergence for\\  
Hamiltonian Monte Carlo\thanksref{a2}}
\runtitle{Coupling \& Convergence for HMC}

\thankstext{a2}{N.B-R. was supported in part by the Provost's Fund for Research at Rutgers
University Camden under Project No.~205536, and also in part by the NSF Research Network in Mathematical Sciences: 
``Kinetic description of emerging challenges in multiscale problems of natural sciences'', (PI: Eitan Tadmor; NSF Grant \#: 11-07444). A.E.\ and R.Z.\ have been supported by the Hausdorff Center for Mathematics. All authors thank the referees for valuable comments.}

 \begin{aug}
 \author{\fnms{Nawaf} \snm{Bou-Rabee}\thanksref{m1}\ead[label=e1]{nawaf.bourabee@rutgers.edu}}
 \and
 \author{\fnms{Andreas}
 \snm{Eberle}\thanksref{m2}\ead[label=e2]{eberle@uni-bonn.de}}
\\
\and
\author{\fnms{Raphael}
\snm{Zimmer}\thanksref{m2}\ead[label=e3]{raphael.zimmer@uni-bonn.de}}

 \address{\thanksmark{m1}Department of Mathematical Sciences \\ Rutgers University Camden \\ 311 N 5th Street \\ Camden, NJ 08102, USA \\ \printead{e1}}
 \address{\thanksmark{m2}Institut f\"{u}r Angewandte Matehmatik \\  Universit\"{a}t Bonn \\ Endenicher Allee 60 \\  53115 Bonn, Germany \\ \printead{e2}\\  \printead{e3}}
 
% \thankstext{t1}{N.B-R. was supported in part by the Provost's Fund for Research at Rutgers
%University Camden under Project No.~205536, and also in part by the NSF Research Network in Mathematical Sciences: 
%``Kinetic description of emerging challenges in multiscale problems of natural sciences'', (PI: Eitan Tadmor; NSF Grant \#: 11-07444).}
% \thankstext{t2}{A.E. and R.Z. have been supported by the Hausdorff Center for Mathematics and the Bonn International Graduate School of Mathematics (BIGS)}
 
 \runauthor{N.~Bou-Rabee, A.~Eberle, R.~Zimmer}
 \end{aug}

 \begin{abstract}
Based on a new coupling approach, we prove that the transition step of the Hamiltonian Monte Carlo algorithm is contractive w.r.t.\ a carefully designed Kantorovich ($L^1$ Wasserstein) distance. The lower bound for the contraction rate is explicit. Global convexity of the potential is not required, and thus multimodal target distributions are included. Explicit quantitative bounds for the number of steps required to approximate the stationary distribution up to a given error $\epsilon$ are a direct consequence of contractivity. These bounds show that HMC can overcome diffusive behaviour if the duration of the Hamiltonian dynamics is adjusted appropriately.
 \end{abstract}

\begin{keyword}[class=MSC]
\kwd[Primary ]{60J05}
\kwd[; Secondary ]{65P10, 65C05}
\end{keyword}

% Primary:
% 60J05 Markov Processes with Discrete Parameter
% Secondary:
% 65P10  Hamiltonian systems including symplectic integrators
% 65C05 Monte Carlo Methods

\begin{keyword}
\kwd{coupling}
\kwd{convergence to equilibrium}
\kwd{Markov Chain Monte Carlo}
\kwd{Hamiltonian Monte Carlo}
\kwd{Hybrid Monte Carlo}
\kwd{geometric integration}
\kwd{Metropolis-Hastings}
\end{keyword}

\maketitle

%%%%%%%%%%%%%%%%%%%%%%%%%%%%%%%%%%%%%%%%%%%%%
%%%%%%%%%%%%%%%%%%%%%%%%%%%%%%%%%%%%%%%%%%%%%

\section{Introduction}

Markov Chain Monte Carlo (MCMC) is a family of methods to approximately sample from an arbitrary probability distribution.  In conjunction with Bayesian methods, MCMC has revolutionized statistics and enabled applications of statistical inference to machine learning, pattern recognition, and artificial intelligence \cite{andrieu2003introduction,ghahramani2015probabilistic,webb2003statistical,bishop2006pattern,Di2009}.    Much of the classical research activity related to MCMC considered techniques based on random walks.  Regrettably the meandering behavior of these random walks leads to sampling methods that are slow \cite{Gu1998,DiHoNe2000}.  Therefore, a more recent focus of MCMC research activity is to develop faster methods by overcoming this random walk or diffusive behavior. Hamiltonian Monte Carlo (HMC), in principle, provides one way to do this \cite{DuKePeRo1987,Li2008,Ne2011,Sa2014,BoSaActaN2018}.  The basic idea is to ``give the walker momentum''.  With this momentum, the walker intermixes periods of fast running with slower running to efficiently explore features of a probability distribution.  However, beyond this intuition, the mathematical properties of HMC are not well understood.

In this paper, we use a coupling technique \cite[Ch.~14]{LPW} to analyse the HMC algorithm.  
Let $\mu$ denote a probability measure on $\mathbb{R}^d$ with non-normalized density $e^{-U(x)}$. In HMC, the function $U(x)$ is viewed as a potential energy. In its simplest form, 
the algorithm simulates a Markov chain with state space $\mathbb R^d\times \mathbb R^d$ and invariant probability measure  $\hat\mu$ with non-normalized density $e^{-H(x,v)}$ where $H(x,v) = \frac{1}{2} |v|^2 + U(x) $ is viewed as a Hamiltonian. Below we only consider the first component which is a Markov chain on $\mathbb R^d$ with invariant probability measure $\mu$.
A transition step of HMC inputs an initial position $x \in \mathbb{R}^d$ and a duration parameter $T>0$, and outputs a final position by taking 
the following steps:
\begin{description}
\item[Step 1.] Draw an initial velocity $\xi\sim \mathcal N(0,I_d)$.
\item[Step 2.] Run the Hamiltonian dynamics associated to the Hamiltonian function $H $ for a duration $T$ with initial position $x$ and initial velocity $\xi$.
\item[Step 3.] Output the final position of this Hamiltonian dynamics.
\end{description}   
We call the algorithm with this transition step \emph{exact HMC}.
 In practice, the Hamiltonian dynamics in Step 2 is approximated by a numerical integrator. Furthermore, an accept/reject step can be added to remove the bias due to time discretization error. In order to distinguish these from exact HMC, we call the resulting algorithms \emph{unadjusted and adjusted} numerical HMC, respectively.

Below, we introduce a new coupling between the transition steps of two copies of exact HMC, or numerical HMC, respectively. 
The approach we use is based on the framework introduced in \cite{Eb2016A}, and the specific coupling of the velocities is
strongly inspired by a recently developed coupling for second-order Langevin dynamics \cite{EbGuZi2016}.
The underlying idea is to couple two copies of HMC at different positions $x$ and $y$ by coupling their velocities $\xi$ and $\eta$ such that for a positive constant $\gamma\le T^{-1}$, the event $ \xi - \eta =-\gamma (x-y)$ happens with maximal probability, and otherwise, to apply a reflection coupling to the velocities.  In particular, the coupling is designed such that in the free case when $U= 0$, the positions approach each other during the transition step with maximal probability. We use this contractive property of the coupling to obtain an explicit contraction rate for HMC in a specially designed Kantorovich ($L^1$ Wasserstein) metric.

To be more specific, we state a simpified version of one of our main results, which will later be reformulated rigorously as Corollary~\ref{cor:QBHMC} --- a corollary of the fundamental contraction result in Theorem~\ref{thm:contrmainexact}.  Let $\pi(x,dy)$ denote the one-step transition kernel of exact HMC and let $\mathcal W^1$ denote the standard $L^1$-Wasserstein distance.  Assuming sufficient regularity on the potential energy function $U$ (see Assumption~\ref{A123}) including that $\nabla U$ is globally Lipschitz with Lipschitz constant $L$, and $U$ is strongly convex outside a Euclidean ball of diameter $\mathcal R$ with strong convexity constant $K$, we prove that if \[
LT^2\ \le\ \min \left( \frac{K}{L},\frac{1}{4},\frac{1}{256\,L\mathcal R^2}\right)
\]  then for all initial distributions $\nu$ and $\eta$, and for all $n\ge 0$, \[
\mathcal W^1 (\nu \pi^n ,\eta \pi^n )\ \le\ Me^{- cn} \mathcal W^1 (\nu ,\eta) ,\qquad\quad\text{where}
\] 
 \[
c\ =\ \frac{1}{10}\min\left( 1,\frac 12 KT^2 ( 1+\frac{\mathcal R}{T}) e^{-\mathcal R/(2T)}\right)\, e^{-2\mathcal R/T}, \quad\text{and}\quad M=e^{\frac 52(1+\mathcal R/T)}.
\]
More precisely, we prove in Theorem~\ref{thm:contrmainexact} that the transition kernel $\pi$ is even contractive with contraction rate $c$ w.r.t.\ an $L^1$ Wasserstein distance $\mathcal W_\rho$ based on an explicit metric $\rho$ that is equivalent to the Euclidean distance. This statement can be used to quantify the speed of convergence of HMC to equilibrium, and it also directly implies completely explicit bias and variance bounds, as well as concentration inequalities for ergodic averages, see e.g.\  \cite{JoulinOllivier}.

A remarkable feature of the contraction rate $c$ is that 
under our hypothesis on $LT^2$, it only depends on $K$ and  $\mathcal{R} / T$.  
In particular, it does not depend explicitly on the dimension, although frequently there will be a hidden
dimension dependence through the values of the constants $K$ and $\mathcal R$.
If we choose $T$ proportional to $\mathcal{R}$, and assume that $K$ and $L \mathcal{R}^2$ are fixed (which excludes the possibility of high energy barriers), then the rate does not 
deteriorate as $\mathcal{R}$ increases. %or the dimension $d$ increase.
Noting that the Hamiltonian dynamics is run for time $T$ during each transition step,
we can conclude that a given approximation accuracy can be obtained after running the dynamics for a total time of kinetic order $O(\mathcal{R})$, 
where $\mathcal{R}$ basically is the diameter of a ball where the distribution $\mu$ concentrates in.
On the other hand, a Random Walk based method would require a time of diffusive order $O(\mathcal R^2)$. 
Hence if $T$ is chosen adequately then HMC can indeed overcome diffusive behaviour.

In Theorems \ref{thm:maincontraction} and \ref{thm:QBNHMC}, we extend our results to numerical HMC with a velocity Verlet integrator. The corresponding results are more involved than for exact HMC, but the bound $c$ for the contraction rate is the same provided the time discretization step size $h$ is chosen sufficiently small depending on the other parameters, including the dimension. We consider both  adjusted and unadjusted numerical HMC. In the unadjusted case, the contraction bounds are easier to derive and take a nicer form that
is close to the corresponding results for exact HMC, but the price to pay is an additional bias
due to the fact that the invariant measure for unadjusted HMC does not coincide with $\mu$.
The resulting error has to be controlled by other techniques that are out of the scope of this 
work, see \cite{mangoubi2017rapid,durmuseberle2019}. In Theorems \ref{thm:contrlyapexact}
and \ref{thm:lyapcontraction},
we also state versions of our main
results where the asymptotic strong convexity assumption is replaced by a Foster-Lyapunov drift condition, which permits $U$ that are only asymptotically convex.  
\smallskip

Several recent works have studied ergodic properties of HMC methods.
In \cite{BoSa2016}, geometric ergodicity has been proven for a variant of exact HMC (called randomized HMC) where the lengths of the durations of the Hamiltonian dynamics at the different transitions of the Markov chain are 
i.i.d.\ exponential random variables with mean $T$.  The proof relies on Harris' theorem, which requires a (local) version of Doeblin's condition: a minorization condition for the transition probabilities at a finite time and in a compact set. Unfortunately,
given the complicated form of these transition probabilities, the minorization condition involves non-explicit constants, and in particular, the dependence of the convergence rate on parameters in HMC is unclear. 
We remark that randomized HMC is related to Anderson's dynamics, which describes a molecular system interacting with a heat bath \cite{An1980, Li2007,  ELi2008}.    Convergence of Anderson's dynamics on an $n$-torus was proven in Ref.~\cite{ELi2008} by showing that Doeblin's condition holds.  Recently, geometric ergodicity for HMC, but without explicit rates, has been shown 
by Durmus, Moulines and Saksman \cite{DurmusMoulinesSaksman}, cf.\ also \cite{Livingstone} for a related work. 

Closely related to our results is 
\cite{mangoubi2017rapid}, which significantly extends ideas from \cite{Seiler}. In \cite{mangoubi2017rapid}, Mangoubi and Smith apply
coupling techniques in order to analyse the properties of 
HMC in high dimension under the assumption of strong convexity of $U$; see also \cite{cheng2017underdamped} for related work on second-order Langevin dynamics. 
They obtain quite sharp results on the dimension dependence 
caused by different numerical integrators of the Hamiltonian dynamics, at the price of
imposing very restrictive assumptions on $U$. Complementary to this, the approach presented here is more broadly applicable and offers significant flexibility for further extensions.
It provides a suitable basic framework for studying dimension dependence caused by intrinsic 
properties of the underlying model, but in contrast to \cite{mangoubi2017rapid}, we
do not provide sharp bounds for the dimension dependence caused by time
discretization. A major difference of our approach to \cite{mangoubi2017rapid,cheng2017underdamped} is that these works rely on
synchronous couplings of the initial velocities in HMC, i.e., they set $\eta = \xi$.  
This simplifies the analysis considerably, but as a
consequence, the couplings are contractive only if the stationary distribution is strongly log-concave.  Another difference is that in \cite{mangoubi2017rapid} and \cite{cheng2017underdamped}, the coupling is applied to the exact dynamics, 
whereas the numerical discretization is controlled by a perturbative approach. In
contrast, the coupling introduced below is contractive both for exact and numerical HMC. 
Its superiority to synchronous couplings is supported both by theoretical results and by numerical simulations.
In connection with \cite{heng2017unbiased}, the coupling may also be useful to parallelize HMC.\smallskip  

Let us finally remark that the Hamiltonian flow is what, in principle, enables HMC to make large moves in state space that reduce correlations in the resulting Markov chain.
One might hope that,  by increasing the duration $T$ further, the final position moves even further away from the initial position, thus
reducing  correlation.  However, simple examples show that this outcome is far from assured.  For example, for a standard normal distribution,
the corresponding Hamiltonian flow is a planar rotation with period $2\pi$. It is easy to see  that, if the initial position is taken from the target distribution, as $T$ increases from $0$ to $\pi/2$, the correlation between the initial and final positions decreases and for $T = \pi/2$, the initial and final positions are independent. However increasing $T$ beyond $\pi/2$ will
 cause an increase in the correlation and for $T = \pi$, the chain is not even ergodic.  For general distributions, it is likely that a small $T$ will lead to a highly correlated chain, while choosing $T$ too large  may cause the Hamiltonian trajectory to make a U-turn and fold back on itself, thus increasing correlation \cite{HoGe2014}. Generally speaking the performance of HMC may be very sensitive to changes in $T$ as first noted by Mackenzie in \cite{Ma1989}.  This sensitivity is reflected in our conditions on the duration parameter $T$.

\section{Main results} \label{sec:main_results}

\subsection{Hamiltonian Monte Carlo} \label{sec:main_results:HMC}
Fix a function
$U\in C^4(\mathbb R^d)$ such that $\int\exp(-U(x))\, dx <\infty$, and denote by
\begin{equation}\label{eq:H}
H(x,v)\ =\ U(x)\,+\, \frac{1}{2}\norm{v}^2,\qquad x,v\in\mathbb R^d,
\end{equation}
the corresponding Hamiltonian. 
Hamiltonian Monte Carlo (HMC) is an MCMC method for approximate sampling
from probability measures of the form
\begin{equation}\label{eq:mu}
\mu (dx)= \mathcal Z^{-1}\,\exp (-U(x))\, dx,\quad
\hat\mu (dx\, dv) = \hat{\mathcal Z}^{-1}\,\exp (-H(x,v))\, dx\, dv,
\end{equation}
on $\mathbb R^d$, $\mathbb R^d\times \mathbb R^d$, respectively, where
$\mathcal Z=
\int\exp (-U(x))\,dx$ and $\hat {\mathcal Z}=(2\pi )^{d/2}\mathcal Z$.\smallskip

We consider HMC as a Markov chain on $\mathbb R^d$ (not on the 
phase space $\mathbb R^d\times\mathbb R^d$). The transition step from $x$
is given by $x\mapsto X'(x)$ with
 \begin{eqnarray}
 	\label{*} X'(x) &=& q_T(x,\xi) \,I_{A(x)} \ + \ x \,I_{A(x)^C}.
 	\end{eqnarray}
Here the duration
$T:\Omega\rightarrow \R_+$ is in general a random variable on the underlying probability space $(\Omega ,\mathcal A,P)$ with a given distribution
$\nu$ (e.g.\ $\nu=\delta_s$ or $\nu=\text{Exp}(\lambda^{-1})$), $\xi\sim
N(0,I_d)$ and $\mathcal U\sim \text{Unif}(0,1)$ are independent random variables, and the acceptance event for a proposed transition is either
 \begin{equation} 	
 	A(x) = \{\, \mathcal U \, \leq \, \exp\left(H(x,\xi)-H(q_T(x,\xi),p_T(x,\xi))\
 	\right) \,\} \quad \text{or}\quad A(x) = \Omega ,
 	\label{*A} 
 \end{equation}  
corresponding to adjusted and unadjusted HMC, respectively.
Below we only consider the case where $T\in (0,\infty)$ is a given deterministic constant. Moreover,
\begin{eqnarray*}
	\phi_t(x,v) &=& \left(q_t(x,v), p_t(x,v)\right)\qquad ( t\in [0,\infty ),\ x,v\in\mathbb R^d)
\end{eqnarray*}
is the exact Hamiltonian flow, or a numerical approximation of the Hamiltonian flow.
The exact Hamiltonian flow is the solution of the ODE
\begin{eqnarray}\label{eq:Hamilton}
	\frac{d}{dt} q_t \ = \ p_t,\quad \frac{d}{dt} p_t \ = \ -\nabla U(q_t),\quad
	\left(q_0(x,v),p_0(x,v)\right) \ = \ (x,v).
\end{eqnarray}
The corresponding Markov chain with transition step determined by \eqref{*}, \eqref{*A} and \eqref{eq:Hamilton} is called \emph{exact HMC}. Notice that  for exact HMC,
$H(q_T(x,\xi),p_T(x,\xi))=H(x,\xi)$. Hence all proposed transitions are accepted without adjustment, and the transition step is simply given by 
\begin{equation}
\label{eq:10}X'(x)\ =\ q_T(x,\xi).
\end{equation}
In practice, the Hamiltonian flow has to be approximated by a numerical integrator. Here, we focus on the \emph{velocity Verlet} integrator with discretization step size $h>0$. In this case, $\phi_t=(q_t,p_t)$ is the solution
of the equation 
\begin{equation}
\label{velVerlet}
	\frac{d}{dt} q_t   \, = \,   p_{\lb{t}_h}-\frac{h}{2} \ \nabla
	U(q_{\lb{t}_h}),\ \frac{d}{dt} p_t \, = \, -\frac{1}{2}\left(\nabla
	U(q_{\lb{t}_h})\ +\ \nabla U(q_{\ub{t}_h})\right)
\end{equation}
with initial condition $\left(q_0(x,v),p_0(x,v)\right) = (x,v)$, where
\begin{equation}
\label{eq:round}
\lb{t}_h= \max\{ s\in h\Z \ :\ s\leq t \} \mbox{\quad and\quad} \ub{t}_h= \min\{ s\in
h\Z \ :\ s\geq t \} .
\end{equation}
The corresponding Markov chain with transition step determined by \eqref{*}, \eqref{*A} and \eqref{velVerlet} is called \emph{adjusted resp.\ unadjusted} numerical HMC. 
For brevity, whenever $h>0$ is fixed, we write $\lb{t}$ and $\ub{t}$
instead of $\lb{t}_h$ and $\ub{t}_h$, respectively. 
Since the velocity Verlet integrator does not preserve the Hamiltonian exactly, 
the rejection event $A(x)^C$ is not empty in general for adjusted numerical HMC. However, for fixed $x$, the rejection probability goes to $0$ as $h\downarrow 0$.
\medskip

The HMC algorithm induces a time-homogeneous Markov chain on $\mathbb R^d$ with transition kernel
\begin{eqnarray}
\label{eq:transitionkernel}	\pi (x,B) &=& P[X'(x)\in B] \\
\nonumber 	&=& P[\{q_T(x,\xi)\in B\}\cap A(x)] \ + \ (1-P[A(x)]) \ \delta_x(B).
\end{eqnarray}
Here $1-P[A(x)]$ is the rejection
probability for a proposed transition from $x$. For exact and adjusted numerical HMC, the probability measure $\mu$ defined by
\eqref{eq:mu} is invariant for $\pi $ under the assumptions made below, cf.\ e.g.\ \cite{BoSaActaN2018,Ne2011}. For unadjusted numerical HMC, the invariant probability measure for $\pi$ in general does not agree with $\mu$, but when it exists, typically approaches $\mu$ as $h\downarrow 0$.

\subsection{Assumptions} \label{sec:main_results:setting} 
For our first main result we impose the following regularity condition on $U$:
 \begin{assumption}\label{A123}
 $U$ is a function in $C^4(\mathbb R^d)$ satisfying the following conditions:
 	\begin{itemize}
 	\item[(A1)] $U$ has a local minimum at $0$, and $U(0)=0$.
 	\item[(A2)] $U$ has bounded second, third and fourth derivatives. We set
 \begin{equation}
 \label{eq:LMN}L=\sup\|\nabla^2U\| ,\quad M=\sup\|\nabla^3U\| ,\quad N=\sup\|\nabla^4U\|  .
 \end{equation}
 \item[(A3)] $U$ is strongly convex outside a Euclidean ball, i.e., there exist constants $\mathcal R\in [0,\infty)$ and $K\in(0,\infty)$
	s.t.\ for all $x,y\in\R^d$ with $\norm{x-y}\geq \mathcal R$,
	\begin{eqnarray}\label{15} 
		(x-y)\cdot(\nabla U(x)-\nabla U(y))  & \geq &   K \norm{x-y}^2.
	\end{eqnarray}
 	\end{itemize}
 	 \end{assumption}

Notice that (A3) implies that $U$ has a local minimum. Hence if (A3) holds then (A1) can always be satisfied by centering the coordinate system appropriately and subtracting a constant from $U$.
	Conditions (A1) and (A2) imply  
	\begin{eqnarray} \label{A1.5}
		\norm{\nabla U(x)} &=& \norm{\nabla U(x) - \nabla U(0)} \ \leq \ L\norm{x}
		\quad \mbox{for any } x\in\R^d.
	\end{eqnarray} 
Alternatively, it is possible to replace (A3) by a Lyapunov type drift condition. 
%but this requires a slightly different approach that will be considered in a forthcoming work on randomized HMC. For some of the results stated below only a part of the assumptions is required.
\begin{assumption}\label{A1245}
$U$ is a function in $C^4(\mathbb R^d)$ satisfying (A1) and (A2). Moreover, there exists a
 measurable function $\Psi :\mathbb R^d\to [0,\infty)$ and constants $\lambda ,\alpha \in (0,\infty )$ and $R_2\in (0,\infty ]$ such that 
 	\begin{itemize}
 	 \item[(A4)] \ 
 	The level set $\{ x\in \mathbb R^d: \Psi (x)\le 4\alpha /\lambda\}$
 is compact.
 \item[(A5)] \ For all $x\in\mathbb R^d$ with $|x|< R_2$, $$(\pi\Psi )(x)\ \le\ (1- \lambda )\Psi (x) \, +\,\alpha .$$
 	\end{itemize}
 	 \end{assumption}
 	 If Assumption \ref{A1245} is satisfied then we set
 	 \begin{equation}\label{DefcalR}
 	 \mathcal R\ :=\ \sup\,\left\{  |x|: x\in\mathbb R^d\text{ such that } \Psi (x)\le 4\alpha /\lambda\right\} .
 	 \end{equation}
 In the results below, this constant will play a similar r\^ole as the constant $\mathcal R$ in (A3).  
 For exact and unadjusted HMC, we usually choose $R_2=\infty$ in Assumption \ref{A1245}. For adjusted numerical HMC, Lyapunov functions satisfying (A5) globally may fail to exist since on larger balls smaller step sizes may be required to ensure stability. 
In this case, we can still prove a contraction result on a large ball of radius $R_2<\infty $ if a corresponding Lyapunov condition is satisfied.

Although the drift condition (A5) is not stated as 
explicitly as condition (A3), it can be verified in an explicit way for several important classes of examples.

\begin{example}[Quadratic Lyapunov functions]\label{example:LYAP1}
Suppose that Assumptions (A1) and (A2) are satisfied, and there exist constants $\kappa ,C\in (0,\infty )$ such that 
\begin{equation}
\label{DC}
x\cdot \nabla U(x)\ \ge \ \kappa |x|^2-C\qquad\text{for all $x\in\mathbb R^d$}.
\end{equation}
Then there exists $n_0\in\mathbb N$ such that for exact HMC and for unadjusted numerical HMC with
step size $h=T/n$, the Lyapunov condition
\begin{equation}
\label{LyapPsi} (\pi\Psi )(x)\ \le\ (1- \kappa T^2/8 )\Psi (x) \, +\, (C+2d) T^2 \qquad\text{with }\,\Psi (x):=|x|^2
\end{equation}
is satisfied for all $x\in \mathbb R^d$ provided  $n\ge n_0$ and $L(T^2+hT)\le \kappa /(10\, L)$. For adjusted numerical HMC, the same assertion holds for $|x|\le R_2$ where $R_2$ is an
arbitrary finite constant, provided $h^2R_2^3$ is sufficiently small. 
\end{example}
\begin{example}[Exponential Lyapunov functions]\label{example:LYAP2}
Suppose that Assumptions (A1) and (A2) are satisfied and there exist constants $\kappa ,C, Q\in (0,\infty )$ such that
\begin{equation}
\label{DC1}
x\cdot \nabla U(x)\ \ge \ \kappa |x|-C\quad\text{and}\quad |\nabla U(x)|\ \le\ Q\qquad\text{for all $x\in\mathbb R^d$.}
\end{equation}
Then there exists $\delta >0 $ and a smooth function $\Psi:\mathbb R^d\to (0,\infty )$ with
$\Psi (x)=\exp (\delta |x|)$ for $|x|\ge 1/\delta$ such that for exact or unadjusted HMC 
\begin{equation}
\label{LyapPsiexp} (\pi\Psi )(x)\ \le\ \delta \kappa T^2 \, (5- \Psi (x)/7)
\end{equation}
is satisfied for all $x\in \mathbb R^d$ provided $L(T^2+hT)\le 1$. Explicitly, one can choose
$\delta =\kappa /(4C+8d+Q^2T^2)$. Again, a corresponding assertion holds
for adjusted HMC provided $|x|\le R_2$ and $h$ is chosen sufficiently small  depending on $R_2$.
\end{example}
A sketch of the proofs of \eqref{LyapPsi} and \eqref{LyapPsiexp} is included in Appendix~\ref{app:explicit_lyap}.

\subsection{Coupling}\label{sec:main_results:couplings}
We now introduce a coupling for the transition steps of two copies of the HMC chain starting at different initial conditions $x$ and $y$. The coupling is defined in a different way depending on whether $x$ and $y$ are far apart or sufficiently close. 

\subsubsection{Synchronous coupling for $|x-y|\ge 2\mathcal R$}\label{sec:main_results:couplings:synch_coupling} 
The easiest way to couple the transition probabilities $\pi (x,\cdot)$
and $\pi (y,\cdot)$ for two states $x,y\in\mathbb R^d$ is to use the same random
variables $\xi$ and $\mathcal U$ in both cases for the momentum refreshment and
to decide whether a proposed move is accepted. The corresponding coupling
transition is given by $(x,y)\mapsto (X'(x,y),Y'(x,y))$ where
\begin{equation}\label{**S}
\begin{aligned}
	X'(x,y) \ &= \ q_T(x,\xi) \, I_{A(x)} \ + \ x  \, I_{A(x)^C}, \\ 
	Y'(x,y) \ &= \ q_T(y,\xi) \,  I_{A(y)} \ + \ y \, I_{A(y)^C},
	\end{aligned}
\end{equation} 
with $A(x)$, $A(y)$ defined as in \eqref{*A} above with the same $\xi$ and $\mathcal U$ in both cases. We will apply synchronous
coupling for $\norm{x-y}\geq 2\mathcal R$. Here we can exploit the strong convexity condition (A3) to ensure contractivity for the coupling transition.

\subsubsection{A contractive coupling for $|x-y|< 2\mathcal R$}
For $\norm{x-y}< 2\mathcal R$ we use a different coupling that enables us to
derive a weak form of contractivity even in the absence of convexity.
Let $\gamma>0$ be a positive constant. The precise value of the parameter 
$\gamma$ will be chosen in an appropriate way below.
The coupling transition step is now given by 
\begin{equation}\label{**}
\begin{aligned}
		X'(x,y) \ &= \ q_T(x,\xi) \, I_{A(x)} \ + \ x \, I_{A(x)^C}, \\
		Y'(x,y) \ &= \ q_T(y,\eta) \,  I_{\hat{A}(y)} \ + \ y \,  I_{\hat{A}(y)^C},
\end{aligned}
\end{equation}
with the event $A(x)$ defined as in \eqref{*A} above, and
\begin{equation}	\label{hatA}	
		\hat{A}(y) \,=\, \{\, \mathcal U \ \leq \ \exp(H(y,\eta)-H(q_T(y,\eta),p_T(y,\eta)) 
		\, \}\quad\text{or}\quad \hat{A}(y) \,=\, \Omega 
\end{equation}
in the adjusted and unadjusted case, respectively.
Here the \emph{same} random variable $\mathcal U$ as in \eqref{*A} is used to decide
whether the proposed move to $q_T(y,\eta)$ is accepted. Moreover, we set
\begin{eqnarray}
	\label{eta}
	\eta &:=& \begin{cases}
		\xi  \ + \ \gamma  z  & \text{if } \ \tilde{\mathcal U} \ \leq \
		\dfrac{\varphi_{0,1}(e\cdot \xi + \gamma \norm{z})}{\varphi_{0,1}(e\cdot \xi)},
		\\
		\xi \ -\  2 (e\cdot \xi) e  & \text{otherwise},
		\end{cases}
\end{eqnarray}
where $z=x-y$, $e=z/\norm{z}$, $\varphi_{0,1}$ denotes the density of
the standard normal distribution, and $\tilde{\mathcal U}\sim\text{Unif}(0,1)$ is
independent of $T$, $\xi$ and $\mathcal U$.\medskip

%%%%%%%%%%%%%%%%%%%%%%%%%%%%%%%%%%%%%%%%%%%%%%%%%%%%%%%%%%%%%%%%%%%%%%%%
This coupling is partially motivated by a coupling for second order Langevin diffusions introduced 
in \cite{EbGuZi2016}. It is defined
in such a way that $\xi-\eta=-\gamma z$ holds with
the maximal possible probability, and a reflection coupling is applied
otherwise. As illustrated in Figure~\ref{fig:coupling}, the reason for this choice is that the difference process
$q_t(x,\xi)-q_t(y,\eta)$ is contracting in a time interval $[0,t_0]$
if the difference $\xi-\eta$ of the initial velocities is negatively
proportional to the difference of the initial positions.

In order to verify that $(X'(x,y),Y'(x,y))$ is indeed a coupling of the transition probabilities $\pi (x,\cdot )$ and $\pi (y,\cdot ) $, we remark that the distribution of $\eta$ is 
$N(0,I_d)$ since, by definition of $\eta$ in \eqref{eta} and a change of variables, \begin{eqnarray*}
P[ \eta \in B ]
&  =& E \left[ I_B( \xi+\gamma z)\,  \frac{\varphi_{0,1}(e\cdot \xi + \gamma \norm{z})}{\varphi_{0,1}(e\cdot \xi)} \wedge 1 \right] \\
&&\quad + E \left[ I_B( \xi - 2 (e \cdot \xi) e ) \left( 1 - \frac{\varphi_{0,1}(e\cdot \xi + \gamma \norm{z})}{\varphi_{0,1}(e\cdot \xi)}  \right)^+ \right] \\
&=& \int I_B( x+\gamma z)\, {\varphi_{0,I_d}( x + \gamma {z})}\wedge {\varphi_{0,I_d}( x)} \,dx\\
&&\quad +\int I_B( x - 2 (e \cdot x) e ) \left( \varphi_{0,I_d}( x)- {\varphi_{0,I_d}( x + \gamma {z})}{}  \right)^+\, dx\\
&=&\int I_B(x)\,\varphi_{0,I_d}( x)\, dx\   =\ P[\xi \in B] 
\end{eqnarray*}
for any measurable set $B$. Here $a \wedge b$ denotes the minimum of real numbers $a$ and $b$, and we have used that $ \varphi_{0,I_d}( y- 2 (e \cdot y) e)= \varphi_{0,I_d}( y)= \varphi_{0,I_d}( -y)$.   As a byproduct of this calculation, note also that \[
P[ \eta \ne \xi + \gamma z ] = \int \left( \varphi_{0,I_d}( x)- {\varphi_{0,I_d}( x + \gamma {z})}{}  \right)^+ dx = d_{\TV}(\mathcal{N}(0, I_d) , \mathcal{N}(\gamma z, I_d) )
\] where $d_{\TV}$ is the total variation distance.  Hence, by the coupling characterization of the total variation distance, $\xi-\eta=-\gamma z$ does indeed hold with maximal possible probability.
%Therefore, $(X'(x,y),Y'(x,y))$ is indeed a coupling of the transition probabilities $\pi (x,\cdot )$ and $\pi (y,\cdot ) $. 

%%%%%%%%%%%%%%%%%%%%%%%%%%%%%%%%%%%%%%%%%%%%%%%%%%%%%%%%%%%%%%%%%%%%%%%%
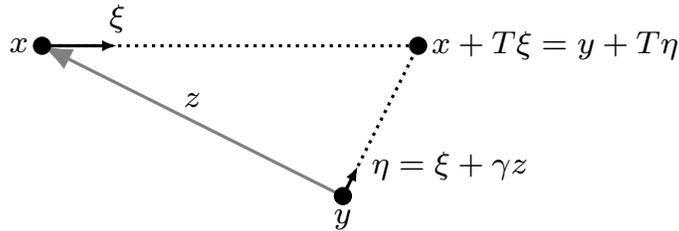
\begin{figure}[ht]
\centering
\begin{tikzpicture}[scale=1.0]
\begin{scope}[very thick] %Rdecoration={markings,mark=at position 1 with {\arrow[scale=4,thick,Latex]{>}}}] 
%\node[black, scale=1.5,fill=white] at (8,5.5) {$\Gamma_i(\x)$};
%\draw[->, thick](0.0,0.0) -- (0.0,5.0);
%\draw[->, thick](0.0,0.0) -- (10.0,0.0);
       \tikzmath{\gamm=0.2; \lamb=1/\gamm; 
       		     \x1 = 1; \x2 =2; \y1=5; \y2=0.0;  \u1=\x1+1; \u2=\x2+0;
		     \z1=\x1-\y1; \z2=\x2-\y2;
                      \v1=\y1+(\u1-\x1)+\gamm*(\x1-\y1); 
                      \v2=\y2+(\u2-\x2)+\gamm*(\x2-\y2); 
                      \px1=\x1+\lamb*(\u1-\x1);   \px2=\x2+\lamb*(\u2-\x2); 
                      \py1=\y1+\lamb*(\v1-\y1); \py2=\y2+\lamb*(\v2-\y2); 
                      } 
                     
                      \draw[->,gray,-{Latex[length=4mm]}](\y1,\y2)--(\x1,\x2) node [midway, above, black, scale=1.5]  {$z$};
                      
\filldraw[color=black,fill=black] (\x1,\x2) circle (0.1) node [left,black, scale=1.5]  {$x$};
\filldraw[color=black,fill=black] (\y1,\y2) circle (0.1) node [below,black, scale=1.5]  {$y$};
%\draw[,-{Latex[length=4mm]}](0,0.)-- (\x1,\x2) node [midway,left,black, scale=1.5]  {$x$};
%\draw[,-{Latex[length=4mm]}](0,0.)-- (\y1,\y2) node [midway,below, scale=1.5]  {$y$};
\draw[->,-{Latex[length=2mm]}](\x1,\x2)--(\u1,\u2) node [above,black, scale=1.5] {$\xi$};
\draw[->,-{Latex[length=2mm]}](\y1,\y2)--(\v1,\v2) node [right,black, scale=1.5]  {$\eta = \xi+\gamma z$};
\draw[-,dotted](\x1,\x2)--(\px1,\px2) node [right,black, scale=1.5]  {$x + T \xi = y + T \eta$};
\draw[-,dotted](\y1,\y2)--(\py1,\py2) ;
\filldraw[color=black,fill=black] (\py1,\py2) circle (0.1);

%\node[black, scale=1.5,fill=white] at (1.0*\u1,1.15*\u2) {$\xi$};
 \end{scope}
\end{tikzpicture}
\caption{\small A diagram showing the basic idea behind the coupling in the case $\gamma=T^{-1}$. 
%with initial positions $x$ and $y$, their difference $z$, initial velocity
%of the first component $\xi$, and initial velocity of the second component $\eta = \xi + \gamma z$ with   
The dotted lines connect the initial positions $x$ and $y$ with the final position
 $q_T(x,\xi) = q_T(y,  \eta)$ for $U = 0$.   When $U \neq 0$, $q_t(x,\xi) - q_t(y,  \eta)$ is still contracting for small $t$.}
 %On the other hand, note that if $\gamma=0$ and $U \equiv 0$, then  $|x - y| = |q_T(x,\xi)-q_T(y,\xi)|$. 
 %}
  \label{fig:coupling}
\end{figure}
%%%%%%%%%%%%%%%%%%%%%%%%%%%%%%%%%%%%%%%%%%%%%%%%%%%%%%%%%%%%%%%%%%%%%%%%

\subsection{Numerical illustration of couplings}

Before stating our theoretical results, we test the coupling defined by \eqref{**} numerically on the following distributions:
\begin{itemize}
\item A normal mixture distribution where the mixture components are twenty two-dimensional Gaussian distributions with covariance matrix given by the $2 \times 2$ identity matrix and with mean vectors given by $20$ independent samples from the uniform distribution over the rectangle $[0, 10] \times [0,10]$.  The energy barriers are not large.  This example is adapted from \cite{LiWo2001,KoZhWo2006}. 
\item A Laplace mixture distribution where the mixture components are twenty two-dimensional (regularized) Laplace distributions using the same covariance matrix and
mean vectors as in the preceding example.  However, unlike the preceding example, in this example the underlying potential is not strongly convex outside a ball, but instead satisfies Assumption~\ref{A1245} 
with respect to an exponential Foster-Lyapunov function as in Example~\ref{example:LYAP2}.
\item A banana-shaped distribution whose associated potential energy $U: \mathbb{R}^2 \to \mathbb{R}$ is given by the Rosenbrock function 
$U(x,y) = (1-x)^2 + 10 (y-x^2)^2$. This function is highly non-convex and unimodal with a global minimum at the point $(1,1)$ where $U(1,1)=0$.  This minimum lies in a long, narrow, banana shaped valley.  
\end{itemize} 
For simplicity, we apply the coupling globally and choose the step size $h$ to integrate
the Hamiltonian dynamics small enough to ensure that essentially all proposed moves are accepted.  
Realizations of the coupling process with $T=1$ and $\gamma=1$ are shown in Figure~\ref{fig:coupling_sample_paths}.   
We chose these parameters only for visualization purposes.  
The different components of the coupling are shown as different color dots.   
The insets of the figures show the distance between the components of the coupling as a function of the number of steps.  
%The simulation is terminated when this distance is within $10^{-12}$.   

Figure~\ref{fig:mean_coupling_times} shows the average time after which 
 the distance between the components of the coupling is for the first time within $10^{-9}$.  To produce this figure, we generated $10^5$ samples of the coupled process
for one hundred different values of the duration parameter $T$.  We chose the coupling parameter $\gamma$ equal to either $T^{-1}$,
or equal to zero which corresponds to a synchronous coupling.  The former choice is motivated by Figure~\ref{fig:coupling}.

%%%%%%%%%%%%%%%%%%%%%%%%%%%%%%%%%%%%%%%%%%%%%%%%%%%%%%%%%%%%%%%%%%%%%%%%
\begin{figure}[t]
\begin{center}
\includegraphics[width=0.3\textwidth]{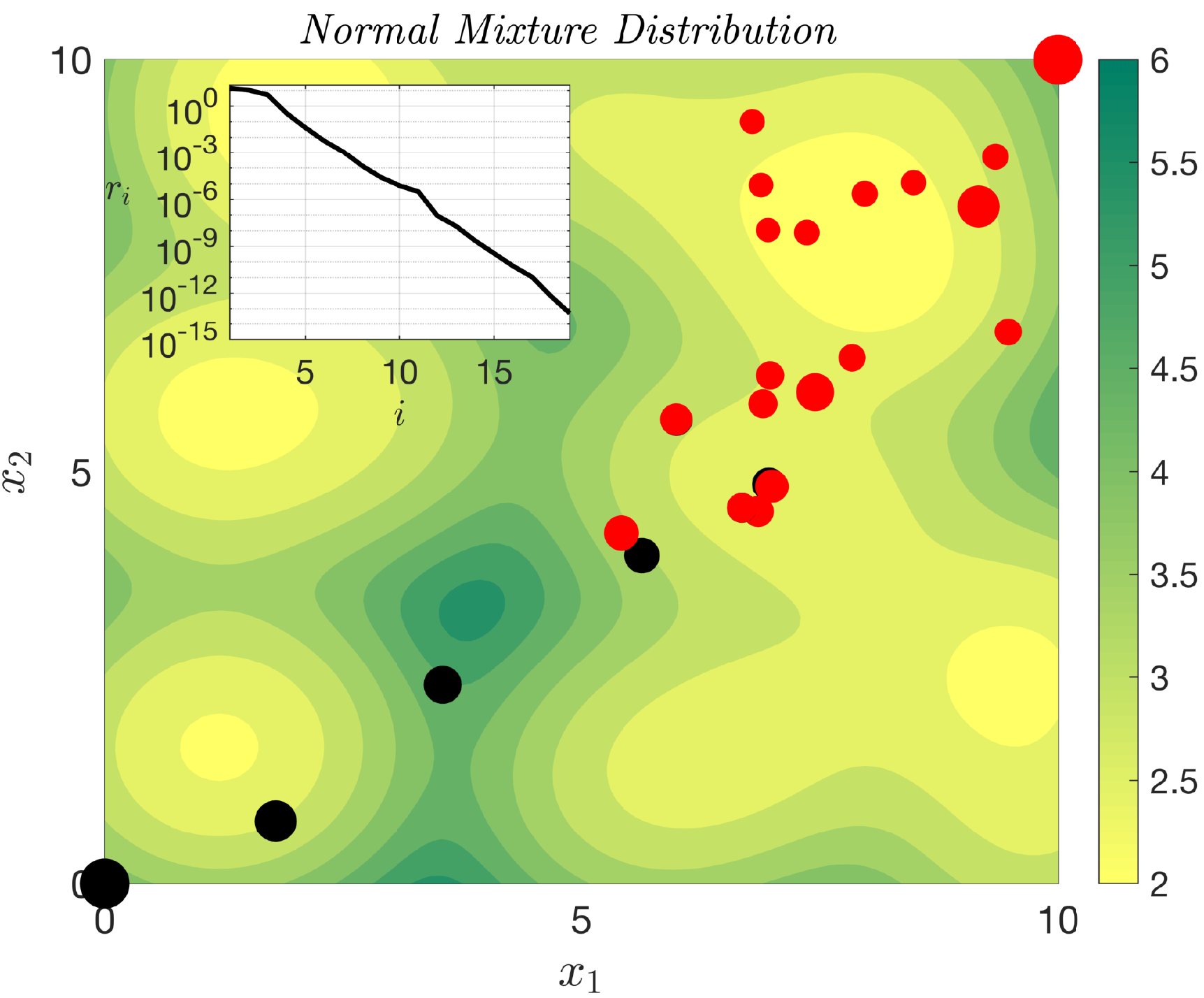} 
\hspace{0.125in}
\includegraphics[width=0.3\textwidth]{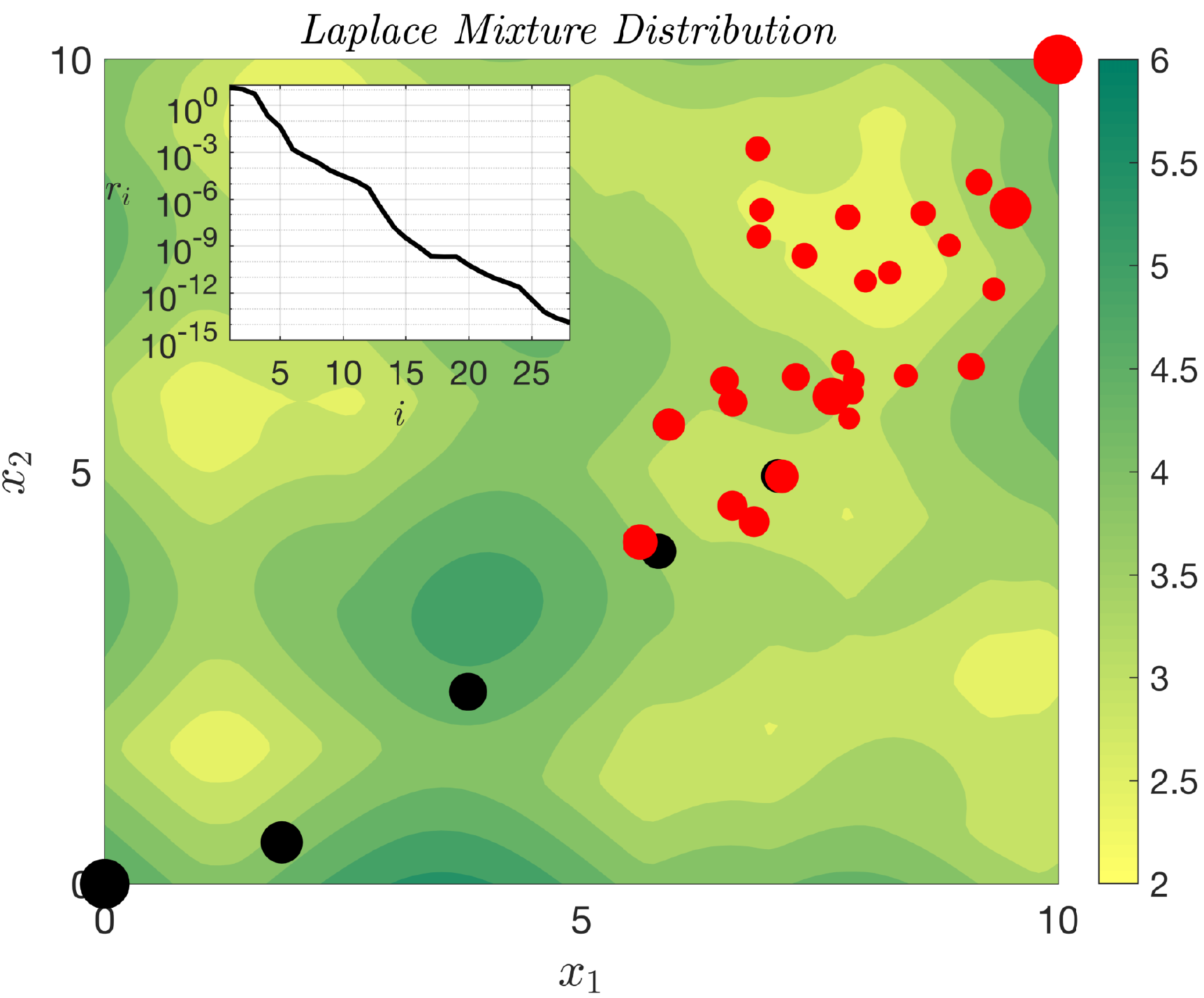} 
\hspace{0.125in}
\includegraphics[width=0.3\textwidth]{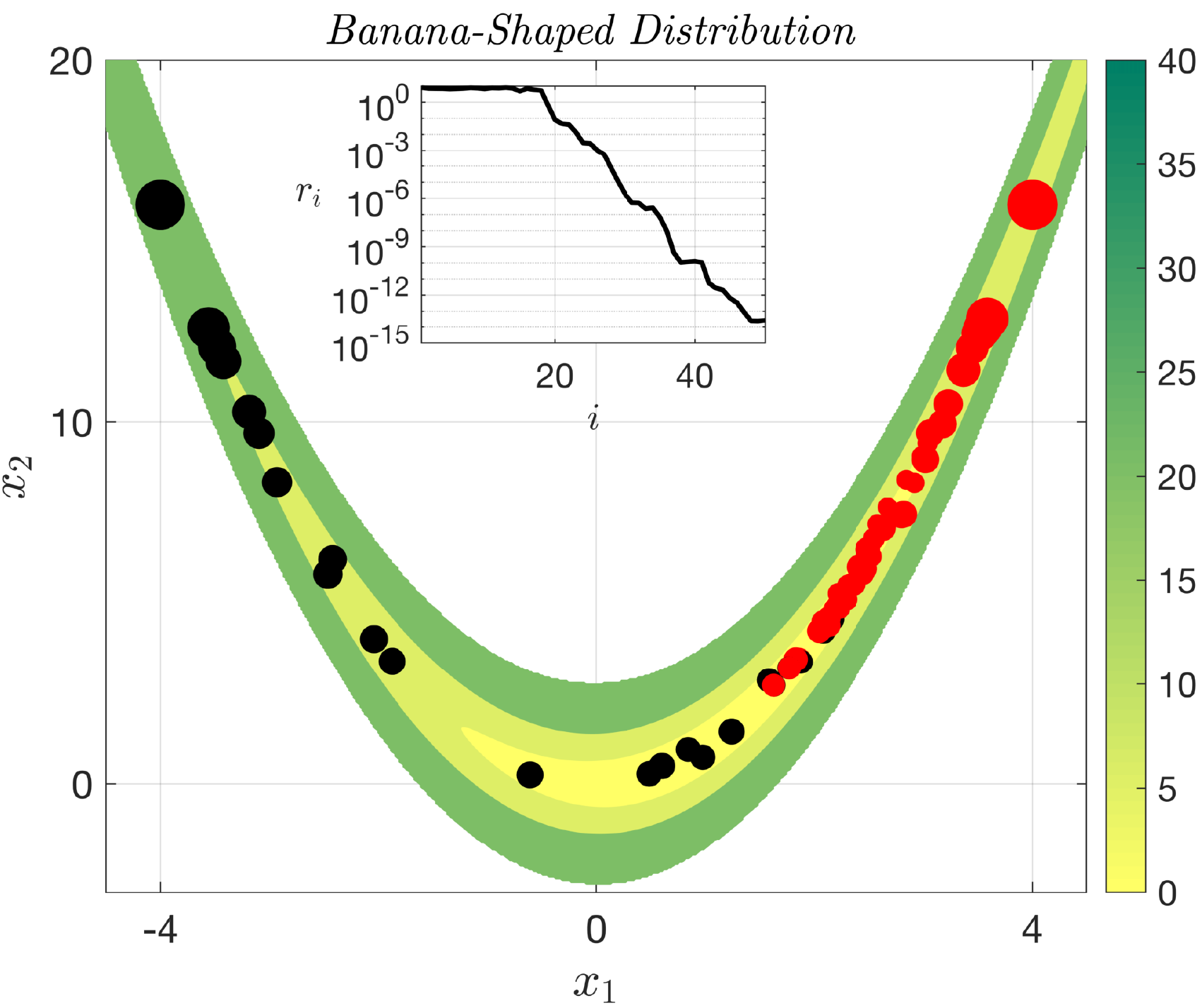}
\hbox{
            (a) 
        \hspace{1.425in} (b) 
        \hspace{1.425in} (c) 
        } 
\end{center}
\caption{ \small
This figure illustrates realizations of the coupling with $T=1$ and $\gamma=1$.   The different components of the coupling are shown as different color dots.  
The size of the dots is related to the number of steps: points along the trajectory corresponding to a larger number of steps have
smaller markers.  A contour plot of the underlying potential energy function is shown in the background.
The inset plots the distance $r_i$ between the components of the coupling as a function of the step index $i$. 
The simulation is terminated when this distance first reaches $10^{-14}$.  
 In (a), (b), and (c), this occurs in 20, 29, and 51 steps, respectively.
 }
 \label{fig:coupling_sample_paths}
\end{figure}

\begin{figure}[t]
\begin{center}
\includegraphics[width=0.3\textwidth]{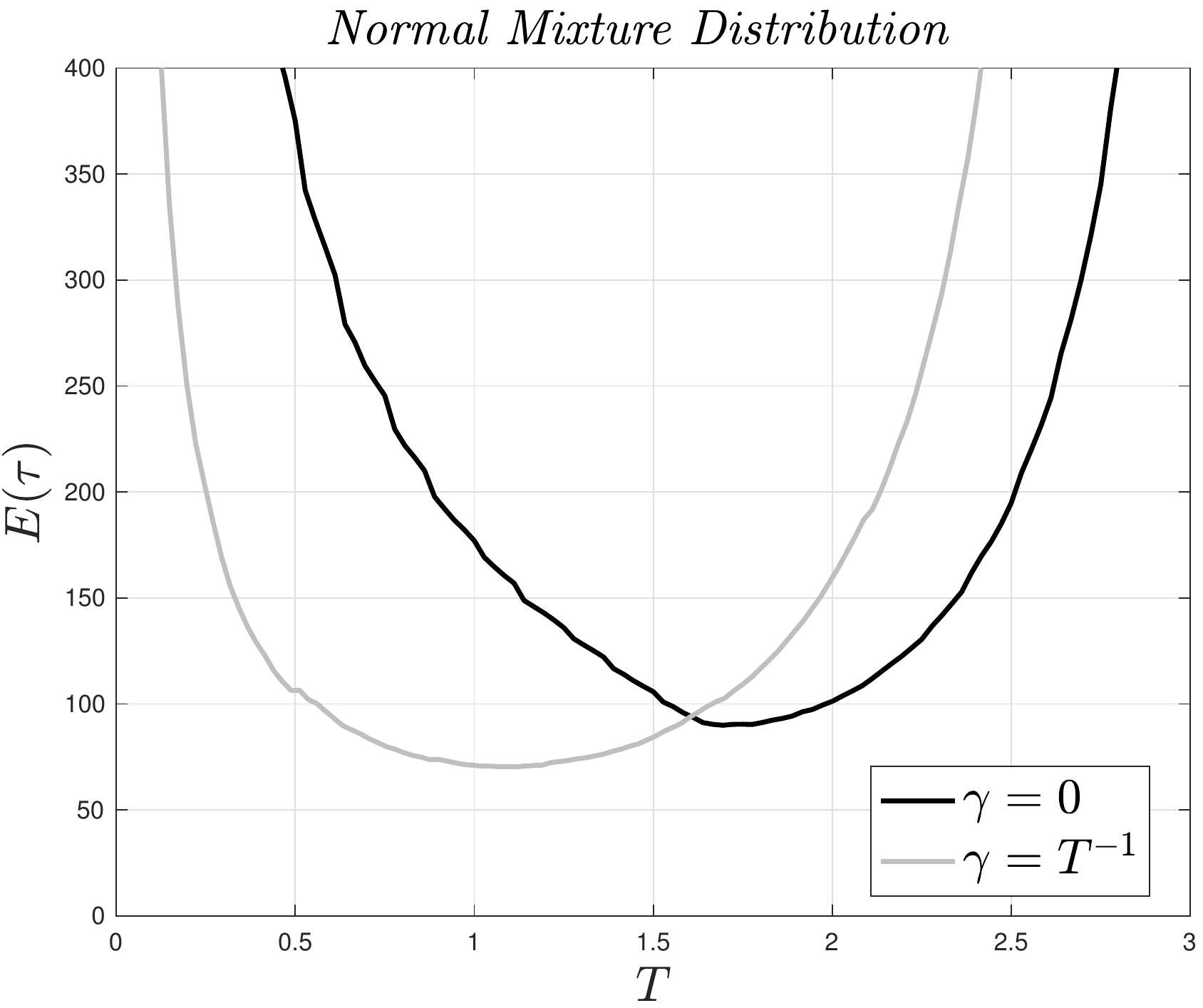} 
\hspace{0.125in}
\includegraphics[width=0.3\textwidth]{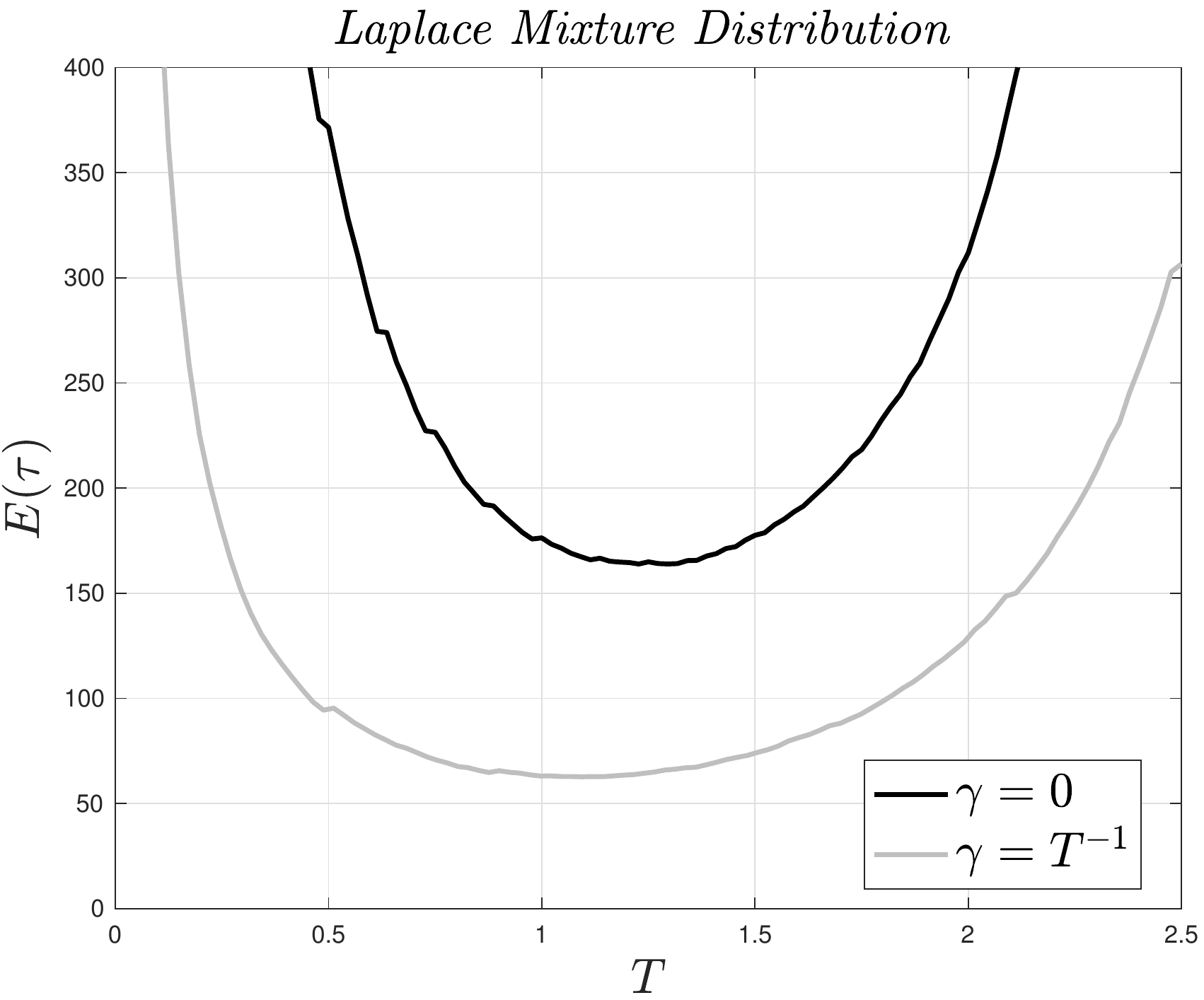} 
\hspace{0.125in}
\includegraphics[width=0.3\textwidth]{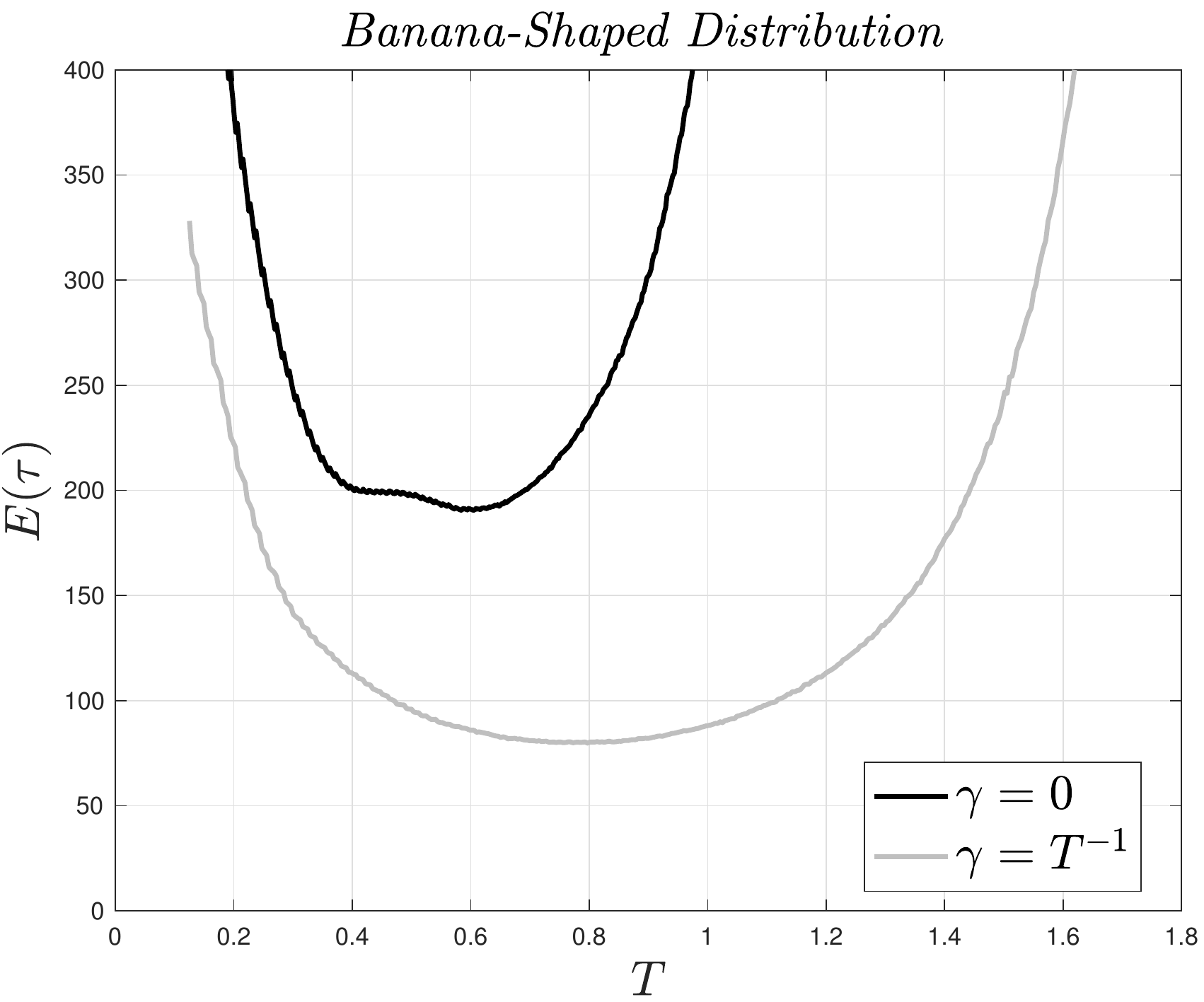}
\hbox{
            (a) 
        \hspace{1.425in} (b) 
        \hspace{1.425in} (c) 
        } 
\end{center}
\caption{ \small This figure illustrates the average of the random time $\tau$ after which 
 the distance between the components of the coupling is for the first time within $10^{-9}$. The estimated average is plotted
as a function of the duration $T$ of the Hamiltonian dynamics
for $\gamma=0$ (black) and $\gamma=T^{-1}$ (gray).  The latter choice is motivated by Figure~\ref{fig:coupling}.   From (a), note that the minimum of the function
is smaller and occurs at a smaller value of $T$ when $\gamma=T^{-1}$.  This difference is more pronounced in (b) and (c), because the underlying potential is
not strongly convex outside a ball in (b) and is highly nonconvex in (c).
The kinks in the graphs are due to an artificial periodicity phenomena with respect to the time step size, 
and can be reduced by reducing the time step size.}
 \label{fig:mean_coupling_times}
\end{figure}
%%%%%%%%%%%%%%%%%%%%%%%%%%%%%%%%%%%%%%%%%%%%%%%%%%%%%%%%%%%%%%%%%%%%%%%%

\subsection{Contractivity}

We now state our main contraction bounds for the coupling introduced above. For given $x,y\in\mathbb R^d$ let
$$r(x,y)\ =\ |x-y|,\qquad R'(x,y)\ =\ |X'(x,y)-Y'(x,y)|,$$
denote the coupling distance before and after the transition step.
For exact HMC we set $h:=0$, whereas for numerical HMC, $h>0$ is the discretization step size. 

\subsubsection{Contractivitiy by strong convexity} 

The assumed strong convexity of $U$ outside of a euclidean ball directly implies contractivity of a transition step for exact HMC for initial values $x$ and $y$ that are sufficiently far apart. 

\begin{theorem}[Contractivity for exact HMC, strongly convex case]\label{SCA}
Suppose that Assumption \ref{A123} is satisfied, and let $h=0$.
Then for any $x,y\in\R^d$ and $T\in\R_+$ such that $\norm{x-y}\geq 2 \mathcal R$
and $LT^2\leq K/L$, 
\begin{eqnarray}\label{SC1}
	R'(x,y) &\leq& \left( 1- \frac{1}{2} K T^2\right) r(x,y).
\end{eqnarray}
\end{theorem}
The proof is a direct consequence of Lemma \ref{C} below, see Section
\ref{sec:proofsmain}. A similar result is proven in \cite{mangoubi2017rapid}. \medskip

Notice that contractivity is only guaranteed for $LT^2$ smaller than the 
conditioning number $K/L$. Sometimes, contraction bounds for longer durations can be obtained.
However, as discussed in the introduction, due to possible periodicity of the Hamiltonian flow, in general these 
do hot hold for arbitrary $T$.

\begin{example}[Bivariate Normal Target]\label{example:bivariatenormal}
Consider $U(x) = \frac{1}{2} x^T \Sigma^{-1} x$, where $\Sigma = \begin{bmatrix} \sigma_{\max}^2 & 0 \\ 0 & \sigma_{\min}^2 \end{bmatrix}$ with $\sigma_{\max} \ge \sigma_{\min}>0$.
In this case, Assumption \ref{A123} is satisfied with $\mathcal R=0$,
%for all $x, y \in \mathbb{R}^2$, \begin{align*}
%| \nabla U(x) - \nabla U(y) | &\le  \sigma_{\min}^{-2} | x  - y | \;, \quad \text{and} \\
%( x - y) \cdot ( \nabla U(x) - \nabla U(y) ) &\ge \sigma_{\max}^{-2} | x - y |^2 \;.
%\end{align*}  Hence, 
$L = \sigma_{\min}^{-2} $ and $K= \sigma_{\max}^{-2}$.  
%Moreover, since $U$ is a strongly convex function, $\mathcal R = 0$,  and hence, a
%synchronous coupling is applied globally.  
Theorem~\ref{SCA} gives a global contraction for synchronous coupling with rate $K T^2 / 2$ provided that $ T^2 \le \sigma_{\min}^4/  \sigma_{\max}^2 $.  In particular, a necessary condition is that $T$ is no greater than $\sigma_{\min}$, which avoids periodicities in the Hamiltonian dynamics  \cite{Ma1989}.
\end{example}

Next we consider both adjusted and unadjusted numerical HMC. We fix an upper bound $h_1>0$ for the discretization step size $h$. We assume that
\begin{equation}\label{LTT}
LT(T+h_1)\ \le\ K/L.
\end{equation}
Under similar conditions as in Theorem \ref{SCA} we obtain contractivity on average for coupled HMC transition steps:

\begin{theorem}[Contractivity for numerical HMC, strongly convex case] \label{SCB}
Suppose that Assumption \ref{A123} is satisfied, and fix
$T,R_2,h_1\in(0,\infty)$ such that \eqref{LTT} holds. Then there exists $h_0>0$ depending only on $K$, $L$, $M$, $N$, $T$, $R_2$ and $d$ such that for any $h\in(0,\min (h_0,h_1)]$ with
$T/h\in\Z$ and for any $x,y\in\R^d$ with $\norm{x-y}\geq 2\mathcal R$ and
$\max(\norm{x},\norm{y})\leq R_2$,
\begin{eqnarray}\label{SC2}
	E[R'(x,y)] &\leq&\left( 1 -\frac{1}{4} KT^2\right) r(x,y).
\end{eqnarray}
Moreover, for fixed $K$, $L$, $M$ and $N$, $h_0$ can be chosen such that
$h_0^{-1}$ is of order $O(R_2)$ for unadjusted numerical HMC and of order $O((1+T^{-1/2})(R_2^2+d))$ for adjusted numerical HMC.
\end{theorem}

The proof is given in Section
\ref{sec:proofsmain}. Key ingredients are the bound for contractivity of the proposal in Lemma \ref{C} and a bound for the probability that the proposal move gets accepted for one of the components of the coupling and rejected for the other component, cf.\ Theorem \ref{AR}.
For unadjusted HMC, the proof simplifies considerably, and $h_0$ can be chosen independently of the dimension.

\subsubsection{Contractivity without global convexity}
Even if we do not assume convexity, we can still obtain contractivity for $x,y$ at a bounded distance if we replace $r(x,y)=|x-y|$ by a modified metric. To this end we consider a distance function of the form
\begin{equation}\label{rho}
\rho(x,y) \ =\ f(r(x,y)),\qquad x,y\in\mathbb R^d,
\end{equation} 
where $f:[0,\infty )\to [0,\infty )$ is a concave function given by
\begin{equation}\label{f}
f(r) \ =\ \int_0^{r} \exp(-a\min (s, R_1))\ ds
\end{equation}
with parameters $a,R_1\in (0,\infty)$ to be specified below.\medskip 
%We used the notation
%$s\wedge t= \min(s,t)$ for $s,t\in\R$. On the torus we will
%choose $R_0=R_1=\text{diam}(\T^d)$ such that $f(r)=\int_0^{r}
%\exp(-a s)\ ds$. 

We again fix an upper bound $h_1>0$ for the discretization step size $h$. We now replace \eqref{LTT} by the more stringent assumption
\begin{equation}\label{A0}
L(T+h_1)^2\ \le\ \min\,\left(\frac{K}{L},\frac{1}{4},\frac{1}{16\,\Lambda}\right)
\end{equation}
where $\Lambda := 16L\mathcal R^2$. Under this condition 
we obtain contractivity on average w.r.t.\ the metric $\rho$ if the parameters
$\gamma $, $a$ and $R_1$ defining the coupling and the metric are adjusted 
appropriately. Explicitly, we set
\begin{eqnarray}
\label{Cgamma}\gamma &:=& \min \left( T^{-1},\mathcal R^{-1}/4\right),\\
\label{Ca}a &:=&  T^{-1},\\
\label{CR1} R_1 &:=&\frac 52\, (\mathcal R+T).
\end{eqnarray} 

For exact HMC, we can prove a global contraction under Assumption \ref{A123}.

\begin{theorem}[Contractivity for exact HMC, general case] \label{thm:contrmainexact}
Suppose that Assumption \ref{A123} is satisfied, and fix
$T\in(0,\infty)$ such that $LT^2\le\min \left( \frac{K}{L},\frac{1}{4},\frac{1}{16\,\Lambda}\right)$. Let $\gamma$, $a$ and $R_1$ be given by \eqref{Cgamma}, \eqref{Ca} and \eqref{CR1}, respectively. Then for exact HMC, for all $x,y\in\R^d$,
\begin{eqnarray}\label{contrmain}
	E[f(R'(x,y))] &\leq& \left( 1 -c\right) f(r(x,y)),\qquad\qquad\mbox{where}
\end{eqnarray}
\begin{equation}\label{crate}
c\ =\ \frac{1}{10}\min\left( 1,\frac 12 KT^2 ( 1+\frac{\mathcal R}{T}) e^{-\mathcal R/(2T)}\right)\, e^{-2\mathcal R/T}.
\end{equation}
%\eqref{contrmain} holds  with the contraction rate $c$ given by \eqref{crate}.
\end{theorem}

For numerical HMC, a similar contraction property holds on every finite ball in $\mathbb R^d$.
The contraction rate does not depend on the radius $R_2$ of the ball, but on a larger ball, a smaller step size is required to ensure contractivity.

\begin{theorem}[Contractivity for numerical HMC, general case] \label{thm:maincontraction}
Suppose that Assumption \ref{A123} is satisfied, and fix
$T,R_2,h_1\in(0,\infty)$ such that \eqref{A0} holds. Let $\gamma$, $a$ and $R_1$ be given by \eqref{Cgamma}, \eqref{Ca} and \eqref{CR1}, respectively. Then there exists $h_\star >0$ 
depending only on $K$, $L$, $M$, $N$, $\mathcal R$, $T$, $R_2$ and $d$ s.t.\ for any $h\in(0,\min (h_\star ,h_1)]$ with
$T/h\in\Z$ and for any $x,y\in\R^d$ with 
$\max(\norm{x},\norm{y})\leq R_2$, \eqref{contrmain} holds  with the contraction rate $c$ given by \eqref{crate}.
%\begin{eqnarray}\label{contrmain}
%	E[f(R'(x,y))] &\leq& \left( 1 -c\right) f(r(x,y)),\qquad\qquad\mbox{where}
%\end{eqnarray}
%\begin{equation}\label{crate}
%c\ =\ \frac{1}{10}\min\left( 1,\frac 12 KT^2 ( 1+\frac{\mathcal R}{T}) e^{-\mathcal R/(2T)}\right)\, e^{-2\mathcal R/T}.
%\end{equation}
Moreover, for fixed $K$, $L$, $M$ and $N$, $h_\star $ can be chosen such that
$h_{\star}^{-1}$ is of order $O(R_2)$ for unadjusted, and of order $O((1+T^{-1/2}+\mathcal R^{1/2})(R_2^2+d))$ for adjusted numerical HMC.
\end{theorem}

The proofs of Theorems \ref{thm:contrmainexact} and \ref{thm:maincontraction} are given in Section
\ref{sec:proofsmain}.
%Key ingredients are the bound for contractivity of the proposal in Lemma \ref{C} and a bound for the probability that the proposal move gets accepted for one of the components of the coupling and rejected for the other component, cf.\ Theorem \ref{AR}. 

%\begin{remark}[Dimension dependence]
%The lower bound $c$ for the contraction rate in Theorems \ref{thm:contrmainexact}  and \ref{thm:maincontraction} does not depend explicitly on the dimension. However, in applications, the parameter $\mathcal R$ may depend on the dimension which would cause an implicit (and possibly exponential) dimension dependence of $c$. This can not be avoided in the general setup considered here. In concrete models (e.g.\ perturbations of product measures), modifications of the approach might avoid a possible dimension dependence, see e.g.\
%\cite{EberleMH, Eb2016A, Zimmer} for related results.
%For \textcolor{red}{adjusted} numerical HMC, the dimension also affects the upper bound $h_\star$ for the discretization step size which is relevant for the computational complexity. Under restrictive assumptions on the potential $U$ (strong convexity and product structure), more precise results on the dimension dependence are proven in  \cite{mangoubi2017rapid}.
%\end{remark}

The following simple example demonstrates that our results provide applicable bounds in multimodal situations if $T$ is adjusted appropriately:

\begin{example}[Gaussian Mixture]
Consider a mixture of two Gaussians with means $\pm 2 \sigma$ and variances $\sigma^2$ where $\sigma >0$.   The corresponding potential  is \[
U(x) = - \log\left( \exp\left(-\dfrac{(x-2 \sigma)^2}{2 \sigma^2} \right) + \exp\left(-\dfrac{(x+2 \sigma)^2}{2 \sigma^2} \right) \right).
\] 
In this case, $\sigma^2U''(x) = {1 - 4 \operatorname{sech}\left( {2 x}/{\sigma} \right)^2 } $, $\, L = \sup |U''| = 3 / \sigma^2$, and $U''(x) \ge 2/(3 \sigma^2)$ for all $|x| > \sigma$, which allows us to choose $K=2/(3 \sigma^2)$ and $\mathcal R = 2 \sigma$. Hence,
the condition on $L T^2$ in Theorem~\ref{thm:contrmainexact} reduces to
$L T^2 \le {1}/{3072}$, and the rate in \eqref{crate} reduces to \[
c\ =\ \frac{1}{10}\min\left( 1,\frac 13 \frac{T^2}{\sigma^2}  ( 1+4 \frac{\sigma}{T}) e^{-2 \sigma/T}\right)\, e^{-8 \sigma/T} \;.
\]  If we choose $T$ proportional to $\sigma$, then this rate is constant.
\end{example}

\begin{remark}[Scope of coupling approach]
There are various possible extensions of the basic coupling approach presented here. These
include componentwise coupling on product spaces, two-scale couplings on high dimensional spaces, and geometry-adapted couplings. See \cite{Eb2016A, Zimmer, Wang}
for corresponding approaches for reflection coupling of diffusion processes -- we expect that similar extensions exist for our couplings of HMC. 
For example, it might be sufficient to apply asynchronous coupling to a few components and couple the other components synchronously \cite{Zimmer}, and for models with a nontrivial
geometry such as the banana shaped distribution in Figure \ref{fig:coupling_sample_paths},
the coupling should be designed with straight lines replaced by geodesics in an
appropriate geometry \cite{cranston}.

In Assumption \ref{A123}, at least the global Lipschitz condition on $\nabla U (x)$ is crucial for our proofs. Nevertheless, convergence 
may still hold under a relaxed condition if the step size is adjusted appropriately, possibly
depending on $x$. The asymptotic strong convexity assumption can be replaced by a Lyapunov condition, see below. Finally,
Example \ref{example:bivariatenormal} shows that condition \eqref{A0} on $T$ can not be avoided in general. For appropriate classes of models, it could possibly be improved by considering an adequate non-Euclidean geometry. More generally, we point out that even if
some of the assumptions are not satisfied, the coupling is still well-defined, and one can check empirically
if two copies with distant starting points approach each other. Of course, this does not 
guarantee convergence but it may serve as a diagnostic tool for HMC.
\end{remark}

\subsubsection{Contractivity under Foster-Lyapunov condition}
\label{sec:ConLyapunov}
It is possible to replace the asymptotic strong convexity condition in Assumption 
\ref{A123} by the Lyapunov drift condition in Assumption \ref{A1245}. In this case we consider
a semimetric of the form 
\begin{equation}\label{rhoeps}
\rho_\epsilon (x,y)\ = \ \sqrt{f\left(\min (|x-y|,2\mathcal R)\right)\cdot \left(1+\epsilon\Psi (x)+\epsilon\Psi (y)\right)}
\end{equation}
where $\epsilon >0$ is a positive constant, $\mathcal R$ is defined by \eqref{DefcalR}, and $f$ is again given by \eqref{f}. The idea of using corresponding semimetrics for deriving contraction properties under Lyapunov conditions goes back to \cite{HairerMattinglyScheutzow} and \cite{Butkovsky}, see also \cite{EGZHarris}.\smallskip

Similarly as above, we fix an upper bound $h_1>0$ for the discretization step size $h$ such that
\begin{equation}\label{A0L}
L(T+h_1)^2\ \le\ \min\,\left(\frac{1}{4},\frac{1}{16\,\Lambda}\right)
\end{equation}
where $\Lambda := 16L\mathcal R^2$. We choose the parameters $\gamma$, $a$ and $R_1$ as in \eqref{Cgamma}, \eqref{Ca} and
\eqref{CR1} respectively, and we set
\begin{equation}\label{Cepsilon}
\epsilon \ :=\ \frac{1}{40\, A}\, e^{-2\mathcal R/T}.
\end{equation}
Then under a global Lyapunov condition, we can prove global contractivity on average w.r.t.\ the semimetric
$\rho_\epsilon$ for exact and for unadjusted numerical HMC .
\begin{theorem}[Contractivity for exact and unadjusted HMC under Lyapunov condition] \label{thm:contrlyapexact}
Consider exact HMC or unadjusted numerical HMC with step size $h\in (0,h_1]$ such that $T/h\in\Z$.  Suppose that Assumption \ref{A1245} is satisfied with $R_2=\infty$, and \eqref{A0L} holds.
Let $\gamma$, $a$, $R_1$ and $\epsilon$ be given by \eqref{Cgamma}, \eqref{Ca}, \eqref{CR1} and \eqref{Cepsilon}, respectively. Then for all $x,y\in\R^d$,
\begin{eqnarray}\label{contrmain2}
	E[\rho_\epsilon (X'(x,y),Y'(x,y))] &\leq& \left( 1 -c\right) \rho_\epsilon(x,y),\qquad\qquad\mbox{where}
\end{eqnarray}
\begin{equation}\label{crate2}
c\ =\ \frac{1}{80}\min\left( 10\, \lambda\, ,\, e^{-2\mathcal R/T}\right).
\end{equation}
%\eqref{contrmain} holds  with the contraction rate $c$ given by \eqref{crate}.
\end{theorem}
For adjusted numerical HMC, we again obtain contractivity on a given ball for sufficiently small step sizes.

\begin{theorem}[Contractivity for numerical HMC under Lyapunov condition] \label{thm:lyapcontraction}
Consider adjusted numerical HMC with step size $h\in (0,h_1]$ such that $T/h\in\Z$. Suppose that Assumption \ref{A1245} is satisfied for some $R_2\in (0,\infty )$, and \eqref{A0L} holds. Let $\gamma$, $a$, $R_1$ and $\epsilon$ be given by \eqref{Cgamma}, \eqref{Ca}, \eqref{CR1} and \eqref{Cepsilon}, respectively. Then there exists $h_\star >0$ 
depending only on $L$, $M$, $N$, $\mathcal R$, $R_2$ and $d$ s.t.\ for any $h\in(0,\min (h_\star ,h_1)]$ with
$T/h\in\Z$ and for any $x,y\in\R^d$ with 
$\max(\norm{x},\norm{y})\leq R_2$, \eqref{contrmain} holds  with the contraction rate $c$ given by \eqref{crate}.
%\begin{eqnarray}\label{contrmain}
%	E[f(R'(x,y))] &\leq& \left( 1 -c\right) f(r(x,y)),\qquad\qquad\mbox{where}
%\end{eqnarray}
%\begin{equation}\label{crate}
%c\ =\ \frac{1}{10}\min\left( 1,\frac 12 KT^2 ( 1+\frac{\mathcal R}{T}) e^{-\mathcal R/(2T)}\right)\, e^{-2\mathcal R/T}.
%\end{equation}
Moreover, for fixed $L$, $M$ and $N$, $h_\star $ can be chosen such that
$h_{\star}^{-1}$ is of order $O((1+\mathcal R^{1/2})(R_2^2+d))$ for adjusted numerical HMC.
\end{theorem}

The proofs of Theorems \ref{thm:contrlyapexact} and \ref{thm:lyapcontraction} are given in Section
\ref{sec:proofsmain}.
%Key ingredients are the bound for contractivity of the proposal in Lemma \ref{C} and a bound for the probability that the proposal move gets accepted for one of the components of the coupling and rejected for the other component, cf.\ Theorem \ref{AR}. 

%\begin{remark}[Dimension dependence]
%The lower bound $c$ for the contraction rate in Theorems \ref{thm:contrmainexact}  and \ref{thm:maincontraction} does not depend explicitly on the dimension. However, in applications, the parameter $\mathcal R$ may depend on the dimension which would cause an implicit (and possibly exponential) dimension dependence of $c$. This can not be avoided in the general setup considered here. In concrete models (e.g.\ perturbations of product measures), modifications of the approach might avoid a possible dimension dependence, see e.g.\
%\cite{EberleMH, Eb2016A, Zimmer} for related results.
%For \textcolor{red}{adjusted} numerical HMC, the dimension also affects the upper bound $h_\star$ for the discretization step size which is relevant for the computational complexity. Under restrictive assumptions on the potential $U$ (strong convexity and product structure), more precise results on the dimension dependence are proven in  \cite{mangoubi2017rapid}.
%\end{remark}

\subsubsection{Dimension dependence of contraction rates}
\label{sec:dimension}

Dependence on the dimension of the bounds for the contraction rates derived above arises by two different effects:\medskip\\
\emph{a) Dimension dependence due to model properties}\smallskip\\
%\begin{remark}[Dimension dependence]
The rate $c$ in Theorems \ref{thm:contrmainexact}  and \ref{thm:maincontraction} does not depend explicitly on the dimension $d$. However, 
a dimension dependence arises if the underlying parameters in Assumptions
\ref{A123} and \ref{A1245} are dimension dependent. For example,
the parameter $\mathcal R$ in (A3) may depend on $d$ which would then cause an implicit (and possibly exponential) dimension dependence of the bound $c$ for the contraction rate in Theorems \ref{thm:contrmainexact} and \ref{thm:maincontraction}.
Similarly, the constant $\alpha$ in (A5) will usually depend on $d$ and
cause an exponential degeneration of the rates in Theorems \ref{thm:contrlyapexact} and
\ref{thm:lyapcontraction}. 
This possible degeneration in high dimensions
can not be avoided in the general setup considered here,
and it occurs in a very similar form for other Markov processes with invariant measure
$\mu$, see e.g.\ \cite{Eb2016A,EbGuZi2016,EGZHarris}. Indeed, note that 
our assumptions do not exclude models with a phase transition in the limit $d\to\infty$.

For more specific model classes, modifications of our coupling approach can avoid a possible dimension dependence. We mention two such situations that will be analysed in detail in
forthcoming work, see also the also the related results for other classes of Markov processes in
\cite{EberleMH, Eb2016A, Zimmer}.
\begin{itemize}
\item For models with weak interactions (e.g.\ perturbations of product measures), a 
componentwise coupling similar to the one suggested for Langevin diffusions in \cite{Eb2016A} can be applied to obtain
dimension free contraction rates for exact or unadjusted numerical Hamiltonian Monte Carlo.
\item For probability measures $\mu $ that have a sufficiently nice density w.r.t.\ a Gaussian measure on $\mathbb R^d$ or on an infinite dimensional Hilbert space, a two scale 
coupling approach has been applied in \cite{Zimmer} to derive dimension free contraction
rates in the case of overdamped Langevin diffusions with invariant measure $\mu$.
Such measures arise for example in transition path sampling and Bayesian inverse 
problems \cite{HairerStuartVoss, DashtiStuart}. In forthcoming work, we show that in combination with the results 
in this paper, the approach in \cite{Zimmer} can be carried over to HMC.
\end{itemize}
%\end{remark}
\emph{b) Dimension dependence due to numerical discretization}\smallskip\\
For numerical HMC, the dimension also affects the upper bound $h_\star$ for the discretization step size. This is relevant for resulting upper bounds of the computational complexity which are typically of order $1/(h_\star c)$. 

In the results above, this dimension dependence of the maximal step size does not occur for unadjusted HMC. Although in this case, the contraction
properties of the Markov chain are completely dimension free, a dimension dependence 
arises because the step size has to be chosen sufficiently small to ensure that the 
invariant measure of the unadjusted Markov chain is close to $\mu$. The analysis of this dimension dependence requires different techniques that are not the content of this work.
Under rather strong assumptions (including strong convexity), quite sharp bounds for the resulting total dimension dependence for several integrators have been derived by Mangoubi and Smith \cite{mangoubi2017rapid}. In a more general context, sharp bounds for the distance between the invariant measures are derived
in the forthcoming work \cite{durmuseberle2019}.

On the other hand, for adjusted HMC, our bounds require 
$h^{-1}$ of order $d$, thus causing a dimension dependence of at least the same order
for the computational complexity. It is an open question if these bounds can be improved. 
Scaling limits and the results in \cite{mangoubi2017rapid} suggest that at least under 
more restrictive assumptions, a better dimension dependence (possibly up to $d^{1/4}$)
might hold for the computational complexity. On the other hand, we are not aware of 
a result proving a contraction under a bound on $h^{-1}$ of order $d^\alpha$ with
$\alpha <1$.

\subsection{Quantitative bounds for distance to the invariant measure}
\label{sec:QB}

For exact HMC, Theorem \ref{thm:contrmainexact} establishes global contractivity of the transition kernel $\pi (x,dy)$ w.r.t.\ the Kantorovich ($L^1$ Wasserstein) distance
$$\mathcal W_\rho (\nu ,\eta )\ =\  \inf_{\gamma\in C(\nu ,\eta )}\int\rho (x,y)\,\gamma (dx\, dy)$$
on probability measures $\nu ,\eta$ on $\mathbb R^d$. Here the infimum is over all couplings $\gamma$ of $\nu$ and $\eta$. Since the metric $\rho$ is
comparable to the Euclidean Distance on $\mathbb R^d$, contractivity w.r.t.\ 
$\mathcal W_\rho$ immediately implies a quantitative bound on the standard
$L^1$-Wasserstein distance
$$\mathcal W^1 (\nu\pi^n ,\mu )\ =\  \inf_{\gamma\in C(\nu\pi^n ,\mu )}\int |x-y|\,\gamma (dx\, dy)$$
between the law of the HMC chain after $n$ steps and the invariant probability measure $\mu$.

\begin{corollary}\label{cor:QBHMC}
Suppose that Assumption \ref{A123} is satisfied, and let $T\in (0,\infty )$ such that
\begin{equation}
\label{A0A}
LT^2\ \le\ \min \left( \frac{K}{L},\frac{1}{4},\frac{1}{256\,L\mathcal R^2}\right).
\end{equation}
Then for any $n\in\mathbb N$ and for any probability measures $\nu ,\eta$ on $\mathbb R^d$,
\begin{eqnarray}
\label{QBHMC1}
\mathcal W_\rho (\nu\pi^n,\eta\pi^n)& \le & e^{-cn}\mathcal W_\rho (\nu ,\eta ),\qquad\text{and}\\
\label{QBHMC1a}
\mathcal W^1 (\nu\pi^n,\eta\pi^n)& \le & Me^{-cn}\mathcal W^1 (\nu ,\eta ),
\end{eqnarray}
where $c$ is given by \eqref{crate}, and $M=\exp\left( \frac 52 (1+\mathcal R/T)\right)$. In particular, for a given 
constant $\epsilon\in (0,\infty )$, the standard $L^1$ Wasserstein distance $\Delta (n)=\mathcal W^1(\nu\pi^n ,\mu )$ w.r.t.\ $\mu$ after $n$ steps of the chain with initial distribution $\nu$ satisfies $\Delta (n)\le\epsilon$ provided
\begin{equation}
\label{QBHMC2}
n\ \ge\ \frac 1c\, \left( \frac 52+\frac{5\mathcal R}{2T}+\log\frac{\Delta (0)}{\epsilon}\right) . 
\end{equation}
\end{corollary}

The corollary is a rather direct consequence of Theorem \ref{thm:contrmainexact}. The proof is given in Section \ref{sec:LYAP}.

\begin{remark}[Kinetic bounds]
One remarkable feature of the result in Corollary \ref{cor:QBHMC} is that for a 
given initial error $\Delta (0)$, the number of steps required to stay below a certain error bound $\epsilon$ can be chosen universally provided $T$ is chosen proportional to $\mathcal R$. Notice, however, that by Condition 
\eqref{A0A}, it is only possible to choose $T$ proportional to $\mathcal R$ with a fixed proportionality constant if $L\mathcal R^2$ is bounded by a fixed constant~!
\end{remark}

\begin{remark}[Quantitative bounds for ergodic averages]
MCMC methods are often applied to approximate expectation values w.r.t.\ the target distribution by ergodic averages of the Markov chain. Our results (e.g.\ \eqref{QBHMC1}) directly imply completely explicit bounds for
biasses and variances, as well as explicit concentration inequalities for these ergodic averages in the case of HMC. Indeed, the general results by Joulin and Ollivier \cite{JoulinOllivier} show that such bounds follow directly from an $L^1$ Wasserstein contraction w.r.t.\ an arbitrary metric $\rho$, which is precisely the statement shown above.
 \end{remark}

We now return to numerical HMC. Here, our main results in Theorems \ref{thm:maincontraction} and \ref{thm:lyapcontraction} only establish contractivity w.r.t.\ $\mathcal W_\rho$ on a ball of given radius $R_2$. In order to derive bounds for the distance to the invariant measure of the law after $n$ steps, we additionally have to control exit 
probabilities from the ball. This is achieved by another Lyapunov bound that we first
state in a general form. Suppose that $\pi (x,dy)$ is the transition kernel of a 
Markov chain on a complete separable metric space $(S,\rho )$, and let $\mathcal W_\rho$ denote the corresponding Kantorovich distance on 
probability measures on $S$.

\begin{assumption}\label{ALyap}
The following conditions are satisfied for a constant $C\in (0,\infty )$ and measurable functions $\psi ,\varphi :S\to (0,\infty )$:
 	\begin{itemize}
 	\item[(C1)] \emph{Main Lyapunov condition:} There is a constant $\lambda\in [1,\infty )$ such that
 	$$(\pi\psi )(x)\ \le\ \lambda\psi (x)\qquad\text{for any } x\in S\text{ s.t.\ }\psi (x)\le C.$$
 \item[(C2)] \emph{Additional global Lyapunov condition:} There is a constant $\beta\in [1,\infty )$ s.t.
 	$$(\pi\varphi )(x)\ \le\ \beta\varphi (x)\quad\text{and}\quad
 	\rho (x,y)\ \le\ \varphi (x)+\varphi (y)\qquad\text{for any } x,y\in S.$$
 \item[(C3)] \emph{Local contractivity:} There are a measurable map $(X',Y'):S\times S\times\Omega\to S\times S$ defined on a probability space $(\Omega ,\mathcal A, P)$ and a constant $c\in (0,\infty )$ such that for any $x,y\in S$,
 $(X'(x,y,\cdot ),Y'(x,y,\cdot ))$ is a realization of a coupling of $\pi (x,\cdot )$
 and $\pi (y,\cdot )$ satisfying
 $$E\left[ \rho (X'(x,y,\cdot ),Y'(x,y,\cdot ))\right]\ \le\ e^{-c}\rho (x,y)\qquad
 \text{if }\psi (x)\le C\text{ and }\psi (y)\le C.$$
\end{itemize}\end{assumption}

The proof of the following theorem is given in Section \ref{sec:LYAP}.

\begin{theorem}\label{thm:LYAP}
Suppose that Assumption \ref{ALyap} is satisfied. Then for any $n\in\mathbb N$ and for any probability measures $\nu ,\eta$ on $(S,\mathcal B(S))$,
\begin{equation}
\label{25A}
\mathcal W_\rho (\nu\pi^n,\eta\pi^n)\ \le\ e^{-cn}\mathcal W_\rho (\nu ,\eta )\,
+\, \beta^n\lambda^{n-1}(\int\psi\, d\nu+ \int\psi\, d\eta )\, \delta (C),
\end{equation}
where
\begin{equation}
\label{deltaC}
\delta (C)\ :=\ \sup\left\{ \frac{\varphi (x)+\varphi (y)}{\psi (x)+\psi (y)}\, :\, 
x,y\in S\text{ s.t.\ }\psi (x)>C\text{ or }\,\psi (y)>C\right\}\, .
\end{equation}
\end{theorem}

Theorem \ref{thm:LYAP} can be applied to bound the distance to the invariant measure $\mu$ after $n$ steps of adjusted numerical HMC. 
Suppose that $\pi$ is
the corresponding transition kernel for a given discretization step size $h>0$,
and let $\rho$ denote the metric on $S=\mathbb R^d$ defined by
\eqref{rho}, \eqref{f}, \eqref{Ca} and \eqref{CR1}. We will then apply 
Theorem \ref{thm:LYAP} with Lyapunov functions of the form $\varphi (x)=2Td^{1/2}+|x|$ and $\psi (x)=\exp (U(x)^{2/3})$. The right side of \eqref{25A} can then be bounded by a given constant $\epsilon >0$ by first choosing $n$ such that the first term is bounded by $\epsilon /2$, and then choosing $C$ such that the second term is bounded by $\epsilon /2$ as well. As a consequence of
Theorem \ref{thm:maincontraction} and Theorem \ref{thm:LYAP}, we can prove that a similar number of steps as for exact HMC is also sufficient for an 
approximation of the invariant measure by numerical HMC, provided $h$ is
chosen sufficiently small.

\begin{theorem}
\label{thm:QBNHMC}
Suppose that Assumption \ref{A123} is satisfied. Let $T,h_1\in (0,\infty )$ such that \eqref{A0} holds, let $\nu$ be a probability measure on $\mathbb R^d$,
and let $\Delta (n)=\mathcal W^1(\nu\pi^n,\mu)$ denote the standard $L^1$~Wasserstein distance to the invariant probability measure after $n$ steps of adjusted numerical HMC with initial distribution $\nu$. Let $\epsilon \in (0,\infty )$ and $n\in\mathbb N$ such that
\begin{equation}
\label{QBNHMCn}
n\ \ge\ \frac{1}{c}\left(\frac{5}{2}+\frac{5\mathcal R}{2T}+\log^+\left(\frac{2\Delta (0)}{\epsilon}\right)\right)
\end{equation}
where $c$ is given by \eqref{crate}. Then there exists $h_{\star\star}>0$ 
depending only on $K$, $L$, $M$, $N$, $\mathcal R$, $T$, $d$, $\nu$ and $n$, such that for any
$h\in (0,\min (h_{\star\star},h_1))$ with $T/h\in\mathbb Z$,
\begin{equation}
\label{W1eps}
\Delta (n)\ \le\ \epsilon .
\end{equation}
Furthermore, $h_{\star\star}$ can be chosen such that for fixed values of
$K$, $L$, $M$, $N$, $h_{\star\star}^{-1}$ is of order 
$$O\left( \left(1+T^{-\frac 12}+\mathcal R^{\frac 12}\right)\left(d^{\frac 32}n^{\frac 32}+\left(1+{\mathcal R}/{T}\right)^{\frac 32 } \left(d+A(\nu)+\log\epsilon^{-1}\right)^{\frac 32}+\mathcal R^2\right)\right) ,$$ 
where
$A(\nu)=\log\int\exp (U^{2/3})\, d\nu$.
\end{theorem}

We finally remark that for unadjusted HMC, the above result does not apply in the same form. In this case, there is an additional error due to the fact that the invariant measure does not coincide with $\mu$. The careful control of this error requires different techniques that are out of scope of this work, see \cite{mangoubi2017rapid,durmuseberle2019}.

\section{A priori estimates}\label{sec:apriori}
In this section we state several bounds for the Hamiltonian flow, for the coupling, and for acceptance-rejection probabilities that will be crucial in the proof of our main result. The proofs of all the results stated in this section are
included in Section \ref{sec:proofs}.

\subsection{Bounds for the Hamiltonian flow and for velocity
Verlet}\label{sec:main_results:boundsforHMCandVV} 
%The following results are
%stated for $S=\R^d$.
%They imply corresponding results on the torus $\T^d$, since a smooth function
%$U:\T^d\rightarrow \R$ corresponds to a periodic smooth function on
%$\R^d$. In this case, $\norm{x}$ is replaced by $d(x,0)$, where $d$ is the
%intrinsic distance on the flat torus. 
%
%\begin{assumption} \label{A1}
%\begin{enumerate}
%  \item $U\in C^2(\R^d)$ \label{A1.1}
%  \item $\exists L\in \R_+: \ \norm{\nabla^2 U} \ \leq \  L$ \label{A1.2}
%  \item $\exists M\in\R+: \ \norm{\nabla^2 U(x)-\nabla^2 U(y)} \ \leq \ M
%  \norm{x-y} \quad \forall x,y\in\R^d$ \label{A1.3}
%  \item $U$ has a local minimum at $0$\label{A1.4}
%\end{enumerate}
%\end{assumption}
%
%\begin{remark}
%	For some of the results below only a part of the assumptions are required.
%	The assumptions \ref{A1.4} and \ref{A1.2} imply that 
%	\begin{eqnarray} \label{A1.5}
%		\norm{\nabla U(x)} &=& \norm{\nabla U(x) - \nabla U(0)} \ \leq \ L\norm{x}
%		\quad \forall x\in\R^d.
%	\end{eqnarray} 
%	Instead of \ref{A1.4}, one could also directly assume
%	\eqref{A1.5}.
%\end{remark}

In the following, we consider $t\in[0,\infty)$ and $h\in[0,1]$ such that $t/h\in
\mathbb Z$ if $h>0$. We assume throughout that Assumption \ref{A123} is satisfied, and
\begin{eqnarray}
	L(t^2+ht) &\leq& 1. \label{0}
\end{eqnarray}
Recall that $\phi_t=(q_t,p_t)$ denotes the Hamiltonian flow for $h=0$,
and the flow of the velocity Verlet integrator for $h>0$. The proofs of the
following statements are provided in Section
\ref{sec:proofs}. 

\begin{lemma}\label{0A}
For any $x,v\in\R^d$,
\begin{eqnarray}
\label{6}
\ \ \max_{s\leq t} \norm{q_s(x,v)-(x+sv)}&\leq& L (t^2+th) 
\max(\norm{x},\norm{x+tv}),\quad\mbox{and}
\\ \label{7}
\max_{s\leq t} \norm{p_s(x,v)-v} &\leq &
Lt\, \max_{s\le t}|q_s(x,v)|\\
\nonumber &\le &
Lt(1 + L(t^2+th)) \max(\norm{x},\norm{x+tv}).
\end{eqnarray}
In particular,
\begin{eqnarray}\label{8}
\max_{s\leq t}\norm{q_s(x,v)} &\leq& 2 \max(\norm{x},\norm{x+tv}),\qquad\mbox{and}\\
\label{9}
\max_{s\leq t}\norm{p_s(x,v)} &\leq&
\norm{v} + 2 L t  \max(\norm{x},\norm{x+tv}).
\end{eqnarray}
\end{lemma}

\begin{lemma}\label{A}
	For any $x,y,u,v\in\R^d$,
	\begin{eqnarray}		
\lefteqn{\max_{s\leq t} \norm{q_s(x,u)-q_s(y,v)-(x-y)-s(u-v)}}\nonumber\\
\label{10} 		& \leq &   L (t^2+th)\max\left(\norm{x-y},\norm{(x-y)+t(u-v)}\right),\qquad\mbox{and}
	\\
\nonumber\lefteqn{	\max_{s\leq t} \norm{p_s(x,u)-p_s(y,v)-(u-v)} 
\ \le\ Lt\, \max_{s\le t}|q_s(x,u)-q_s(y,v)|}
\\ &\leq  &  
		Lt(1+L(t^2+th))\max\left(\norm{x-y},\norm{(x-y)+t(u-v)}\right) .\label{11}
	\end{eqnarray}
%	In particular,
%	\begin{eqnarray}
%		\label{12}
%		&&\max_{s\leq t} \norm{q_s(x,u)-q_s(y,v)}\\ && \ \leq \  \norm{(x-y)+t(u-v)} +
%		L(t^2+ht)\max(\norm{x-y},\norm{(x-y)+t(u-v)}).\nonumber
%	\end{eqnarray}
\end{lemma}

 \begin{remark}
 	The lemma shows that on sufficiently short time intervals, the first variation of
 	velocity Verlet can be controlled by that of the corresponding motion with
 	constant velocity. In particular, contractivity for small times holds if $u-v=-\gamma (x-y)$
 	for some $\gamma>0$. \end{remark} 
 	
We will show next that in the region of strong convexity, the bounds in Lemma \ref{A}
 	can be improved if the initial velocities coincide. For such initial
 	conditions, \eqref{10} and \eqref{11} imply
 	\begin{eqnarray}
 	\label{13}
 	\norm{q_t(x,v)-q_t(y,v)-(x-y)} &\leq& L\,(t^2+ht)\,\norm{x-y},\\
 	\label{14} \norm{p_t(x,v)-p_t(y,v)} &\leq&
 	L\,t\,(1\,+\,L\,(t^2+ht))\,\norm{x-y}.
 	\end{eqnarray}
%Now suppose that in addition to the assumptions above, we have strong convexity
%outside a Euclidean ball, i.e.\ ,
%\begin{assumption}
%	\label{A2} There exist constants $R\in [0,\infty)$ and $K\in(0,\infty)$
%	such that for all $x,y\in\R^d$ with $\norm{x-y}\geq R$,
%	\begin{eqnarray}\label{15} 
%		(x-y)\cdot(\nabla U(x)-\nabla U(y))  & \geq &   K \norm{x-y}^2.
%	\end{eqnarray}
%\end{assumption}
For $|x-y|\ge 2\mathcal R$, the bound in \eqref{13} can be improved considerably:
\begin{lemma}\label{C}
%Suppose that Assumptions \ref{A1} and \ref{A2} are satisfied. Then 
There exists a finite constant
$C\in(0,\infty)$, depending only on $L$ and $M$, such that the
bound
\begin{eqnarray}\label{16}
	\norm{q_t(x,v)-q_t(y,v)}^2 &\leq& (1-K t^2/2)\, \norm{x-y}^2
\end{eqnarray}
holds for any $t,h$ as above such that
\begin{eqnarray} \label{A4}
	L\,(t^2+th) &<& K/L,
\end{eqnarray}
and for any $x,y,v\in\R^d$ such that
\begin{eqnarray}\label{A5}
	\norm{x-y} \ \geq \ 2\mathcal R \qquad\text{and}\qquad (1+\norm{x}+\norm{v})\, h \ \leq
	K/C.
\end{eqnarray}
\end{lemma}

\begin{remark}
The lemma does not provide a bound if $\norm{v}$ is very large. However,
  in this case we still have the upper bound
  \begin{eqnarray}\label{17}
  	\norm{q_t(x,v)-q_t(y,v)} &\leq& (1+L(t^2+th))\, \norm{x-y}
  \end{eqnarray}
  that follows from \eqref{13}. Hence if $\norm{v}$ is large with small
  probability, then we still get a contraction on average.
For the exact Hamiltonian dynamics, there is no corresponding
  restriction on $\norm{x}$ and $\norm{v}$. Here, the lemma immediately yields a
  contraction result for synchronous coupling.
\end{remark}

In the case of the exact Hamiltonian flow, i.e.\ for $h=0$, we have
 \begin{eqnarray}\label{18}
 	H(\phi_t(x,v)) = H(x,v) \qquad \text{for any } t\in\R \text{ and }
 	x,v\in\R^d.
 \end{eqnarray}
We are now going to quantify the error in \eqref{18} in the case where the exact flow is
  replaced by the flow of the velocity Verlet integrator. This is crucial to
  quantify the acceptance-rejection probabilities. 
%  We have
%  to replace Assumption \ref{A1} by a slightly stronger assumption: 
%  
% \begin{assumption}\label{A3}
% 	\begin{enumerate}
% 	  \item $U\in C^4(\R^d)$.
% 	  \item $U$ has bounded second, third and fourth derivatives.
% 	  \item $U$ has a local minimum at $0$.
% 	\end{enumerate}
% \end{assumption}
% 
% We set
% 	  $L=\sup\norm{\nabla^2 U}$, $M=\sup\norm{\nabla^3 U}$ and
% 	  $N=\sup\norm{\nabla^4 U}$.
 
\begin{lemma}\label{D}
	There exist finite
	constants
	$C_1,C_2\in(0,\infty)$ that depend only on $L$, $M$ and $N$ such that the
	bounds
	\begin{eqnarray}
		\label{H1} \norm{(H\circ\phi_t-H)(x,v)} & \leq &  C_1 t  h^2
		\max(\norm{x},\norm{v})^3,
	\\
	\label{H2} \norm{\partial_{(z,w)} (H\circ\phi_t-H)(x,v) } 
		& \leq &   C_2 th^2 \max(\norm{x},\norm{v})^3
		\max(\norm{z},\norm{w}),
	\end{eqnarray}
	hold for any $x,v,z,w\in\R^d$ and $t,h$ as above satisfying \eqref{0}.
\end{lemma}

%\begin{remark}
%	The same bounds hold on the torus if we set $\norm{x}=d(x,0)$.
%\end{remark}

\subsection{Bounds for acceptance-rejection probabilities}

We now provide some crucial bounds for probabilities and expectations that involve acceptance-rejection events and the coupling. Recall that the coupling
that we consider for $|x-y|<2\mathcal R$ ensures that $\xi -\eta =-\gamma z$
with the maximal possible probability, where $z=x-y$. The following lemma enables us to control probabilities and expectations when $\xi -\eta \neq -\gamma z$.

\begin{lemma}\label{lem:X}
For any $p\ge 1$ there exist finite constants $C_p$ and $\tilde C_p$ such that for any choice of $\gamma$,
\begin{eqnarray}
\label{X1} P[\xi-\eta\neq -\gamma z] &\le & |\gamma z|/\sqrt{2\pi}\\
\label{X2} E[|e\cdot\xi|^p;\, \xi-\eta\neq -\gamma z] &\le & C_p|\gamma z|\, \max (|\gamma z|,1)^p,\\
\label{X3} E[|\xi|^{2p};\, \xi-\eta\neq -\gamma z] &\le & \tilde C_p|\gamma z|\, \left( (d-1)^{p}+\max (|\gamma z|,1)^{2p}\right).
\end{eqnarray}
\end{lemma}
The a priori bounds \eqref{H1} and \eqref{H2} for velocity Verlet can be used to obtain a rather precise control for the rejection probabilities in HMC, and for the probability that in a coupling for HMC, the proposal is accepted for one component and rejected for the other. The resulting bounds are crucial to prove contractivity on average for the coupling.

\begin{theorem}\label{AR}
%Suppose that Assumption \ref{A3} is satisfied. Then 
There exist finite constants
$C_1,C_2,C_3\in(0,\infty)$ that depend only on $L,M$ and $N$ such that the
following bounds hold for any $x,y\in\R^d$ and $T\in[0,\infty)$, $h\in[0,1]$
such that $L(T^2+hT)\leq 1$:
\begin{eqnarray}
\label{AR1} P[A(x)^C | \xi]  \ &\leq &     C_1 T(1+T)
\max(\norm{x},\norm{\xi})^3 h^2,\\
\label{AR2} P[A(x)^C] \  &\leq &  C_1 T(1+T) (\norm{x}^3 + 2d^{3/2}) h^2,\\
\label{AR3} P[\hat{A}(y)^C | \eta] \   &\leq &     C_1 T(1+T)
\max(\norm{x},\norm{\eta})^3 h^2,\\
\label{AR4} P[\hat{A}(y)^C ]  \ &\leq & C_1 T(1+T) (\norm{x}^3 + 2d^{3/2})
h^2,
\\
\label{AR5}
\ P[A(x)\Delta A(y) | \xi] \  &\leq &   C_2 T(1+T)
\max(\norm{x},\norm{y},\norm{\xi})^3 \norm{x-y} h^2,\\
\label{AR6}
P[A(x)\Delta A(y)] \   &\leq &   C_2 T(1+T)
(\max(\norm{x},\norm{y})^3+2d^{3/2}) \norm{x-y} h^2,
\end{eqnarray}
\begin{eqnarray}
\label{AR7}\lefteqn{ P[A(x)\Delta \hat{A}(y) | \xi,\eta]}  \\
\nonumber   &\leq &   C_2 T(1+T)
\max(\norm{x-y},\norm{\xi -\eta }) h^2  
\cdot \max(\norm{x},\norm{y},\norm{\xi},\norm{\eta})^3
 \end{eqnarray}
 Furthermore, if $\gamma |x-y|\le 1$, then
 %{\color{red}XXX urspruenglich $\le 1/2$ aber das verbessert wohl nur C3} 
 \begin{eqnarray}
%\label{AR8} && P[A(x)\Delta \hat{A}(y)] \\&& \ \leq \ C_3 T(1+T) (1+\norm{x-y})
%(\max(\norm{x},\norm{y})^3+2 d^{3/2}) \norm{x-y} h^2.\nonumber\\
\label{AR9} && E[\max (|x|,|y|,|\xi |,|\eta |);\, A(x)\Delta \hat{A}(y)] \\&& \ \leq \ C_3\max (1,\gamma)\, T(1+T) 
%(1+\norm{x-y})
\left(\max(\norm{x},\norm{y})^4+d^{2}\right) \norm{x-y} h^2.\nonumber
 \end{eqnarray}
\end{theorem}

%The proof is based on the bounds in Lemma \ref{D}, see Section \ref{sec:proofs}.

\section{Proofs of a priori bounds}\label{sec:proofs}

If $h>0$ then we define $\lb{t}=\lb{t}_h$ and $\ub{t}=\ub{t}_h$ by \eqref{eq:round}. For $h=0$ we set $\lb{t}=\ub{t}=t$. In both cases,
$(q_t,p_t)$ solves \eqref{velVerlet}.

\begin{proof}[Proof of Lemma \ref{0A}]
We fix $x,v\in\R^d$. Let $x_s=q_s(x,v)$ and $v_s=p_s(x,v)$.
By \eqref{velVerlet}, we have for any $s\in[0,t]$ that
\begin{eqnarray*}
		x_s &=& x  +  \int_0^s v_{\lb{r}} \, dr  -  \frac{h}{2} \int_0^s
		\nabla U(x_{\lb{r}}) \, dr \\
		&=& x +  s v  -  \frac{1}{2} \int_0^s \int_0^{\lb{r}} \left(\nabla
		U(x_{\lb{u}})  + \nabla
		U(x_{\ub{u}})\right) \, du  \, dr \,-\,  \frac{h}{2}  \int_0^s \nabla
		U(x_{\lb{r}}) \,dr.
	\end{eqnarray*}
	By \eqref{A1.5} and since $t\in h\Z$,
	\begin{eqnarray*}
		\norm{x_s-x-s v} &\leq& \frac{L}{2} \int_0^s \int_0^r
		\left(\norm{x_{\lb{u}}}  + \norm{x_{\ub{u}}}\right) \, du \, dr  + 
		\frac{hL}{2} \int_0^s \norm{x_{\lb{r}}} \, dr
		\\&\leq& \frac{L}{2} (t^2+th) \max_{s\leq t} \norm{x_s},\qquad\text{ and thus}\\
		\max_{s\leq t} \norm{x_s-x-s v}  &\leq & \frac{L}{2} (t^2+th)
		\left(\max_{s\leq t} \norm{x_s-x-sv}+\max(\norm{x},\norm{x+vt})\right).
	\end{eqnarray*}
By \eqref{0}, we
obtain:  
\begin{eqnarray}\nonumber
	\max_{s\leq t} \norm{x_s-x-s v}  &\leq & L(t^2+th)
	\max(\norm{x},\norm{x+v t}),
\\  \max_{s\leq t} \norm{x_s} &\leq & (1+L(t^2+th)) \max(\norm{x},\norm{x+v t})\nonumber\\ & \le & 2 \max(\norm{x},\norm{x+v t}).\label{eq:4a}
\end{eqnarray}
We now derive bounds for $v_s$. By \eqref{velVerlet} and \eqref{A1.5},
\begin{eqnarray*}
	v_s &=& v - \frac{1}{2} \int_0^s \left((\nabla U)(x_{\lb{r}}) \ + \ (\nabla
	U)(x_{\ub{r}})\right) \, dr,
	\\
	\norm{v_s-v} & \leq & \frac{L}{2} \int_0^s \left( \norm{x_{\lb{r}}} \ + \
	\norm{x_{\ub{r}}}\right) \ dr.
\end{eqnarray*}
Since $t\in h\Z$, we obtain by \eqref{eq:4a} and \eqref{0},
\begin{eqnarray*}
	\max_{s\leq t} \norm{v_s-v} &\leq & L t \max_{s\leq t} \norm{x_s}  \ \leq \  
	L t (1+L(t^2+th))\max(\norm{x},\norm{x+v t}),\\
	 \max_{s\leq t} \norm{v_s} &\leq& 
 \norm{v} \ + \ 2Lt \max(\norm{x},\norm{x+v t}).
\end{eqnarray*}
\end{proof}

\begin{proof}[Proof of Lemma \ref{A}]
	The proof can be carried out in a similar way to the proof of Lemma \ref{0A},
	where instead of  \eqref{A1.5}, we directly apply the Lipschitz bound
	$|\nabla U(x)-\nabla U(y)|\le L|x-y|$ for $x,y\in\mathbb R^d$.
\end{proof}

\begin{proof}[Proof of Lemma \ref{C}]
	Notice that we are in the case where the initial velocities
	coincide. We fix $x,y,v\in\R^d$ such that \eqref{A5} holds true and
	set $x_s=q_s(x,v)$, $v_s=p_s(x,v)$, $y_s=q_s(y,v)$, 
	$z_s=q_s(x,v)-q_s(y,v)$ and $w_s=p_s(x,v)-p_s(y,v)$.
	In particular, $z_0=x-y$ and $w_0=0$. Let 
	$$ z_t^\star\ =\ \max_{s\le t}|z_s|\qquad\text{and}\qquad w_t^\star\ =\ \max_{s\le t}|w_s| .$$
	By Lemma \ref{A}, we have
\begin{eqnarray}
\label{eq:0}|z_t-z_0| & \le & L(t^2+ht)|z_0|,\quad\text{\ and}\\
\label{eq:0a}w_t^\star \ \le \  Ltz_t^\star & \le & 2Lt\quad\text{for any }t\in h\mathbb Z_+\text{ s.t.\ }L(t^2+ht)\le 1.
\end{eqnarray}	
The following computations are valid for $t\in\R_+$ such that $\norm{z_s}\geq
	\mathcal R$ for $s\in[0,t]$. Recall that by \eqref{velVerlet}, 
	\begin{eqnarray}\label{c1}
		\dot{z}_t &=& w_{\lb{t}} - \frac{h}{2} \left(\nabla U(x_{\lb{t}}) -
		\nabla U(y_{\lb{t}})\right),\\ \label{c2}
		\dot{w}_t &=& - \frac{1}{2} \left(\nabla U(x_{\lb{t}}) - \nabla
		U(y_{\lb{t}}) + \nabla U(x_{\ub{t}}) - \nabla
		U(y_{\ub{t}})\right).
	\end{eqnarray} 
	Let $a(t):=\norm{z_t}^2$ and $b(t):=2z_t\cdot w_t$. Our goal is to derive an upper bound for $a(t)$. To this end we note that $a$ and $b$ satisfy
	the following differential equations:
	\begin{eqnarray*}
		\dot{a}(t) &=& b(t) + \delta(t),\\
		\dot{b}(t) &=& - z_t\cdot\left(\nabla U(x_{\lb{t}})-\nabla
		U(y_{\lb{t}})+\nabla U(x_{\ub{t}})-\nabla U(y_{\ub{t}})\right) 
		\\&& + 2w_{\lb{t}} \cdot w_t  - h w_t\cdot\left(\nabla
		U(x_{\lb{t}})-\nabla U(y_{\lb{t}})\right) \\
		&=& -2 z_t\cdot\left(\nabla U(x_{t})-\nabla U(y_t))\right) + 2\norm{w_t}^2 +
		\epsilon(t),
		\end{eqnarray*}
		where
		\begin{eqnarray*}
		\delta(t) &=& 2z_t\cdot(w_{\lb{t}}-w_t) - h z_t\cdot\left(\nabla
		U(x_{\lb{t}})-\nabla U(y_{\lb{t}})\right)\\
		&=& \delta_1(t)+\delta_2(t)+\delta_3(t) \quad \text{with}\\
		\delta_1(t) &=& 2(t-\lb{t}-h/2) z_{\lb{t}}\cdot\left(\nabla
		U(x_{\lb{t}})-\nabla
		U(y_{\lb{t}})\right),\\
		\delta_2(t) &=& 2(t-\lb{t}-h/2) (z_t-z_{\lb{t}})\cdot\left(\nabla
		U(x_{\lb{t}})-\nabla
		U(y_{\lb{t}})\right),\\
		\delta_3(t) &=& (t-\lb{t}) z_t \cdot \left(\nabla
		U(x_{\ub{t}})-\nabla
		U(y_{\ub{t}})-\nabla
		U(x_{\lb{t}})+\nabla
		U(y_{\lb{t}})\right),\\
		\epsilon(t) &=& \epsilon_1(t) + \epsilon_2(t) + \epsilon_3(t)
		\quad\text{with}\\
		\epsilon_1(t) &=& 2 z_t\cdot\left(\nabla U(x_t)-\nabla U(y_t)\right)\\ &&
		+z_t\cdot\left(-\nabla U(x_{\lb{t}})+\nabla U(y_{\lb{t}})-\nabla
		U(x_{\ub{t}})+\nabla U(y_{\ub{t}})\right),\\
		\epsilon_2(t)&=& 2w_t\cdot(w_{\lb{t}}-w_t),\\
		\epsilon_3(t) &=& -h w_t\cdot\left(\nabla U(x_{\lb{t}})-\nabla U(y_{\lb{t}})
		\right).
	\end{eqnarray*} 
	We see that
$$
				\dot{b}(t) \ =\  -2 K a(t) + \beta(t)$$
	with a function $\beta$ satisfying
	\begin{equation}
	\label{betat}\beta(t) \ \leq\ 2 \norm{w_t}^2 + \epsilon (t).
\end{equation}			
	 The initial value problem
	\begin{align*}
		\dot{a} \ &= \ b + \delta, & a(0)&=\norm{z_0}^2, \\
		\dot{b} \ &= \ -2K a + \beta, & b(0) &= 0, 
	\end{align*}
	has a unique solution that is given by
			\begin{eqnarray}\nonumber
		a(t) &=& \cos\left(\sqrt{2K}t\right) \norm{z_0}^2 + \int_0^t
		\cos\left(\sqrt{2K}(t-r)\right)\delta(r)\, dr \\
		 && 
%		 -  \frac{1}{\sqrt{2K}}		\sin\left(\sqrt{2K} t\right) \delta(0)
		 + \int_0^t
		\frac{1}{\sqrt{2K}} \sin\left(\sqrt{2K} (t-r)\right) \beta(r)\,
		dr.\label{c5}
	\end{eqnarray}
	We now bound the terms $\delta$, $\epsilon$ and $\beta$.
Note first that the assumptions imply that $K t^2 \leq Lt^2 \leq 1 \leq
 	 \pi^2/2$. Hence $t\leq \pi/\sqrt{2K}$, and thus $\sin(\sqrt{2K}(t-r))\geq 0$ for any
 $r\in[0,t]$. 
% Moreover, by \eqref{15},
%\begin{eqnarray}\label{c5a}
% \delta(0) &=& - h (x-y)\cdot
%\left(\nabla U(x)-\nabla U(y)\right) \ \leq \ -hK \norm{x-y}^2 \leq 0.
%\end{eqnarray}
%Hence the first term in the second line of equation \eqref{c5}
%is negative.
Let $\bar{t}:=(\lb{t}+\ub{t})/2=\lb{t}+h/2$. Then for $f\in C^1$, 
\begin{eqnarray*}
	\norm{\int_{\lb{t}}^{\ub{t}} \left(r-\bar{t}\, \right) f(r) \, dr} &=&
	\norm{\int_{\lb{t}}^{\ub{t}} \left(r-\bar{t}\, \right)
	\left(f(r)-f(\bar{t})\right)\, dr}
	\\
	&\leq& \int_{\lb{t}}^{\ub{t}} \left(r-\bar{t}\,\right)^2 \, dr \, \sup\norm{f'}
	\ =\  \frac{h^3}{12} \sup \norm{f'}.
\end{eqnarray*}
Therefore, we obtain 
\begin{eqnarray*}
	\int_{\lb{t}}^{\ub{t}} \cos\left(\sqrt{2K}(t-r)\right) \delta_1(r) \, dr
	&\leq& \frac{h^3}{6}  \sqrt{2K}\,L  \norm{z_{\lb{t}}}^2.
	\end{eqnarray*}
	In particular, for $t\in
	h\Z_+$,
	\begin{eqnarray}\label{AB1}
	\int_0^t \cos\left(\sqrt{2K}(t-r)\right) \, \delta_1(r)\, dr &\leq& t h^2
	\frac{\sqrt{2K}L}{6} z_t^{\star ,2}
\end{eqnarray} 
where $z_t^{\star ,2}:=(z_t^\star )^2$.
Moreover, $\delta_2(t)$ is given by 
\begin{align*}
	2(t-\bar{t}) (t-\lb{t})
	(w_{\lb{t}}-\frac{h}{2}(\nabla U(x_{\lb{t}})-\nabla
	U(y_{\lb{t}}))\cdot(\nabla U(x_{\lb{t}})-\nabla
	U(y_\lb{t})),
\end{align*}
 and hence by \eqref{14}, for $t\in h\Z_+$, 
\begin{equation}
	\int_0^t \cos(\sqrt{2K}(t-r)) \, \delta_2(r)\, dr \leq  
\frac 12	h^2t(Lw_t^\star z_t^\star +\frac h2L^2z_t^{\star ,2})
%	 2 \frac{h}{2}
%	\int_0^t (r-\lb{r})\, dr\,\left(L w_t^\star 
%	z_t^\star + \frac{h}{2} L^2 z_t^{\star,2} \right)
	\leq   \frac{5}{4} t^2h^2  L^2 z_t^{\star,2}.\label{AB2}
\end{equation} 
%where $z_t^\star=\max_{s\leq t} \norm{z_s}$, $z_t^{\star,2}=\max_{s\leq t}
%\norm{z_s}^2$ and $w_t^\star=\max_{s\leq t} \norm{w_s}$.
%{\color{red} Da
%$\delta_2$ sich geaendert hat, hat sich auch die Abschaetzung gaendert.
%Moeglicherweise ist eine bessere Konstante moeglich.}
In order to control $\delta_3$ in an efficient way note that
\begin{eqnarray}\nonumber
\lefteqn{\norm{\nabla U (x_{\ub{t}})-\nabla U (x_{\lb{t}})-\nabla U (y_{\ub{t}})
	+\nabla U(y_{\lb{t}})}}\\\nonumber
	&=&\norm{\int_{\lb{t}}^{\ub{t}} \left(\nabla^2 U(x_r)
	\dot{x}_r-\nabla^2 U(y_r) \dot{y_r}\right) \, dr}
	\\\label{c6}&\leq& M \int_{\lb{t}}^{\ub{t}} \norm{z_r} \norm{\dot{x}_r}
	\, dr + L \int_{\lb{t}}^{\ub{t}} \norm{\dot{z}_r} \, dr
	\\\nonumber&\leq& Mh z_{\ub{t}}^\star
	\left(\norm{v_{\lb{t}}}+\frac{hL}{2} \norm{x_{\lb{t}}}\right) + Lh
	\left(\norm{w_{\lb{t}}} + \frac{hL}{2} \norm{z_{\lb{t}}}\right).
\end{eqnarray}
Therefore, we obtain for $t\in h\Z_+$, by \eqref{14},
\begin{eqnarray}\nonumber
\lefteqn{\int_0^t \cos\left(\sqrt{2K} (t-r)\right) \delta_3(r) \, dr  \ \leq \  t
	\delta_3^\star(t)}
	\\&\leq& th^2\left(M z_t^{\star,2}\left(v_t^\star+\frac{hL}{2}x_t^\star\right)
	+ 2 L^2 t z_t^{\star,2} + \frac{hL^2}{2} z_t^{\star,2}\right). \label{AB3}
\end{eqnarray} 
Next, we derive bounds for $\beta(t)$. We first observe that by \eqref{betat} and \eqref{14},
$$\beta(t)\ \leq\ 2 L^2 t^2 z_t^{\star,2} +\epsilon^\star (t),$$ 
and hence
\begin{eqnarray}
\lefteqn{\int_0^t \frac{1}{\sqrt{2K}} \sin\left(\sqrt{2K}(t-r)\right)
	\beta(r) \, dr}\label{c7}
	\\&  \leq & 2L^2 \int_0^t (t-r) r^2 \, dr \, z_t^{\star,2} + \int_0^t
	(t-r) \, dr\, {\epsilon^\star (t)} 
	\ =\ \frac{1}{6}L^2 t^4 z_t^{\star,2} + \frac{1}{2}t^2
	\epsilon^\star(t).\nonumber
\end{eqnarray}
The terms $\epsilon_1$, $\epsilon_2$ and $\epsilon_3$ can be controlled
similarly to $\delta_1$, $\delta_2$ and $\delta_3$. Analogously to
\eqref{c6}, we obtain
\begin{eqnarray*}
 	\norm{\epsilon_1(t)} &\leq & z_t^\star \norm{\nabla U(x_{t})-\nabla
 	U(y_{t})-\nabla U(x_{\lb{t}}) + \nabla U(y_{\lb{t}}) } 
 	\\&&+ \ z_t^\star \norm{\nabla U(x_{\ub{t}})-\nabla U(y_{\ub{t}})-\nabla
 	U(x_{{t}}) + \nabla U(y_{{t}}) }  \\
 	&\leq& 2Mh z^{\star,2}_{\ub{t}} \left(v^\star_{\ub{t}}+\frac{hL}{2}
 	x_{\ub{t}}^\star\right) + 2Lh z_{\ub{t}}^\star \left(w^\star_{\ub{t}}+
 	\frac{hL}{2} z^\star_{\ub{t}}\right), 
\end{eqnarray*}
and thus by \eqref{14}, for $t\in h\Z_+$,
\begin{eqnarray*}
	\epsilon_1^\star(t) &\leq& h z_t^{\star,2} \left(2 M v_t^\star + 4 L^2t +
	 h L M x_t^\star + h L^2\right),\\
	\epsilon_2^\star(t) &\leq&  2Lh w_t^\star z_t^\star \ \leq \ 4 L^2
	h t z_t^{\star,2}, \\
	\epsilon_3^\star(t)  &\leq& L h w_t^\star z_t^\star \ \leq \ 2 L^2 h t z_t^{\star,2}.
\end{eqnarray*}
Thus in total, we obtain by \eqref{c7},
\begin{eqnarray*}
\lefteqn{ \int_0^t \frac{1}{\sqrt{2K}} \sin(\sqrt{2K}(t-r))\, \beta(r)\, dr}
	\\ &\leq& \left(\frac{1}{6}L^2 t^4 + \frac{1}{2}t^2
	h\left(2Mv_t^\star+10L^2t+hLMx_t^\star+hL^2\right) \right)
	z_t^{\star,2}.
\end{eqnarray*}
%In combination with the bounds for $\delta_1$, $\delta_2$ and $\delta_3$, we
%obtain by \eqref{c5} and \eqref{c5a} for $t\in h\Z$, 
By \eqref{c5}, %\eqref{c5a}
, \eqref{AB1}, \eqref{AB2} and \eqref{AB3}, we obtain
\begin{eqnarray}
\label{c7a}
\norm{z_t}^2 & =& a(t) \ \leq\ \cos\left(\sqrt{2K} t\right) \norm{z_0}^2 +
\left(\frac{1}{6}L^2 t^4 + A_t\right) z_t^{\star,2},\qquad\text{where}\\
\nonumber A_t&=& \left(t^2
	h+th^2\right)\left(Mv_t^\star+5L^2t+hLMx_t^\star+hL^2\right) +L^{3/2}t h^2 + 2 L^2  h^2 t^2. 
\end{eqnarray}
By Lemma \ref{0A},
%since by \ldots, 
%\begin{eqnarray*}
%	x_t^\star &\leq& 2 L \norm{x} + 2 L t \norm{v} \ \leq \ 2
%	\max(L,\sqrt{L})(\norm{x}+\norm{v}), \quad \text{and} \\
%	v_t^\star &\leq& 3 \norm{v} + 2 L t \norm{x} \ \leq \ 3 \max(L,\sqrt{L})
%	(\norm{x}+\norm{v}),
%\end{eqnarray*}
there is a finite constant $C$ that depends only on $K,L$ and $M$ such that for
any
$t\in h\Z_+$ satisfying \eqref{A4}, we have 
\begin{eqnarray}\label{c10}
	A_t &\leq & Ct^2h(1+\norm{x}+\norm{v}). 
\end{eqnarray}
Noting that
%Then by \eqref{A4}, we have (since $\cos(x)\leq
%1-\frac{x^2}{2}+\frac{x^4}{24}$),
%\begin{eqnarray}\label{c8}
$	\cos(\sqrt{2K} t) \leq 1-K t^2 +  K^2
	t^4/6,$
and $K^2 t^4 \leq  L^2 t^4 \leq K
	t^2$ by \eqref{A4}, we obtain by \eqref{c7a} and \eqref{A5},
%	\label{c9}
%	\frac{1}{6}K^2 t^4 &\leq& \frac{1}{6} L^2 t^4 \ \leq \  \frac{1}{6} K
%	t^2.
%\end{eqnarray}
%Hence we obtain by \eqref{c7a}, \eqref{c8}, \eqref{c9} and \eqref{c10}: 
\begin{eqnarray}
\label{c11}
\norm{z_t}^2 &\leq& (1-K t^2) \norm{z_0}^2 + \left(\frac{1}{3}K t^2 +
Cht^2(1+\norm{x}+\norm{v})\right) z_t^{\star,2}\\
\label{c12}  &\leq& (1-K t^2) \norm{z_0}^2 + \frac{1}{2}
	K t^2 z_t^{\star,2}
\end{eqnarray}
%{\color{red}Aufgrund anderer Konstanten muesste der Rest ab hier noch
%angepasst werden (falls das vorher nun alles ok ist).} 
%Assuming \eqref{A5}, we
%have
%\begin{eqnarray}
%	\label{c12} \norm{z_t}^2 &\leq& (1-K t^2) \norm{z_0}^2 + \frac{1}{2}
%	K t^2 z_t^{\star,2}
%\end{eqnarray}
for any $t\in h\cdot \mathbb Z_+$ s.t.\ \eqref{A4} holds. This inequality then
implies
\begin{eqnarray}
	\label{c13} \norm{z_t}^2 &\leq& (1-\frac{1}{2} K t^2) \norm{z_0}^2
\end{eqnarray}
for $t$ as before. Indeed, suppose first that $h>0$. Then \eqref{c13} follows directly from \eqref{c12} if
$\norm{z_t}\leq \norm{z_0}$ holds for any $t>0$ satisfying \eqref{A4}. Now suppose for
a contradiction that there exists $t>0$ s.t.\ \eqref{A4} holds and
$\norm{z_t}>\norm{z_0}$. Since $z_t$ is linear on each partition interval, 
we may assume that $t\in h\Z_+$. Let $t_0$ denote the smallest $s\in
h\Z_+$ for which $\norm{z_s}>\norm{z_0}$. Then
$	z_{t_0}^\star = \norm{z_{t_0}}$, and hence by \eqref{c12},
	$\norm{z_{t_0}}\leq |z_0|$
in contradiction to the definition of $t_0$. Thus \eqref{c13} holds for all $t$ as
above.	
%	 (1-K t_0^2) \norm{z_0}^2 + \frac{1}{2} K
%	t_0^2 \norm{z_{t_0}}^2, \quad\text{and thus}\\
%	\norm{z_{t_0}}^2 &\leq & \frac{1-K t_0^2}{1-K t_0^2/2} \norm{z_0}^2 \
%	\leq\ \norm{z_0}^2 \quad \lightning.
%\end{eqnarray*}
%Therefore $\norm{z_t}\leq \norm{z_0}$ for all $t\in h\Z_+$ s.t.\
%\eqref{A4} holds, and thus \eqref{c12} implies \eqref{c13}. The argument applies
%for $h>0$. 
For $h=0$, we can argue similarly by the intermediate value theorem.\smallskip

Summarizing, we have shown that \eqref{c13} holds for $t\in h\Z_+$
satisfying \eqref{A4} provided $\norm{z_s}\geq \mathcal R$ for all $s\in[0,t]$ and $h$
satisfies \eqref{A5}. To conclude the proof
suppose that $\norm{z_0}\geq 2\mathcal R$. We claim that then $\norm{z_t}\geq \mathcal R$
holds for all $t$ satisfying \eqref{A4}. Indeed let $t_1:=\inf\{t:
\norm{z_t}<\mathcal R\}$. Then $\norm{z_s}\geq \mathcal R$ on $[0,t_1]$. Suppose for a
contradiction that $L(t_1^2 + t_1 h)<K/L\leq 1$. Then by
\eqref{c13}, $\norm{z_s}\leq \norm{z_0}$ for $s\in[0,t_1]$. Hence by \eqref{c1} and
\eqref{eq:0a},
\begin{eqnarray*}
\lefteqn{|z_{t_1}-z_0| \ = \  \norm{\int_0^{t_1} \left( w_{\lb{s}}-\frac{h}{2}(\nabla
	U(x_{\lb{s}})-\nabla U(y_{\lb{s}}))\right)\, ds}}\\ 
	&\leq& \frac{L t_1^2}{2} \norm{z_0} + \frac{L h t_1}{2} \norm{z_0} \ = \
	\frac{1}{2} L(t_1^2+h t_1) \norm{z_0} \ < \ \frac{1}{2} \norm{z_0},
\end{eqnarray*}
and thus $\norm{z_{t_1}}>\frac{1}{2}\norm{z_0} \geq  \mathcal R$ in contradiction to the definition of $t_1$. 
\end{proof} 

\begin{proof}[Proof of Lemma \ref{D}]
Fix $x,v,z,w\in\R^d$. We set $x_t=q_t(x,v)$, $v_t=p_t(x,v)$ and
$H_t = H(x_t, v_t) \ = \  \frac{1}{2} \norm{v_t}^2  \ + \ U (x_t)$.
Then
%\begin{eqnarray*}
%	H_t \ - \ H_0 &=& \int_0^t \frac{d}{ds} H_s \, ds, \qquad\text{with}
%\end{eqnarray*}
\begin{eqnarray}\nonumber
\lefteqn{\frac{d}{dt} H_t\   =\  v_t \cdot \frac{d}{dt} v_t \ + \ \nabla U(x_t)\cdot
\frac{d}{dt} x_t} \\
\nonumber &=&   - \frac{1}{2} v_t\cdot \left(\nabla U(x_{\lb{t}})+\nabla
 U(x_{\ub{t}})\right) + v_{\lb{t}} \cdot \nabla U(x_t)    -  \frac{h}{2}
 \nabla U(x_{\lb{t}})\cdot \nabla U(x_t), \\
\nonumber &=&  \rn{1}_t \ + \ \rn{2}_t \ + \ \rn{3}_t \ + \ \rn{4}_t,\qquad\text{where}
\end{eqnarray}
\begin{eqnarray}
\label{1t} \rn{1}_t&=& - \frac{1}{2} v_{\lb{t}}\cdot \left(\nabla U(x_{\lb{t}}) \ + \ \nabla
U(x_{\ub{t}}) \ - \ 2\nabla U(x_t)\right), \\
\label{2t} \rn{2}_t&=&  - \frac{1}{2} (v_{{t}}-v_{\lb{t}})\cdot \left(\nabla U(x_{\lb{t}}) \ +
 \ \nabla U(x_{\ub{t}}) \ - \ 2\nabla U(x_t)\right),\\
\label{3t}  \rn{3}_t&=&  (v_{\lb{t}}-v_{t})\cdot \nabla U(x_t) \ - \ \frac{h}{2}
\norm{\nabla U(x_t)}^2,
\\
\label{4t}  \rn{4}_t&=&   \frac{h}{2} \left(\nabla U(x_t) \ - \ \nabla
    U(x_{\lb{t}})\right)\cdot \nabla U(x_t). 
\end{eqnarray}
Furthermore,
\begin{eqnarray*}
\lefteqn{v_{\lb{t}}-v_t \ =\ \frac{t-\lb{t}}{2} \left(\nabla U(x_{\lb{t}}) \ +\ 
\nabla U(x_{\ub{t}})\right)}
\\ &=& (t-\lb{t}) \nabla U(x_t)  +  \frac{t-\lb{t}}{2} \left(\nabla
U(x_{\lb{t}})  +  \nabla U(x_{\ub{t}})  -  2\nabla U(x_t) \right),
\end{eqnarray*} 
\begin{eqnarray*}
\lefteqn{\nabla U(x_t)-\nabla U(x_{\lb{t}}) \ =\  \int_{\lb{t}}^t \nabla^2 U(x_s)
	\cdot \left(v_{\lb{t}} - \frac{h}{2} \nabla U(x_{\lb{t}})\right) \, ds} \\
	&=& (t-\lb{t}) \nabla^2 U(x_{\lb{t}}) \cdot v_{\lb{t}}  - 
	\int_{\lb{t}}^t (\nabla^2 U(x_s)  -  \nabla^2 U(x_{\lb{t}}))\cdot
	v_{\lb{t}} \,  ds\\
	&&  -  \frac{h}{2} \int_{\lb{t}}^t \nabla^2 U(x_s)\cdot \nabla
	U(x_{\lb{t}}) \, ds, \qquad\text{and hence,}
\end{eqnarray*}
\begin{eqnarray}\nonumber
 \lefteqn{ 2 \nabla U(x_t)  -  \nabla U(x_{\lb{t}})  -  \nabla U(x_{\ub{t}})}
	  \\ \nonumber &=&   (t -  \lb{t}  +  t  -  \ub{t}) \nabla^2 U(x_{\lb{t}})\cdot
	v_{\lb{t}}  +  \int_{\lb{t}}^t (\nabla^2 U(x_s)  -  \nabla^2 U(x_{\lb{t}})) \cdot
	v_{\lb{t}} \,  ds
	 \\  \nonumber&& -  \int_t^{\ub{t}} (\nabla^2 U(x_s) -  \nabla^2
	U(x_{\lb{t}})) \cdot  
	v_{\lb{t}} \, ds\,  -\, \frac{h}{2} \int_{\lb{t}}^{t} \nabla^2 U(x_s)\cdot \nabla
	U(x_{\lb{t}}) \, ds  \\  
	\label{eq:5ta} && +  \frac{h}{2} \ \int_t^{\ub{t}} \nabla^2
	U(x_s) \cdot \nabla U(x_{\lb{t}}) \, ds \\
	&=& 2(t-\bar{t}) \nabla^2 U(x_{\lb{t}})\cdot v_{\lb{t}} \ + \ \rn{5}_t,\nonumber
\end{eqnarray}
where $\bar{t}=(\lb{t}+\ub{t})/2$ and
\begin{eqnarray}
	\norm{\rn{5}_t}  &\leq& M \norm{v_{\lb{t}}} \int_{\lb{t}}^{\ub{t}}
	\norm{x_s-x_{\lb{t}}} \, ds + \frac{1}{2} L^2 h^2 
	\norm{x_{\lb{t}}}\label{eq:5tb}
	\\
\nonumber 	&=& \frac{1}{2} M h^2
	\norm{v_{\lb{t}}}
	\norm{v_{\lb{t}}-\frac{h}{2} \nabla U(x_{\lb{t}})} + \frac{1}{2}
	L^2 h^2
	\norm{x_{\lb{t}}}
\end{eqnarray}
Consequently, $\rn{1}_t = \rn{1}_t^a + \rn{1}_t^b$, where
\begin{eqnarray*}
	\rn{1}_t^a &=& (t-\bar{t}) v_{\lb{t}} \cdot \nabla^2 U(x_{\lb{t}}) v_{\lb{t}},
	\quad \rn{1}_t^b \ = \ \frac{1}{2} v_{\lb{t}}\cdot\rn{5}_t .
\end{eqnarray*}
In particular, for any $t\in h\mathbb Z_+$, $
	\int_{0}^{t} \rn{1}_s^a \, ds = 0$, and
%	\norm{\rn{1}_t^b} &\leq& \frac{1}{4} h^2 \norm{v_{\lb{t}}}
%	\left(M
%	\norm{v_{\lb{t}}}^2+\frac{hLM}{2}\norm{v_{\lb{t}}}\norm{x_{\lb{t}}}+ L^2
%	\norm{x_{\lb{t}}}\right),\\
$$	\rn{1}_t^{b,\star}\ =\ \sup_{s\le t}	\norm{\rn{1}_s^b}\ \leq\ \frac{h^2}{4} \left(M
	(v_t^\star)^3+\frac{hLM}{2}(v_t^\star)^2 x_t^\star + L^2 x_t^\star v_t^\star
	\right).
$$
Similarly, we obtain
\begin{eqnarray}\nonumber
	\norm{\rn{2}_t} &\leq& \frac{L}{2} \norm{v_t-v_{\lb{t}}}
	\left(\norm{x_{\lb{t}}-x_t}+\norm{x_{\ub{t}}-x_t}\right)\\
\nonumber	&=&\frac{L}{4}h^2 \norm{\left(\nabla U(x_{\lb{t}})+\nabla
	U(x_{\ub{t}})\right)}\cdot \norm{v_{\lb{t}}-\frac{h}{2}\nabla U(x_{\lb{t}})}\quad\text{for }t\ge 0,\\
\nonumber	\rn{2}_t^\star &\leq&  \frac{L^2}{2}h^2 \left(x_t^\star v_t^\star +
	\frac{Lh}{2} (x_t^\star)^2\right)	\quad\text{for }t\in h\mathbb Z_+,\\
%\end{eqnarray*}
%We continue estimating $\rn{3}_t$.
% \begin{eqnarray*} calculations from sheet 0c
% 	\rn{3}_t &=& (t-\lb{t}-\frac{h}{2}) \norm{\nabla U(x_t)}^2 
% 	\\&& + \frac{t-\lb{t}}{2}
% 	\left(\nabla U(x_{\lb{t}})+\nabla U(x_{\ub{t}})-2\nabla
% 	U(x_{t})\right)\cdot\nabla U(x_t)\\
% 	&=& \rn{3}_t^a + \rn{3}_t^b + \rn{3}_t^C, \quad\text{where} \\
% 	\rn{3}_t^a &=& (t-\bar{t}) \norm{\nabla U(x_{\lb{t}})}^2 \\
% 	\rn{3}_t^b &=& (t-\bar{t}) \left(\norm{\nabla U(x_t)}^2-\norm{\nabla
% 	U(x_{\lb{t}})}^2\right) \\
% 	\rn{3}_t^c &=& \frac{t-\lb{t}}{2} \left(\nabla U(x_{\lb{t}}) + \nabla
% 	U(x_{\ub{t}}) - 2\nabla U(x_t) \right)\cdot\nabla U(x_t)
% \end{eqnarray*}
% In particular, $\int_{\lb{t}}^{\ub{t}} \rn{3}_s^a\, ds=0$.
%continuing with sheet 0ca
%AB HIER WEITER Prüfen!
%\begin{eqnarray*}
\nonumber\rn{3}_t &=& \rn{3}_t^a + \rn{3}_t^b + \rn{3}_t^c,\qquad\text{ where}\\
\rn{3}_t^a 
\label{eq:3ta}&=& -(t-\lb{t}-\frac{h}{2}) \norm{\nabla U(x_{\lb{t}})}^2 ,
\\ \label{eq:3tb}\rn{3}_t^b &=&  (v_t - v_{\lb{t}}-\frac{h}{2}\nabla
U(x_t))\cdot \left(\nabla U(x_t)-\nabla U(x_{\lb{t}})\right) ,\\
\label{eq:3tc}\rn{3}_t^c &=& \frac{h}{2} \left(\nabla U(x_t) - \nabla U(x_{\lb{t}})\right)\cdot\nabla
U(x_t)\ =\ \rn{4}_t .
\end{eqnarray}
In particular, for any $t\in h\Z_+$, $\int_0^t \rn{3}_s^a \, ds=0$, and
\begin{eqnarray*}
%	\norm{\rn{3}_t^b} &\leq& \norm{v_t-v_{\lb{t}}+\frac{h}{2}} \norm{\nabla
%	U(x_t)-\nabla U(x_{\lb{t}})}\\
%	&\leq& \left(\norm{v_t-v_{\lb{t}}}+\frac{hL}{2} \norm{x_t}\right) L
%	\norm{x_t-x_{\lb{t}}} \\
%	&\leq& \frac{3}{2} L^2h^2 \left(v_{\lb{t}}+ \frac{hL}{2} \norm{x_t}\right) \\
	\rn{3}_t^{b,\star} &\leq& \frac{3}{2}h^2  L^2 \left(v_t^\star x_t^\star+
	\frac{hL}{2} (x_t^\star)^2 \right), \\
%	\norm{\rn{3}_t^c} &\leq& \frac{hL^2}{2} \norm{x_t-x_{\lb{t}}} \norm{x_t} \\
	\rn{3}_t^{c,\star} &=& \rn{4}_t^\star\ \leq \ \frac 12{h^2L^2} \left(v_t^\star x_t^\star  +
	\frac{hL}{2} (x_t^\star)^2\right) .
\end{eqnarray*}
%Finally, we have 
%\begin{eqnarray*}
%	\norm{\rn{4}_t} &\leq& \frac{hL^2}{2} \norm{x_t-x_{\lb{t}}} \norm{x_t} \leq
%	\frac{L^2h^2}{2} \norm{v_{\lb{t}}-\frac{h}{2}\nabla U(x_{\lb{h}})} \norm{x_t}\\
%	\rn{4}_t^\star &\leq& \frac{L^2h^2}{2} \left(x_t^\star v_t^\star +
%	\frac{hL}{2} (x_t^\star)^2\right)
%\end{eqnarray*}
By combining the bounds, we obtain for $t\in h\Z_+$:
\begin{eqnarray*}
	\norm{H_t-H_0} &=& \norm{\int_0^t
	(\rn{1}_s^a+\rn{1}_s^b+\rn{2}_s+\rn{3}_s^a+\rn{3}_s^b+\rn{3}_s^c+\rn{4}_s)\,
	ds}
%	\\
%	&=& \norm{\int_0^t
%	(\rn{1}_s^b+\rn{2}_s+\rn{3}_s^b+\rn{3}_s^c+\rn{4}_s)\,
%	ds}
%\\&\leq& t
%(\rn{1}_t^{b,\star}+\rn{2}_t^\star+\rn{3}_t^{b,\star}+\rn{3}_t^{c,\star}+\rn{4}^\star_t)
\\&\leq& t h^2 \left(\frac{M}{4} 
	(v_t^\star)^3+\frac{hL}{8}(v_t^\star)^2 x_t^\star 
	+ 3L^2 v_t^\star x_t^\star + \frac{3}{2} hL^3 (x_t^\star)^2\right)
\end{eqnarray*}
This implies the first claim \eqref{H1}, since for $t\in h\mathbb Z_+$ satisfying 
\eqref{0}, both $x_t^\star$ and $v_t^\star$ are bounded by a constant multiple 
of $\max (|x|,|v|)$.\smallskip

Next, we consider the derivative flow
$$
	x_t' \ =\ (\partial_{(z,w)} q_t)(x,v), \qquad 
	v_t' \ =\ (\partial_{(z,w)} p_t)(x,v),$$
where the derivatives are taken w.r.t.\ the
	initial condition. We have
\begin{eqnarray}
\label{eq:xtp}
	\frac{d}{dt} x_t' &=& \partial_{(z,w)} \dot x_t \ = \ v_{\lb{t}}' -
	\frac{h}{2} \nabla^2 U(x_{\lb{t}}) x_{\lb{t}}', \\
\label{eq:vtp}	\frac{d}{dt} v_t' &=& \partial_{(z,w)} \dot v_t \ = \
	-\frac{1}{2}\left(\nabla^2 U(x_{\lb{t}})x_{\lb{t}}' + \nabla^2
	U(x_{\ub{t}})x_{\ub{t}}'\right)
	\end{eqnarray}
with initial condition $(x_0',v_0') = (z,w)$. In particular, 
 for $s,t\in\R_+$ s.t.\ $s\in[\lb{t},\ub{t}]$, 
	\begin{eqnarray*}
		\norm{x_t'-x_s'} &=& \norm{t-s}\cdot\norm{v_{\lb{t}}'-\frac{h}{2}\nabla^2
		U(x_{\lb{t}})x'_{\lb{t}}} \ \leq \ h \left(v_t'^\star + \frac{hL}{2}
		x_t'^\star\right),\\
		\norm{v_t'-v_s'} &=& \frac{1}{2} \norm{t-s} \norm{\nabla^2
		U(x_{\lb{t}})x'_{\lb{t}}+\nabla^2
		U(x_{\ub{t}})x'_{\ub{t}} } \ \leq \ h Lx'^{\star}_{\ub{t}}.
	\end{eqnarray*}
We now first derive a priori bounds for $x_t'^\star$ and $v_t'^\star$. By 
\eqref{eq:xtp} and \eqref{eq:vtp},
\begin{eqnarray*}
%		x_t' &=& x_0' + \int_0^t v_{\lb{s}}'\, ds - \frac{h}{2} \int_0^t \nabla^2
%		U(x_{\lb{s}}) x_{\lb{s}}' \, ds \\
\lefteqn{|x_t'-z - w t|}\\		&=&\left|  \frac{1}{2} \int_0^t \int_0^{\lb{s}}
		\left(\nabla^2U(x_{\lb{r}})x'_{\lb{r}}+\nabla^2
		U(x_{\ub{r}})x'_{\ub{r}}\right)\,dr\,ds\,  -\, \frac{h}{2} \int_0^t \nabla^2 U(x_{\lb{s}})x'_{\lb{s}}\, ds\right|\\
&\leq& \frac{L}{2} \int_0^t \int_0^{\lb{s}}
\left(\norm{x'_{\lb{r}}} + \norm{x'_{\ub{r}}} \right) \, dr \, ds + \frac{hL}{2}
\int_0^t \norm{x'_{\lb{s}}}\, ds
\end{eqnarray*}
For $t\in h\mathbb Z_+$, we obtain
\begin{eqnarray*}
	\max_{s\leq t} \norm{x_s' - z - ws} &\leq&  \frac{L}{2}(t^2+ht)\left(\max_{s\leq t} \norm{z+ws} +
	\max_{s\leq t} \norm{x_s'-z-ws}\right)
\end{eqnarray*}
Hence if $L(t^2+ht)\leq 1$ then
\begin{equation}\label{bxp}
	\max_{s\leq t} \norm{x_s'-z-ws} \ \leq\  L(t^2+ht) \max(\norm{z},\norm{z+wt}). 
\end{equation}
Similarly, by \eqref{eq:vtp} and \eqref{bxp},
\begin{equation}\label{bvp}
\max_{s\leq t} \norm{v_s'-w} \ \leq\  Lt \max_{s\leq t} \norm{x_s'} 
	\ \leq \  2Lt \max(\norm{z}, \norm{z+wt}).
\end{equation}
%\begin{eqnarray*}
%	p'_t &=& v_0' - \frac{1}{2} \int_0^t \left(\nabla^2
%	U(x_{\lb{s}})x'_{\lb{s}}+ \nabla^2
%	U(x_{\ub{s}})x'_{\ub{s}}\right)\, ds \\ 
%	\norm{p'_t-v_0'} &\leq& \frac{1}{2} L \int_0^t
%	\left(\norm{x'_{\lb{s}}}+\norm{x'_{\ub{s}}}\right) \, ds \ \leq \ t L
%	\max\norm{x_s'}, \\
%	\max_{s\leq t} \norm{v_s'-v_0'} &\leq& Lt \max_{s\leq t} \norm{x_s'} 
%	\ \leq \  2Lt \max(\norm{z}) \norm{z+wt} \\
%	\max_{s\leq t} \norm{v_s'}&\leq& \norm{w} + 2 t L \max(\norm{z},\norm{z+wt})
%\end{eqnarray*}
Now we can derive bounds for $H_t'$. We have
\begin{equation}\label{dHp}
\frac{d}{dt} H_t' = \left(\frac{d}{dt} H_t\right)' = \rn{1}_t' + \rn{2}_t' +
\rn{3}_t' + \rn{4}_t'.
\end{equation}
Similarly as above, we bound the terms $\rn{1}_t'$, $\rn{2}_t'$,
$\rn{3}_t'$ and $\rn{4}_t'$ individually. By \eqref{1t}, $\rn{1}_t' = \rn{6}_t - \frac{1}{2} v_{\lb{t}}\rn{7}_t$, where
\begin{eqnarray*}
\rn{6}_t &=&- \frac{1}{2} v'_{\lb{t}}\left(\nabla U(x_{\lb{t}})+\nabla
	U(x_{\ub{t}}) - 2\nabla U(x_t)\right),\\
	\rn{7}_t &=& \nabla^2 U(x_{\lb{t}}) x'_{\lb{t}}+\nabla^2
	U(x_{\ub{t}}) x'_{\ub{t}} - 2\nabla^2 U(x_t) x_t'.
	\end{eqnarray*}
%	\rn{1}_t' &=& - \frac{1}{2} p'_{\lb{t}}\left(\nabla U(x_{\lb{t}})+\nabla
%	U(x_{\ub{t}}) - 2\nabla U(x_t)\right)\\
%	&& - \frac{1}{2} v_{\lb{t}}\left(\nabla^2 U(x_{\lb{t}}) x'_{\lb{t}}+\nabla^2
%	U(x_{\ub{t}}) x'_{\ub{t}} - 2\nabla^2 U(x_t) x_t'\right)\\
%	&=& \rn{6}_t - \frac{1}{2} v_{\lb{t}}\rn{7}_t, \quad \rn{6}_t\ = \
%	\rn{6}_t^a+\rn{6}_t^b,
Similarly to the decomposition of $\rn{1}_t$ above, we have $\rn{6}_t =	\rn{6}_t^a+\rn{6}_t^b$ where $\rn{6}_t^a=(t-\bar{t})v_{\lb{t}}' \cdot \nabla^2 U(x_{\lb{t}}) v_{\lb{t}}$
and $\rn{6}_t^b = \frac{1}{2} v_{\lb{t}}' \cdot \rn{5}_t$
In particular,
for $t\in h\mathbb Z_+$,
$$
	\int_{\lb{t}}^{\ub{t}} \rn{6}^a_s \, ds \ =\  0, \qquad\text{and}\qquad
	\rn{6}_t^{b,\star} \ \leq\  \frac{M h^2}{4} \left(v_t'^{\star} v_t^\star
	+ \frac{hL}{2} v_t'^\star x_t^\star\right).
$$
Furthermore,
$\rn{7}_t = \rn{7}_t^a + \rn{7}_t^b + \rn{7}_t^c$ with
\begin{eqnarray*}
	\rn{7}_t^a &=& \nabla^2 U(x_t) (x'_{\lb{t}} + x'_{\ub{t}} -2 x'_t) \\ & =& 2 (\bar t-t)  \nabla^2 U(x_{\lb{t}})
	\left(v_{\lb{t}}'-\frac{h}{2}\nabla^2 U(x_{\lb{t}}) x'_{\lb{t}}\right),\\
	\rn{7}_t^b &=& \left(\nabla^2 U(x_{\lb{t}}) + \nabla^2 U(x_{\ub{t}}) -
	2\nabla^2 U(x_t)\right) x'_{\lb{t}},\\
	\rn{7}_t^c &=& \left(\nabla^2 U(x_{\ub{t}})  -
	\nabla^2 U(x_t)\right)\left(x'_{\ub{t}}- x'_{\lb{t}}\right) .
\end{eqnarray*}
For $t\in h\Z_+$, $\int_0^t \rn{7}_s^a \, ds = 0$. Moreover, 
similarly to \eqref{eq:5ta} and \eqref{eq:5tb}, we have
$$
2 \nabla^2 U(x_t) - \nabla^2 U(x_{\lb{t}}) - \nabla^2 U(x_{\ub{t}})\
	= 2(t-\bar{t}) \nabla^3 U(x_{\lb{t}}) \cdot v_{\lb{t}}+ \rn{8}_t,
$$
where $\norm{\rn{8}_t}\leq \frac{1}{2} N h^2 \norm{v_{\lb{t}}}
\norm{v_{\lb{t}}-\frac{h}{2}\nabla U(x_{\lb{t}})}+\frac 12LM h^2\norm{x_{\lb{t}}}$.
Therefore, we can decompose $\rn{7}_t^b=\rn{7}_t^d +\rn{7}_t^e$ 
where $\int_0^t \rn{7}_s^d \, ds = 0$ for $t\in h\Z_+$, and
%$$
%	\norm{-\frac{1}{2} v_{\lb{t}}\cdot \rn{7}_t^b} &=& (t-\bar{t}) \nabla^3
%	U(x_{\lb{t}}) [v_{\lb{t}},v_{\lb{t}},x_{\lb{t}}'] + \frac{1}{2} v_{\lb{t}}\cdot
%	\tilde{R}_t x_{\lb{t}}'\\
%	&=&\rn{7}_t^{b,a} + \rn{7}_t^{b,b}
%$$
$$
	\rn{7}_t^{e,\star} \ \leq\ \frac{h^2}{4}\left(N v_t^{\star ,3}+\frac{hLN}{2} v_t^{\star ,2} x_t^\star+
	LM v_t^\star  x_t^\star\right) {x_t'^\star}.
$$
Furthermore, by \eqref{eq:xtp}, we have
%\begin{eqnarray*}
%	\norm{\rn{7}_t^c} &=& \norm{(\nabla^2 U(x_{\ub{t}})-\nabla^2
%	U(x_t))(x_{\ub{t}}'-x'_{\lb{t}})} \\
%	&\leq& L_2 \norm{x_{\ub{t}}-x_t} \norm{x'_{\ub{t}}-x'_{\lb{t}}} \\
%	&\leq& L_2 h^2 \norm{v_{\lb{t}}-\frac{h}{2}\nabla U (x_{\lb{t}})}
%	\norm{p'_{\lb{t}}-\frac{h}{2}\nabla^2 U(x_{\lb{t}})x'_{\lb{t}}} \\
$$	
	\rn{7}_t^{c,\star} \ \leq\ M h^2 \left(v_t^\star+\frac{hL}{2} x_t^\star\right)
	\cdot\left(v_t'^\star+\frac{hM}{2} x_t'^\star \right) .
$$
For the second term we have
$\rn{2}_t' = \rn{9}_t + \rn{10}_t+ \rn{11}_t$ where
\begin{eqnarray*}
	{\rn{9}_t} &=& {-\frac{1}{2}(v'_t-v'_{\lb{t}})(\nabla
	U(x_\lb{t})+\nabla U(x_{\ub{t}})-2\nabla U(x_t))}, \\
	{\rn{10}_t} &=& -{\frac{1}{2} (v_t-v_{\lb{t}}) \left(\nabla^2
	U(x_{\lb{t}})+\nabla^2 U(x_{\lb{t}})-2\nabla^2 U(x_t)\right)x'_t }, \\
	{\rn{11}_t} &=& -{\frac{1}{2}(v_t-v_{\lb{t}})\left(\nabla^2
	U(x_{\lb{t}})(x'_{\lb{t}}-x_t')+\nabla^2 U(x_{\ub{t}})(x_{\ub{t}}'-x_t')\right)}. 
\end{eqnarray*}		
For $t\in h\Z_+$, we obtain by \eqref{eq:xtp} and \eqref{eq:vtp},
\begin{eqnarray*}
	\rn{9}_t^\star &\leq& {L M h^2} x_t'^\star
	(v_t^\star+{hL}x_t^\star /2)/2,\\
	\rn{10}_t^{\star} &\leq& { LM h^2} x_t^\star (v_t^\star +
	{hL} x_t^\star /2) x_t'^\star /2 ,\\
	\rn{11}_t^{\star} &\leq&{L^2 h^2} x_t^\star (v_t'^\star +
	{hL} x_t'^\star /2)/2 .
\end{eqnarray*}
%\begin{eqnarray*}
%	\norm{\rn{9}_t} &=& - \frac{1}{2} (v_t-v_{\lb{t}}) (\nabla^2
%	U(x_{\lb{t}})x_{\lb{t}}'+\nabla^2 U(x_{\ub{t}})x'_{\ub{t}}-2\nabla^2 U(x_t)
%	x_t') \\
%	&=& \rn{9}_t^a + \rn{9}_t^b \\
%	\norm{\rn{9}_t^a} &=& \norm{\frac{1}{2} (v_t-v_{\lb{t}}) (\nabla^2
%	U(x_{\lb{t}})+\nabla^2 U(x_{\lb{t}})-2\nabla^2 U(x_t) x_t' )} \\
%	\rn{9}_t^{a,\star} &\leq& \frac{L_2 L h^2}{2} x_t^\star (v_t^\star +
%	\frac{hL}{2} x_t^\star) x_t'^\star \\
%	\rn{9}_t^b &=& \norm{\frac{1}{2}(v_t-v_{\lb{t}})(\nabla^2
%	U(x_{\lb{t}})(x'_{\lb{t}}-x_t')+\nabla^2 U(x_{\ub{t}})(x_{\ub{t}}'-x_t'))} \\
%	\rn{9}_t^{b,\star} &\leq& \frac{L^2 h^2}{2} x_t^\star (v_t'^\star +
%	\frac{hL}{2} x_t'^\star)
%\end{eqnarray*}
Furthermore, $\rn{3}_t' = (\rn{3}_t^{a})' + (\rn{3}_t^{b})' + (\rn{3}_t^{c})'$. By
\eqref{eq:3ta} and the chain rule,\\ $\int_0^t (\rn{3}_t^a)' \, ds = 0$ for $ t\in h\Z_+$. Moreover, by \eqref{eq:3tb} and the chain rule,
\begin{eqnarray*}
	\norm{(\rn{3}_t^b)'} 
%	&=& \left\lvert(v_t'-v_{\lb{t}}'-\frac{h}{2}\nabla^2
%	U(x_t)x_t')\cdot (\nabla U(x_t)-\nabla U(x_{\lb{t}}))\right.\\&&
%	\left. +\ 
%	(v_{t}-v_{\lb{t}}-\frac{h}{2}\nabla U(x_t))\cdot (\nabla^2 U(x_t)
%	x_t' - \nabla^2 U(x_{\lb{t}})x_{\lb{t}}')\right\rvert \\
	&\leq & \left(\norm{v_t'-v'_{\lb{t}}}+\frac{hL}{2} \norm{x_t'}\right) L
	\norm{x_t-x_{\lb{t}}} \\ 
	&& + \ (\norm{v_t-v_{\lb{t}}}+\frac{hL}{2}
	x_{\lb{t}})\left(L\norm{x_t'-x_{\lb{t}}'}+M\norm{x_t-x_{\lb{t}}}
	\norm{x_t'}\right) \\
	&\leq& \frac{3}{2} L^2h^2 \norm{x_{\lb{t}}'} \norm{v_{\lb{t}}+\frac{hL}{2}
	x_{\lb{t}}} \\
	&& + \ \frac{3}{2} L^2 h^2 \norm{x_t}
	\norm{v'_{\lb{t}}+\frac{hL}{2}x'_{\lb{t}}} + \frac{3}{2} L M h^2 \norm{x_t}
	\norm{v_{\lb{t}}+\frac{hL}{2} x_{\lb{t}}} \norm{x_t'},\\
	(\rn{3}_t^b)'^{,\star} &\leq& \frac{3}{2} L^2 h^2 \left(x_t'^{\star}(v_t^\star +
	\frac{hL}{2}x_t^\star) + x_t^\star (v_t'^\star + \frac{hL}{2} x_t'^\star)\right)\\
	&&\ + \ \frac{3}{2} L M h^2 x_t^\star (v_t^\star + \frac{hL}{2}
	x_t^\star) x_t'^\star
\end{eqnarray*}
Finally, a similar computation as for $(\rn{3}_t^b)'$ shows that
\begin{eqnarray*}
(\rn{3}_t^c)'^{,\star}\ =\ \rn{4}_t'^{,\star} &\le&
\frac 12 L^2h^2\left(  x_t'^\star (v_t^\star + \frac{hL}{2}
	x_t^\star)+ x_t^\star (v_t'^\star + \frac{hL}{2} x_t'^\star)\right)\\
	&&\ +\, \frac 12LMh^2x_t^\star (v_t^\star + \frac{hL}{2}
	x_t^\star) x_t'^\star .
\end{eqnarray*}
Collecting all the bounds derived above, we eventually obtain
\begin{eqnarray*}
	\norm{H_t'-H_0'} &\leq&  th^2 \big(\frac{L_3}{4} (v_t^\star)^3
	x_t'^{,\star} + \frac{L_3 L}{8} h (v_t^\star)^2 x_t^\star x_t'^\star
	\\
	&&\qquad+ Q_1(v_t^\star,x_t^\star) v_t'^{,\star} + x_t
	Q_2(v_t^\star,x_t^\star) x_t'^{,\star}\big)
\end{eqnarray*}
for $t\in h\Z$, where $Q_1$ and $Q_2$ are explicit quadratic forms.
We can conclude that
\begin{eqnarray*}
	\norm{H_t'-H_0'} &\leq& C_2 t (1+t) h^2 \max(\norm{x_0},\norm{v_0})^3
	\max(\norm{x_0'},\norm{v_0'}).
\end{eqnarray*}
\end{proof}

\begin{proof}[Proof of Lemma \ref{lem:X}]
Let $p=0$ or $p\ge 1$. Then by definition of $\eta$,
\begin{eqnarray*}
\lefteqn{E\left[ |e\cdot\xi |^p;\, \xi -\eta \neq -\gamma z\right] }\\
&\le & E\left[ |e\cdot\xi |^p;\, \tilde{U}>{\varphi_{0,1}(e\cdot\xi + \gamma
		\norm{z})}/{\varphi_{0,1}(e\cdot\xi)}\right]\\
&=& \int_{-\infty}^\infty |t|^p\left( \varphi_{0,1}(t)-\varphi_{0,1}(t+|\gamma z|)\right)^+\, dt\\
&=& \int_{-|\gamma z|/2}^\infty |t|^p\left( \varphi_{0,1}(t)-\varphi_{0,1}(t+|\gamma z|)\right)^+\, dt\\
&& = \int_{-|\gamma {z}|/2}^\infty \norm{t}^p \phi_{0,1}(t)\, dt -
	\int_{\gamma \norm{z}/2}^\infty \norm{t- \norm{\gamma z}}^p \varphi_{0,1}(t)\, dt\\
	&& = \int_{-\norm{\gamma z}/2}^{\norm{\gamma z}/2} \norm{t}^p \varphi_{0,1}(t)\,
	dt + \int_{\norm{\gamma z}/2}^\infty \left(\norm{t}^p-\norm{t-
	\norm{\gamma z}}^{p}\right) \varphi_{0,1}(t)\, dt .
\end{eqnarray*}	
For $p=0$, we directly obtain \eqref{X1}, and for $p\ge 1$,
\begin{eqnarray*}
{E\left[ |e\cdot\xi |^p;\, \xi -\eta \neq -\gamma z\right] }
&\le &\frac{1}{\sqrt{2 \pi}} 2^{-p} \norm{\gamma z}^{p+1} + \int_{
	\norm{\gamma z}/2}^\infty  p \norm{t}^{p-1}\norm{ \gamma z} \varphi_{0,1}(t) \, dt\\
	&\leq & C_p \norm{\gamma z} \max(\norm{\gamma z},1)^p
	\end{eqnarray*}
	where $C_p=\max (2^{-p}/\sqrt{2\pi },2pm_{p-1})$ with $m_p$ denoting the $p$ th moment of the standard normal distribution. Finally, since $\xi\sim \mathcal N(0,I_d)$ and the event $\{ \xi -\eta \neq -\gamma z\}$ is measurable w.r.t.\
	$\sigma (e\cdot\xi )$, we obtain 
	\begin{eqnarray*}
	\lefteqn{E\left[ |\xi|^{2p};\, \xi -\eta \neq-\gamma z\right]}\\
	&\le & 2^{p-1}E\left[ |e\cdot \xi|^{2p};\, \xi -\eta \neq-\gamma z\right] +2^{p-1} E\left[ |\xi-(e\cdot\xi) e|^{2p}\right]\,P[ \xi -\eta \neq-\gamma z]\\
	&\le & 2^{p-1}C_{2p}|\gamma z|\, \max (1,|\gamma z|)^{2p}\,+\, 
	2^{p-1}(d-1)^{p}m_{2p} |\gamma z|/\sqrt{2\pi}.
	\end{eqnarray*}
%\frac{1}{\sqrt{2\pi}} 2^{-p} \norm{\gamma z}^{p+1} + p \norm{\gamma z}
%	\int_0^\infty \norm{t}^{p-1} \phi_{0,1}(t)\, dt \\
%&& \leq \norm{\gamma z} (p c_{p-1} + \frac{1}{\sqrt{2\pi}}2^{-p} \norm{\gamma
%z}^p)\\
%&& \leq 
\end{proof}

\begin{proof}[Proof of Theorem \ref{AR}]
 Recall that
 \begin{eqnarray}
 	A(x) &=& \left\{ U  \leq  \exp\left(-H(\phi_T(x,\xi)) + H(x,\xi)
 	\right)\right\},\qquad\mbox{and}\\
 	\hat{A}(x) &=& \left\{ U \leq  \exp\left(-H(\phi_T(y,\eta)) + H(y,\eta)
 	\right)\right\}.
 \end{eqnarray}
Therefore, and since $U$ is independent of $\xi$, we obtain by
Lemma \ref{D},
\begin{eqnarray*}
	P[A(x)^C | \xi] &=& \norm{1- \exp\left(-(H(\phi_T(x,\xi)) -  H(x,\xi))^+
 	\right)} \\
 	&\leq& (H(\phi_T(x,\xi)) -  H(x,\xi))^+ \\
 &\leq& C_1 T(1+T)h^2 \max(\norm{x},\norm{\xi})^3,\qquad\mbox{and hence}\\
	P[A(x)^C] &\leq& C_1 T (1+T) h^2 E[\max(\norm{x}^3,\norm{\xi}^3)] \\
	&\leq& C_1 T(1+T) h^2 (|x|^3 + d^{3/2} \sqrt{8/\pi}).
\end{eqnarray*}
Since $\hat A(y)$ is defined similarly to $A(x)$ with $x,\xi$ replaced by $y,\eta$ and $\eta\sim\xi$, 
we obtain corresponding bounds for  $P[\hat{A}(y)^C|\eta ]$ and
$
P[\hat{A}(y)^C]$.\medskip

Next, we derive the corresponding bounds for the probabilities that the proposed move is accepted for one of the two components and rejected for the other.
By independence of $U$ from $\xi$ and $\eta$ and by Lemma \ref{D}, we have
\begin{align*}
	& P[A(x)\Delta \hat{A}(y)\, |\, \xi,\eta] \\
	& =\ \norm{\exp\left(-(H(\phi_T(x,\xi)) -  H(x,\xi))^+
 	\right)-\exp\left(-(H(\phi_T(y,\eta)) -  H(y,\eta))^+
 	\right)}\\
 	&\leq\ \norm{[H(\phi_T(x,\xi)) -  H(x,\xi)] - [H(\phi_T(y,\eta)) - 
 	H(y,\eta)]}\\
 	&\leq\ \int_0^1\left|\partial_{(x-y,\xi -\eta )}(H\circ\phi_T)(x_u,\xi_u)-\partial_{(x-y,\xi -\eta )}H(x_u,\xi_u)\right|\, du\\
 	&\leq\ C_2 T(1+T) h^2 \int_0^1 \max(\norm{x_u},\norm{\xi_u})^3 \, du \,
 	\max(\norm{x-y},\norm{\xi-\eta}) \\
 	&\leq\ C_2 T(1+T) h^2 \max(\norm{x-y},\norm{\xi -\eta})\,
 	\max(\norm{x},\norm{y},\norm{\xi},\norm{\eta})^3,
\end{align*}
where $x_u=u x + (1-u) y$, $\xi_u = u\xi + (1-u) \eta$, $z=x-y$ and
$W=\xi-\eta$. This proves \eqref{AR7}, and \eqref{AR5} can be shown similarly 
with $\eta$ replaced by $\xi$.\medskip

Next, we bound the unconditioned probabilities of acceptance rejection events.
At first we observe that by \eqref{AR5} and since $\xi\sim\mathcal N(0,I_d)$,
\begin{eqnarray*}
	P[A(x)\Delta A(y)] &\leq& C_2 T(1+T) h^2 \norm{x-y}
	E[\max(\norm{x},\norm{y},\norm{\xi})^3]
	\\ &\leq& C_2 T(1+T)h^2 \norm{x-y} \left(\max(|x|,|y|)^3 + d^{3/2}\sqrt{8/\pi}\right), 
\end{eqnarray*}
which implies \eqref{AR6}.
The proof of a corresponding bound for 
%$P[A(x)\Delta \hat{A}(y)]$ and for 
the expectation in \eqref{AR9} is slightly more complicated. We first note that by \eqref{AR7},
\begin{eqnarray}\nonumber
\lefteqn{ E[\max (|x|,|y|,|\xi |,|\eta |);\, (A(x)\Delta \hat{A}(y))\cap \{W=-\gamma z\}] }\\
\label{AR11}	& \leq &C_2 T(1+T) h^2 \max(1,\gamma) \norm{z}
	E[\max(\norm{x},\norm{y},\norm{\xi},\norm{\xi+\gamma z})^4]
	\\ & \leq &C_2T(1+T)h^2\max(1,\gamma)\norm{z} \left(\max (|x|,|y|)^4 +
	E[(|\xi |+\norm{\gamma z})^4]\right).\nonumber
\end{eqnarray}
%
%
%\begin{remark}
%	We choose $\gamma$ s.t.\ $\gamma R_0\leq \sqrt{2\pi}/3$. Then for
%	$\norm{z}\leq r_0$, we also ahve $\gamma\norm{z}\leq \sqrt{2\pi}(3)$, and hence
%	$\max(\gamma^3 \norm{z}^3,d^{3/2}\sqrt{8/\pi}))d^{3/2}\sqrt{8/\pi}$. THe factor
%	$8$ could be improved in this case.
%\end{remark}
Secondly, on $\{ W \not= - \gamma z\}$, we
	have $	\eta=\xi - 2(e\cdot \xi ) e$ where $e=z/\norm{z}$. In particular, $\eta =\xi$.
%	, and} \\
%		\tilde{U} &>& \frac{\phi_{0,1}(e\cdot\xi + \gamma
%		\norm{z})}{\phi_{0,1}(e\cdot\xi)}.
%	\end{eqnarray*}
Therefore, by \eqref{AR7},
	\begin{eqnarray}\nonumber
	\lefteqn{ E[\max (|x|,|y|,|\xi |,|\eta |);\, (A(x)\Delta \hat{A}(y))\cap \{W\neq -\gamma z\}] }\\
\label{AR12}
		& \leq & C_2 T(1+T)h^2\,  E\left[\max(\norm{z},2\norm{e\cdot
		\xi})\max(\norm{x},\norm{y},\norm{\xi})^4;
		W\neq -\gamma z \right]	\\
	& \leq & C_2 T(1+T)h^2\,  E\left[ \left(\norm{z}+2\norm{e\cdot
		\xi}\right)\left(\max(\norm{x},\norm{y})^4+\norm{\xi}^4\right);
		W\neq -\gamma z \right] .\nonumber
		\end{eqnarray}
By \eqref{AR11}, \eqref{AR12}, and by the bounds in Lemma \ref{lem:X}, we can conclude that there is a finite constant $C_3$ depending only on $L$, $M$ and $N$ such that  for $|\gamma z|\le 1$, 
\begin{eqnarray}\nonumber
\lefteqn{ E[\max (|x|,|y|,|\xi |,|\eta |);\, A(x)\Delta \hat{A}(y)] }\\
	& \leq &C_3 T(1+T) h^2 \max(1,\gamma) \norm{z}
	\,  \left(\max (|x|,|y|)^4 +
	d^2\right).\nonumber
\end{eqnarray}
This proves the last assertion of the theorem.
\end{proof}

\section{Proofs of main results}\label{sec:proofsmain}

\begin{proof}[Proof of Theorem \ref{SCA}]
We fix $x,y\in\mathbb R^d$ such that $|x-y|\ge 2\mathcal R$ and $\max (|x|,|y|)\le R_2$. Since synchronous coupling is applied for $|x-y|\ge 2\mathcal R$, we have $\eta =\xi$. Hence by \eqref{10} and Lemma \ref{C} with $h=0$, we obtain
$$R'(x,y)\ =\ |q_T(x,\xi )-q_T(y,\xi )|\ \le\ (1-KT^2/2)\, r(x,y)$$
provided $LT^2\le K/L$.
\end{proof}

\begin{proof}[Proof of Theorem \ref{SCB}]
We directly give the more involved proof for the case of adjusted numerical HMC. For
unadjusted numerical HMC, the proof simplifies considerably, see the comments at the end 
of the proof.

Without loss of generality, we may assume that $R_2$ is chosen sufficiently large such that
 \begin{eqnarray}\label{XX}
	P[\norm{\xi}>R_2] &\leq& \frac{K}{70L}.
 \end{eqnarray}
 We fix $x,y\in\mathbb R^d$ such that $|x-y|\ge 2\mathcal R$ and $\max (|x|,|y|)\le R_2$. Since synchronous coupling is applied for $|x-y|\ge 2\mathcal R$, we have $\eta =\xi$, and hence
 \begin{align*}
 	R' (x,y)\ &= \  \norm{q_T(x,\xi)-q_T(y,\xi)} &\quad &\text{on } A(x)\cap A(y),
 \\	R'(x,y) \ &= \  r(x,y) &\quad &\text{on } A(x)^C\cap A(y)^C.
 \end{align*} 
 Moreover, on $A(x)\cap A(y)^C$ we have $Y'=y$, and thus
 \begin{eqnarray*}
 		R'(x,y)-r(x,y) &=& \norm{q_T(x,\xi)-y}-\norm{x-y} \ \leq \ \norm{q_t(x,\xi)-x}.
 \end{eqnarray*}
 Similarly, on $A(x)^C\cap A(y),$
 \begin{eqnarray*}
 	R'-r &\leq& \norm{q_T(y,\xi)-y}.
 \end{eqnarray*}
 Therefore, we obtain
 \begin{eqnarray}\nonumber
	E[R'(x,y)-r(x,y)] &\leq& E[\norm{q_T(x,\xi)-q_T(y,\xi)}-\norm{x-y}; A(x)\cap A(y)]
	\\&&+E[\norm{q_T(x,\xi)-x}; A(x)\cap A(y)^C]
	\\&&+E[\norm{q_T(y,\xi)-y}; A(x)^C\cap A(y)]\nonumber
	\\&=:& \rn{1} \ + \ \rn{2} \ + \rn{3}.\nonumber
 \end{eqnarray}
 
In order to control the first term, we choose a constant $C\in (0,\infty )$ as in Lemma \ref{C}, and we assume $h\le\min ( h_1,h_2)$ where $h_2:=\frac{K}{C(1+2R_2)}$.
 Then by Lemma \ref{C},
 \begin{eqnarray*}
 	\norm{q_T(x,\xi)-q_T(y,\xi)} &\leq& (1-\frac{1}{4}KT^2) \norm{x-y}
 	\quad\text{if } \norm{\xi}\leq R_2.
 \end{eqnarray*}
 Therefore, by \eqref{17}, and since $K\le L$,
 \begin{eqnarray*}
 	\rn{1} &\leq& - \frac{1}{4} K T^2 r(x,y) P[A(x)\cap A(y)\cap\{\norm{\xi}\leq R_2\}]
 	+ (LT^2+LTh)r(x,y) P[\norm{\xi}>R_2]\\
 	&\leq& - \frac{1}{4} K T^2 r(x,y) P[A(x)\cap A(y)] \ + \ (\frac{5}{4} L T^2 +
 	L T h) r (x,y)P[\norm{\xi}> R_2] \\
 	&\leq& - \frac{1}{4} KT^2 r(x,y) P[A(x)\cap A(y)] \ + \ \frac 94 LT^2r(x,y) P[\norm{\xi}>R_2]
 \end{eqnarray*}
 for $T\in h\cdot\n$. For $h\le h_1$  we have 
 \begin{equation}
 \label{Lbound}L(T^2+Th)\ \leq\ K/L\ \le \ 1.
 \end{equation}
Therefore, by Theorem \ref{AR}, for $\norm{x},\norm{y}\leq R_2$,
 \begin{eqnarray}
 	P[A(x)^C] + P[A(y)^C] &\leq & 2 C_1 T (1+T) (R_2^3 + 2 d^{3/2}) h^2.
 \end{eqnarray}
 We choose $h_3>0$ such that for $h\le h_3$, the expression on the r.h.s.\ is smaller than
 $1/5$. Because of \eqref{LTT},
 this can be achieved with $h_3^{-2}$ of order $O(R_2^3 + d^{3/2})$.
For $h\le h_3$, we obtain
 \begin{eqnarray*}
 	P[A(x)\cap A(y)] &\geq& 1-\frac{1}{5} \ = \ \frac{4}{5}.
 \end{eqnarray*}
 Therefore, and by \eqref{XX},
 \begin{eqnarray*}
 	\rn{1} &\leq& -\frac{1}{5} K T^2 r + \frac{9}{4} L T^2 r P[\norm{\xi}>R_2] \leq
 	-\frac{1}{6} KT^2r.
 \end{eqnarray*}
\rn{1}n order to control $\rn{2}$, we note that by Lemma \ref{0A}, 
\begin{eqnarray*}
	\norm{q_T(x,\xi)-x} &\leq& T \norm{\xi} \ + \ \max(\norm{x},\norm{x+T\xi}) \
	\leq \ \norm{x} + 2 T \norm{\xi}
\end{eqnarray*}
provided $L(T^2+Th)\leq 1$. Hence in this case we obtain
\begin{eqnarray*}
	\rn{2} &\leq& E[\norm{x}+2T \norm{\xi}; A(x)\cap A(y)^C].
\end{eqnarray*}
A corresponding bound with $x$ and $y$ interchanged holds for $\rn{3}$. Hence by 
the bound \eqref{AR5} for the conditional AR probability given $\xi$, 
\begin{eqnarray*}
 \rn{2} + \rn{3} &\leq& E[\max(\norm{x},\norm{y}) +  2 T \norm{\xi}; A(x)\Delta
 A(y)] \\
 &\leq& C_2 T (1+T) h^2 \norm{x-y} (1+2T) E[\max(\norm{x},\norm{y},\norm{\xi})^4]
 \\
 &\leq& 2 C_2 T(1+T)^2 h^2 (R_2^4 + 3 d^2) r.
\end{eqnarray*}
We choose a strictly positive constant $h_4$ such that for $h\le h_4$, the right hand side is smaller
than $\frac{1}{24} KT^2r$. By  \eqref{LTT}, this can be achieved with $h_4^{-1}$ of order
$O((R_2^2+d)K^{-1/2} T^{-1/2})$. Let $h_0=\min(h_2,h_3,h_4)$. Then for
$h\le \min (h_0,h_1)$, we obtain
$$
	\rn{1} + \rn{2} + \rn{3}\ \leq\  - \frac{1}{5} K T^2 r + \frac{1}{24} K T^2 r \
\leq\  - \frac{1}{8} K T^2 r. 
$$
This completes the proof for adjusted numerical HMC.\smallskip

For unadjusted numerical HMC, the argument simplifies since the rejection events $A(x)^C$ and $A(y)^C$ are empty. Thus it suffices to bound $\rn{1}$, which can be done similarly as 
above for $h\le\min (h_0,h_1)$ where $h_0:=h_2$.
\end{proof}

Next, we directly prove our main result for numerical HMC. Afterwards, we will give
the corresponding proof of Theorem \ref{thm:contrmainexact} for exact HMC, which essentially is an 
easy special case of the more difficult proof for numerical HMC.

\begin{proof}[Proof of Theorem \ref{thm:maincontraction}]
As above, we directly prove the result for adjusted numerical HMC, and we comment on the
simplifications in the unadjusted case in the end.
The parameters $\gamma$, $a$ and $R_1$ have been chosen in \eqref{Cgamma}, \eqref{Ca} and \eqref{CR1} such that the following conditions are satisfied:
\begin{eqnarray}
\label{eq:A}\gamma T &\le &1,\\
\label{eq:B} L(T+h) &\le &\gamma /4,\\
\label{eq:C}\gamma \mathcal R &\le &1/4,\\
\label{eq:D} a T &\ge &1,\\
\label{eq:E}R_1&\ge &\frac 52 \, (1+\gamma T)\mathcal R,\\
\label{eq:F}\exp (a(R_1-2\mathcal R)) &\ge &20.
\end{eqnarray}
Indeed, \eqref{eq:A} and \eqref{eq:C} hold by \eqref{Cgamma}, \eqref{eq:B} holds by \eqref{Cgamma} and \eqref{A0}, 
\eqref{eq:D} holds by \eqref{Ca}, \eqref{eq:E} holds by \eqref{CR1} and \eqref{Cgamma}, and \eqref{eq:F} holds by \eqref{CR1} and \eqref{Ca}. The bounds \eqref{eq:A}-\eqref{eq:F} will be essential in the following arguments.
We have chosen $\gamma$ and $a$ as large resp.\ small as possible such that
\eqref{eq:A}, \eqref{eq:C} and \eqref{eq:D} hold. Then \eqref{eq:B} implies the
additional constraints on $T$ in \eqref{A0}, and $R_1$ is chosen such that
\eqref{eq:E} and \eqref{eq:F} are satisfied. \medskip

To prove contractivity, we fix $x,y\in\mathbb R^d$ such that $\max (|x|,|y|)\le R_2$.
Since $x$ and $y$ are fixed, we briefly write $r$ and $R'$ instead of $r(x,y)$ and
$R'(x,y)$. We consider separately the cases where $|x-y|\ge 2\mathcal R$ and
$|x-y|<2\mathcal R$.\medskip

{\em (i) Contractivity for $|x-y|\ge 2\mathcal R$.} \ For $|x-y|\ge 2\mathcal R$, we can apply the result of Theorem \ref{SCB}. Indeed, choose $h_0$ as in Theorem
 \ref{SCB}. Then for $h\le h_0$, by concavity of $f$ and by \eqref{SC2},
 \begin{equation}
 E[f(R')-f(r)]\, \le\, f'(r)\, E[R'-r]\,
 \label{P1}\le \,-\frac 14 KT^2rf'(r)\, \le\, -c_1f(r)
 \end{equation}
where the lower bound $c_1$ for the contraction rate is given by
\begin{equation}
\label{P2}c_1\ =\ \frac 14 KT^2 \inf_{r>0}\frac{rf'(r)}{f(r)}.
\end{equation}
Recall that $f$ is concave with $f(0)=0$,  and $f$ is linear for $r\ge R_1$. Hence the function $r \mapsto rf'(r)/f(r)$ attains its minimum at $R_1$, where
$f'(R_1)=e^{-aR_1}$ and $f(R_1)=\int_0^{R_1}e^{-as}ds\le\min (R_1,a^{-1})$. Therefore, by \eqref{Ca} and \eqref{CR1},
\begin{eqnarray}
\nonumber c_1 &\ge &\frac 14KT^2\max (1,aR_1)\, e^{-aR_1}\ \ge \ \frac 14KT^2\, \frac 52 (1+\frac{\mathcal R}{T})\, e^{-5/2}e^{-\frac{5\mathcal R}{2T}}\\
\label{P2a}
&>&\ \frac{1}{20}KT^2\,  (1+\frac{\mathcal R}{T})\, e^{-\frac{5\mathcal R}{2T}}.
\end{eqnarray}

{\em (ii) Contractivity for $|x-y|< 2\mathcal R$.} \ For $|x-y|< 2\mathcal R$,
we apply the coupling defined by \eqref{**} and \eqref{hatA}. Let $z=x-y$ and $W=\xi -\eta$. Since $R'=r$ on $A(x)^C\cap\hat A(y)^C$, we have
$$E[f(R')-f(r)]\ =\ \rn{1}+\rn{2}+\rn{3}+\rn{4},\qquad\mbox{where}$$
\begin{eqnarray*}
\rn{1} &=& E\left[ f(R')-f(r);\, A(x)\cap\hat A(y)\cap\{ W=-\gamma z\} \right] ,\\
\rn{2} &=& E\left[ f(R'\wedge R_1)-f(r);\, A(x)\cap\hat A(y)\cap\{ W\neq -\gamma z\} \right] ,\\
\rn{3} &=& E\left[ f(R')-f(R'\wedge R_1);\, A(x)\cap\hat A(y)\cap\{ W\neq -\gamma z\} \right] ,\\
\rn{4} &=& E\left[ f(R')-f(r);\, A(x)\triangle\hat A(y) \right] .
\end{eqnarray*}
Only the first term is responsible for contractivity. The other terms are perturbations that have to be controlled. We will now derive upper bounds 
for each of the four terms. We remark at first that on $A(x)\cap\hat A(y)$,
\begin{equation}
\label{P3}R'\ =\ |q_T(x,\xi )-q_T(y,\eta )|\ \le\ |z+WT|+\max(|z|,|z+WT|)\, LT(T+h)
\end{equation}
by Lemma \ref{A}.\medskip

$\rn{1}$. On $A(x)\cap\hat A(y)\cap\{ W=-\gamma z\}$, we obtain by \eqref{P3},
\eqref{eq:A} and \eqref{eq:B},
\begin{eqnarray*}
R' &\le & |(1-\gamma T)z|\, +\, \max(|z|,|(1-\gamma T)z|)\, LT(T+h)\\
&\le & (1-\gamma T+\frac 14\gamma T)|z|\ =\ (1-\frac 34\gamma T)r.
\end{eqnarray*}
Therefore, by concavity of $f$,
\begin{eqnarray*}
\rn{1} &\le &f'(r)\,  E\left[ R'-r;\, A(x)\cap\hat A(y)\cap\{ W=-\gamma z\} \right] \\
&\le &-\frac 34\gamma Trf'(r)\,\left( 1- P[W\neq -\gamma z]-P[A(x)^C]-P[\hat A(y)^C]\right).
\end{eqnarray*}
By Lemma \ref{lem:X} and by \eqref{eq:C},
$$P[W\neq -\gamma z]\ \le\ \frac{\gamma r}{\sqrt{2\pi}}\ \le\ \frac{1}{4\sqrt{2\pi}}\
<\ \frac{1}{10}.$$
Furthermore, by Theorem \ref{AR}, there is a finite constant $C_1$ depending only on $L$, $M$ and $N$ such that
$$P[A(x)^C]+P[\hat A(y)^C]\ \le\ C_1T(1+T)\, (|x|^3+|y|^3+4d^{3/2})\, h^2.$$
Since $\max (|x|,|y|)\le R_2$ and \eqref{A0} holds, we can conclude that there is
a constant $h_5>0$ depending only on $L$, $M$, $N$, $d$ and $R_2$ such that
for $h\le h_5$,
\begin{equation}
\label{P4}\rn{1}\ \le\ -\frac{27}{40}\gamma Trf'(r).
\end{equation}
Furthermore, for fixed $L$, $M$ and $N$, the constant $h_5$ can be chosen 
by \eqref{A0} such that $h_5^{-2}$ is of order $O(R_2^{3}+d^{3/2})$.\medskip

$\rn{2}$. By definition of $f$, we have for $s\le R_1$,
$$f(s)-f(r)\ =\ \int_r^se^{-at}\, dt\ \le\ \frac 1a e^{-ar}\ =\ \frac 1a f'(r).$$
Therefore, by  \eqref{eq:D} and by Lemma \ref{lem:X}, the second term can be bounded by
\begin{equation}
\label{P5} \rn{2}\ \le\ \frac 1af'(r)\, P[W\neq -\gamma z]\ \le \ \frac{\gamma T}{\sqrt{2\pi}}rf'(r)\ <\ \frac 25\gamma Trf'(r).
\end{equation}

$\rn{3}$. If $W\neq -\gamma z$ then by definition of the coupling,
$$W\ =\ \xi -\eta\ =\ 2 (e\cdot\xi )e\qquad\mbox{where }e=z/|z|,$$
and hence $|z+WT|= |r+2e\cdot\xi T|$. Therefore on $A(x)\cap\hat A(y)\cap\{ W\neq -\gamma z\}$,
$$R'\ \le\ (1+LT (T+h))\, |r+2e\cdot\xi T|\ \le\ \frac 54 |r+2e\cdot\xi T|$$
by \eqref{P3}, \eqref{eq:B} and \eqref{eq:A}. Thus
\begin{eqnarray}
\nonumber\lefteqn{E\left[ (R'- R_1)^+;\, A(x)\cap\hat A(y)\cap\{ W\neq -\gamma z\} \right]}\\
\nonumber &\le &E\left[ (\frac 54 |r+2e\cdot\xi T|- R_1)^+;\,  W\neq -\gamma z \right]\\
\nonumber &= &\int_{-\infty}^\infty (\frac 54 |r+2u T|- R_1)^+\,  (\varphi_{0,1}(u)-\varphi_{0,1}(u+\gamma r))^+\, du\\
\nonumber &= &\int_{-\gamma r/2}^\infty (\frac 54 |r+2u T|- R_1)^+\,  (\varphi_{0,1}(u)-\varphi_{0,1}(u+\gamma r))\, du\\
%\label{P6}&= &\int_{-\gamma r/2}^\infty (\frac 54 |r+2u T|- R_1)^+\,  \varphi_{0,1}(u)\, du\, -\, \int_{\gamma r/2}^\infty (\frac 54 |r+2(u-\gamma r) T|- R_1)^+\,  \varphi_{0,1}(u)\, du\\
\nonumber &= &\int_{\gamma r/2}^\infty \left\{ (\frac 54 |r+2u T|- R_1)^+ - (\frac 54 |r+2(u-\gamma r) T|- R_1)^+\right\}\,  \varphi_{0,1}(u)\, du
\\
\nonumber &\le  &\frac 52 \gamma rT\int_{\gamma r/2}^\infty  \varphi_{0,1}(u)\, du
\ \le\ \frac 54\gamma rT.
\end{eqnarray}
Here we have used that by \eqref{eta},
$$P[W\neq -\gamma z\,|\,\xi ]\ =\ \left( \varphi_{0,1}(e\cdot\xi )-\varphi_{0,1}(e\cdot\xi +\gamma r)\right)^+/\varphi_{0,1}(e\cdot\xi ).$$
Moreover, we have used that by \eqref{eq:E}, $R_1\ge\frac 54(1+\gamma T)r$.
By concavity of $f$ and by \eqref{eq:F}, we obtain
\begin{eqnarray}
\nonumber \rn{3} &\le & f'(R_1)\, E\left[ (R'- R_1)^+;\, A(x)\cap\hat A(y)\cap\{ W\neq -\gamma z\} \right]\\
\label{P7} &\le &\frac 54\gamma Trf'(R_1)\ \le\ \frac 54 e^{-a(R_1-2\mathcal R)}
\gamma Trf'(r)\ \le\ \frac{1}{16}\gamma Trf'(r).
\end{eqnarray}

\rn{4}. By a similar argument as in the proof of Theorem \ref{SCB}, we obtain
\begin{eqnarray*}
\lefteqn{E\left[ R'-r;\, A(x)\triangle\hat A(y) \right]}\\
&\le & E\left[|q_T(x,\xi )-x|;\, A(x)\cap\hat A(y)^C \right]\, +\,E\left[ |q_T(y,\eta )-y|;\, A(x)^C\cap\hat A(y) \right]\\
&\le & E\left[|x|+2T|\xi |;\, A(x)\cap\hat A(y)^C \right]\, +\,E\left[ |y|+2T|\eta |;\, A(x)^C\cap\hat A(y) \right]\\
&\le & (1+2T)\, E\left[\max (|x|,|y|,|\xi |,|\eta |);\, A(x)\triangle\hat A(y) \right]\\
&\le &2C_3\max (1,\gamma ) T(1+T)^2 
%(1+2\mathcal R)
(R_2^4+ d^{2})  h^2r.
\end{eqnarray*} 
Here we have used \eqref{AR9} in the last step. By concavity of $f$ we obtain
\begin{eqnarray}
\nonumber \rn{4} &\le & f'(r)\, E\left[ R'-r;\, A(x)\triangle\hat A(y) \right] \\
\label{P8} &\le & 2C_3\max (1,\gamma ) T(1+T)^2 
(R_2^4+d^{2})  h^2rf'(r)\\
\nonumber &\le &\frac{1}{80}\gamma Trf'(r)
\end{eqnarray}
for $h\le h_6$ where $h_6$ is a positive constant depending only on $L$, $M$, $N$, $\mathcal R$, $R_2$ and $d$ that by \eqref{A0} can be chosen 
such that $h_6^{-2}$ is of order $O((1+\mathcal R)(R_2^{4}+d^{2}))$.\medskip

Combining the bounds for the terms $\rn 1$, $\rn 2$, $\rn 3$ and $\rn 4$ in
\eqref{P4}, \eqref{P5}, \eqref{P7} and \eqref{P8}, we obtain for $h\le \min (h_5,h_6)$,
\begin{eqnarray}
\nonumber E[f(R')-f(r)] &\le & \left( -\frac{27}{40}+\frac{2}{5}+\frac{1}{16}+\frac{1}{80}\right)\, \gamma Trf'(r)\\
&=&-\frac 15\,\gamma Trf'(r)\ \le\ -c_2\, f(r)\label{P9}
\end{eqnarray}
where the contraction rate $c_2$ satisfies
\begin{eqnarray}
\nonumber c_2&=& \frac 15 \gamma T\, \inf_{r\le 2\mathcal R}\frac{rf'(r)}{f(r)}\
\ge\ \frac 15\gamma T\, \max (1,2a\mathcal R)\, e^{-2a\mathcal R}\\
\label{P9a}&=& \frac 15\, \min\left( 1,\frac{T}{4\mathcal R}\right)\, \max \left( 1,\frac{2\mathcal R}{T}\right)\, e^{-2\mathcal R/T}\ \ge\ \frac{1}{10}\, e^{-2\mathcal R/T}.
\end{eqnarray}

{\em (iii) Global contraction.} \ Let $h_\star :=\min (h_0,h_5,h_6)$. Then by combining the bounds in \eqref{P2} and \eqref{P9}, we see that for $h\le \min (h_1,h_\star)$ and for any $x,y\in\mathbb R^d$ with $\max (|x|,|y|)\le R_2$,
$$E[f(R')]\ \le \ (1-c)\, f(r)$$
where $c:=\min (c_1,c_2)$. Moreover, by \eqref{P2a} and \eqref{P9a},
$$c\ \ge\ \frac{1}{10}\, e^{-2\mathcal R/T}\, \min\left( 1,\, \frac 12KT^2(1+\mathcal R/T)e^{-\mathcal R/(2T)}\right) .$$
This completes the proof for adjusted numerical HMC.\smallskip

For unadjusted numerical HMC, the proof simplifies because the rejection events $A(x)^C$ 
and $\hat A(y)^C$ are empty. Thus it suffices to control $\rn{1}$, $\rn{2}$ and $\rn{3}$,
and this can be done similarly as above whenever $h\le\min ( h_1,h_\star)$, where now we can choose $h_\star :=h_0$.
\end{proof}

\begin{proof}[Proof of Theorem \ref{thm:contrmainexact}]
The contraction bound for exact HMC can be derived similarly to the proof of Theorem \ref{thm:maincontraction}. In this case, instead of Theorem
\ref{SCB}, we apply Theorem \ref{SCA} in Step (i). Furthermore, the rejection events $A(x)^C$ and $\hat A(y)^C$, $x,y\in\mathbb R^d$, are empty for exact HMC. Therefore,
the corresponding terms do not have to be taken into account in Step (ii). Consequently, the resulting bound \eqref{contrmain} is valid for all $x,y\in\mathbb R^d$ with the same rate $c$ as above.
\end{proof}

Next, we prove the contraction result under Assumption \ref{A1245}. Again, we directly give the proof for adjusted numerical HMC, and we comment on the simplifications in the other cases afterwards.

\begin{proof}[Proof of Theorem \ref{thm:lyapcontraction}]

We fix $x,y\in\mathbb R^d$ such that $\max (|x|,|y|)\le R_2$ and we set
\begin{eqnarray*}
F(x,y)&=& f\left( \min (|x-y|,2\mathcal R)\right),\qquad F'(x,y)\, =\, F(X'(x,y),Y'(x,y)),\\
G(x,y)& =& 1+\epsilon\Psi (x)+\epsilon\Psi (y),
\qquad G'(x,y)\, =\, G(X'(x,y),Y'(x,y)).
\end{eqnarray*}
Since $x$ and $y$ are fixed, we 
omit the dependences on $x$ and $y$ in the notation. Thus $\rho_\epsilon =\sqrt{FG}$
and $\rho_\epsilon (X',Y') =\sqrt{F'G'}$. 
We consider separately the cases where $|x-y|< 2\mathcal R$ and
$|x-y|\ge 2\mathcal R$.\medskip

{\em (i) Contractivity for $|x-y|\le 2\mathcal R$.}
In this case, the same arguments as in the proof of Theorem \ref{thm:maincontraction} show
that for $h\le \min (h_5,h_6)$,
$$E[F']\ \le\ E[f(R')]\ \le\ (1-c_2)f(r)\ =\ (1-c_2)F $$
for a constant $c_2\ge\frac{1}{10}e^{-2\mathcal R/T}$. Furthermore, (A4) implies
\begin{eqnarray*}
E[G']& =& 1+\epsilon E[\Psi (X')+\Psi (Y')]\ =\ 1+\epsilon (\pi\Psi )(x)+\epsilon (\pi\Psi )(y)\\
&\le &1+\epsilon \Psi (x)+\epsilon \Psi (y)+2\epsilon A\ \le\ (1+2\epsilon A) G.
\end{eqnarray*}
Thus by the choice of $\epsilon$ in \eqref{Cepsilon},
\begin{eqnarray}
\nonumber E[\rho_\epsilon (X',Y')]& =& E[\sqrt{F'G'}]\ \le\ E[F']^{1/2}E[G']^{1/2}\\
\nonumber &\le & (1-c_2)^{1/2}(1+2\epsilon A)^{1/2}\sqrt{FG}\ \le \ \left(1-c_2/2+\epsilon A\right)
\rho_\epsilon\\ 
\label{coni}& \le & \left(1- c_2/4 \right)
\rho_\epsilon .
\end{eqnarray}

{\em (ii) Contractivity for $|x-y|> 2\mathcal R$.} \ If $|x-y|> 2\mathcal R$ then $|x|>\mathcal R$ or $|y|>\mathcal R$. Hence by \eqref{DefcalR}, $\Psi (x)+\Psi (y)>4A/\lambda$, and thus 
by (A4),
\begin{eqnarray*}
E[G'] &=& 1+\epsilon (\pi\Psi )(x)+\epsilon (\pi\Psi )(y)\\
&\le &1+\epsilon \left( (1-\lambda )(\Psi (x)+\Psi (y))+2A\right)\\
&\le &1+\epsilon \left( (1-\lambda /4 )(\Psi (x)+\Psi (y))-A\right)\\
&\le & \max (1-\lambda /4, 1-\epsilon A)\, G\ =\ (1-c_3)\, G
\end{eqnarray*}
where $c_3:=\min (\lambda /4,\epsilon A)=\min (c_2,\lambda )/4$. Since $F'\le f(2\mathcal R)\le f(|x-y|)=F$, we thus obtain
\begin{equation}
\label{conii}
E[\rho_\epsilon (X',Y')]\ =\ E[\sqrt{F'G'}]\ \le\ \sqrt F\, E[G']^{1/2}\ \le\ (1-c_3)^{1/2}\sqrt{FG}
\ \le \ (1-c_3/2)\rho_\epsilon .
\end{equation}
{\em (iii) Global contraction.} \ Let $h_\star :=\min (h_5,h_6)$. Then by combining the bounds in \eqref{coni} and \eqref{conii}, we see that for $h\le \min (h_1,h_\star)$ and for any $x,y\in\mathbb R^d$ with $\max (|x|,|y|)\le R_2$,
$$E[(\rho_\epsilon (X',Y')]\ \le \ (1-c)\, \rho_\epsilon (x,y)$$
where $c:=\min (c_2/4,c_3/2)=\min (c_2/8,\lambda /4)\ge \min (e^{-2\mathcal R/T},20\lambda )/80$. 
\end{proof}

\begin{proof}[Proof of Theorem \ref{thm:contrlyapexact}]
The proof is similar to Theorem \ref{thm:lyapcontraction}. Since the rejection events are empty, the bound holds for all $h\in [0,h_1]$.
%For unadjusted numerical HMC, the proof simplifies because the rejection events $A(x)^C$ 
%and $\hat A(y)^C$ are empty. Thus it suffices to control $\rn{1}$, $\rn{2}$ and $\rn{3}$,
%and this can be done similarly as above whenever $h\le\min ( h_1,h_\star)$, where now we can choose $h_\star :=h_0$.
%The contraction bound for exact HMC can be derived similarly to the proof of Theorem \ref{thm:maincontraction}. In this case, instead of Theorem
%\ref{SCB}, we apply Theorem \ref{SCA} in Step (i). Furthermore, the rejection events $A(x)^C$ and $\hat A(y)^C$, $x,y\in\mathbb R^d$, are empty for exact HMC. Therefore,
%the corresponding terms do not have to be taken into account in Step (ii). Consequently, the resulting bound \eqref{contrmain} is valid for all $x,y\in\mathbb R^d$ with the same rate $c$ as above.
\end{proof}

\section{Proofs of results in Section \ref{sec:QB}}
\label{sec:LYAP}

All bounds in Section \ref{sec:QB} are based on the following observation:

\begin{lemma}[Locally contractive couplings and supermartingales]\label{lem:supermart} \  \\
Let $\pi (x,dy)$ be a Markov transition kernel on a complete separable metric space $(S,\rho )$. Suppose that there exist a constant $c\in (0,\infty )$, a measurable subset $A\subseteq S$, a probability space $(\Omega ,\mathcal A,P)$, and a measurable map $$(x,y,\omega )\mapsto (X'(x,y)(\omega ),\, Y'(x,y)(\omega ))$$ from $S\times S\times\Omega$ to $S\times S$ such that for any
$x,y\in S$, $(X'(x,y),\, Y'(x,y))$ is a realization of a coupling of $\pi (x,\cdot )$ and $\pi (y,\cdot )$, and
\begin{equation}
\label{contrA}
E\left[\rho (X'(x,y),Y'(x,y))\right]\ \le\ e^{-c}\rho (x,y)\qquad\text{for }x,y\in A.
\end{equation}
Then, for any probability measure $\gamma$ on $S\times S$, there is a 
Markov chain $(X_n,Y_n)_{n\ge 0}$ defined on a probability space
$(\tilde\Omega ,\tilde{\mathcal A},\tilde P)$ such that $(X_0,Y_0)\sim\gamma$,
both marginal processes $(X_n)_{n\ge 0}$ and $(Y_n)_{n\ge 0}$ are Markov
chains on $S$ with transition kernel $\pi$, and such that the process
\begin{equation}
\label{Mn}M_n\, =\, e^{c(n\wedge T)}\rho (X_{n\wedge T},Y_{n\wedge T}),\quad
T\, =\, \min\{ n\ge 0:(X_n,Y_n)\not\in A\times A\} ,
\end{equation}
is a non-negative supermartingale w.r.t.\ the filtration generated by
$(X_n,Y_n)_{n\ge 0}$.
\end{lemma}

\begin{proof}
For $x,y\in S\times S$ let 
$$k((x,y),\ \cdot\ )\ :=\ P\circ (X'(x,y),Y'(x,y))^{-1}$$
denote the joint law of $X'(x,y)$ and $Y'(x,y)$. Then $k$ is a transition kernel
on $S\times S$ with marginals $\pi (x,\cdot )$ and $\pi (y,\cdot )$, and by
\eqref{contrA},
\begin{equation}
\label{krho}
(k\rho )(x,y)\ \le\ e^{-c}\rho (x,y)\qquad\text{for any }x,y\in A. 
\end{equation}
Now let $(X_n,Y_n)_{n\ge 0}$ be a time-homogeneous Markov chain on a 
probability space $(\tilde\Omega ,\tilde{\mathcal A},\tilde P)$ with initial distribution $(X_0,Y_0)\sim\gamma$ and transition kernel $k$, and let
$\mathcal F_n=\sigma ((X_i,Y_i):0\le i\le n)$. Then for any $n\ge 0$,
$$\tilde E[\rho (X_{n+1},Y_{n+1})|\mathcal F_n]\ =\ (k\rho )(X_n,Y_n)\ \le\ 
e^{-c}\rho (X_n,Y_n)$$
holds $\tilde P$-almost surely on $\{ (X_n,Y_n)\in A\times A\} $. Therefore,
the process $(M_n)$ defined by \eqref{Mn} is a non-negative $(\mathcal F_n)$-supermartingale.
\end{proof}

The error bound for exact HMC in Corollary \ref{cor:QBHMC} is a direct consequence of Theorem \ref{thm:contrmainexact} and Lemma \ref{lem:supermart}:

\begin{proof}[Proof of Corollary \ref{cor:QBHMC}]
For exact HMC, by Theorem \ref{thm:contrmainexact}, the local contractivity 
condition \eqref{contrA} in Lemma \ref{lem:supermart} is satisfied for $S=A=\mathbb R^d$, $\rho$ and $c$ given by \eqref{rho} and \eqref{crate}, and the 
coupling $(X'(x,y),Y'(x,y))$ introduced above. Now let $\nu$ and $\eta$ be 
probability measures on $\mathbb R^d$, and let $\gamma$ be an arbitrary
coupling of $\nu$ and $\eta$. Then by Lemma \ref{lem:supermart}, there is a
Markov chain $(X_n,Y_n)_{n\ge 0}$ on a probability space $(\tilde\Omega ,\tilde{\mathcal A},\tilde P)$ such that $(X_0,Y_0)\sim\gamma$, both $(X_n)$ and $(Y_n)$ are Markov chains with transition kernel $\pi$ and initial laws $\nu$ and 
$\eta$, respectively, and $M_n=e^{cn}\rho (X_n,Y_n)$ is a non-negative 
supermartingale. Hence for any $n\in\mathbb N$,
$$\mathcal W_\rho (\nu\pi^n,\mu\pi^n)\ \le\ \tilde E[\rho (X_n,Y_n)]\ \le\ 
e^{-cn}E[\rho (X_0,Y_0)]\ =\ e^{-cn}\int\rho\, d\gamma .$$
Taking the infimum over all couplings $\gamma\in\Pi (\nu ,\eta )$, we see
that \eqref{QBHMC1} holds. Furthermore, by \eqref{rho} and \eqref{f},
\begin{equation}
\label{rhonorm}e^{-aR_1}|x-y|\ \le\ \rho (x,y)\ \le\ |x-y|\qquad\text{for any }x,y\in\mathbb R^d.
\end{equation}
Therefore, \eqref{QBHMC1} implies
$$\mathcal W^1(\nu\pi^n,\eta\pi^n)\ \le\ e^{aR_1}e^{-cn}\mathcal W^1(\nu ,\eta ).$$
Choosing $\eta =\mu$, we have $\eta \pi^n=\mu$ for all $n$. Hence
$$\Delta (n)\ \le\ \exp (aR_1-cn)\, \Delta (0).$$
The second part of the assertion now follows, because by \eqref{Ca} and \eqref{CR1}, $aR_1=\frac 52(1+\mathcal R/T)$.
\end{proof}

Now suppose again that $\pi (x,dy)$ is an arbitrary Markov transition kernel 
on a complete separable metric space $(S,\rho )$. For proving Theorem
\ref{thm:LYAP}, we combine Lemma \ref{lem:supermart} with a Lyapunov bound for exit probabilities:

\begin{proof}[Proof of Theorem \ref{thm:LYAP}]
Let $\nu$ and $\eta$ be probability measures on $S$, and let $\gamma$ be a 
coupling of $\nu$ and $\eta$. By (C3) in Assumption \ref{ALyap}, the conditions 
in Lemma \ref{lem:supermart} are satisfied with $A=\{\psi >C\} $. Hence on some probability space $(\tilde\Omega ,\tilde{\mathcal A},\tilde P)$, there is a
coupling $(X_n,Y_n)_{n\ge 0}$ of Markov chains with transition kernel $\pi$ and joint initial law  $(X_0,Y_0)\sim\gamma$ such that $M_n=e^{c(n\wedge T)}\rho (X_{n\wedge T},Y_{n\wedge T})$ is a non-negative 
supermartingale stopped at
$$T\ =\ \min\left\{ n\ge 0:\,\psi (X_n)>C\text{ or }\, \psi (Y_n)>C\right\} .$$
In particular, we obtain
\begin{eqnarray}
\label{pL1} e^{cn}E[\rho (X_n,Y_n);n\le T] & \le & E\left[ e^{c(n\wedge T)}\rho (X_{n\wedge T},Y_{n\wedge T})\right] \\
\nonumber &\le & E[\rho (X_0,Y_0)]\ =\ \int\rho\, d\gamma 
\end{eqnarray}
for any $n\in\mathbb N$. In order to bound the corresponding expectation on the 
complement $\{ n<T\}$, we observe that by Condition (C2) in Assumption
\ref{ALyap} and by the definition of $\delta (C)$ in \eqref{deltaC},
\begin{eqnarray}
\nonumber\lefteqn{E[\rho (X_n,Y_n);n>T]\ \le\ E\left[ \varphi (X_{n})+\varphi (Y_{n});n>T\right]}\\
\label{pL2} &\le & \beta^n E\left[ \varphi (X_{T})+\varphi (Y_{T});n>T\right]\\ 
\nonumber &\le & \beta^n E\left[ \psi (X_{T})+\psi (Y_{T});n>T\right]\, \delta (C).
\end{eqnarray}
Here we have used that by (C2), for any $k\le n$;
\begin{eqnarray*}
E[\varphi (X_n);T=k] &=& E[(\pi^{n-k}\varphi )(X_k);T=k]\ \le\ \beta^{n-k}E[\varphi (X_k);T=k]\\ &\le &\beta^nE[\varphi (X_T);T=k] ,
\end{eqnarray*}
and a corresponding inequality holds for $E[\varphi (Y_n);T=k]$.\smallskip\\
Furthermore, the Lyapunov condition (C1) in Assumption \ref{ALyap} implies that
the stopped process $N_n=\psi (X_{n\wedge T})/\lambda^{n\wedge T}$ is a 
non-negative supermartingale. Therefore,
$$E[\psi (X_T);n>T]\ \le\ \lambda^{n-1}E[\psi (X_T)/\lambda^T]\ \le\ \lambda^{n-1}E[\psi (X_0)]\ =\ \lambda^{n-1}\int\psi\,d\nu .$$
A corresponding bound holds for $E[\psi (Y_T);n>T]$, and thus
\begin{equation}
\label{pL3}
E[\psi (X_T)+\psi (Y_T);n>T]\ \le\ 
\lambda^{n-1}(\int\psi\, d\nu +\int\psi\, d\eta ).
\end{equation}
By combining the bounds in \eqref{pL1}, \eqref{pL2} and \eqref{pL3}, we obtain
$$E[\rho (X_n,Y_n)]\ \le\ e^{-cn}\int\rho\, d\gamma\, +\, \beta^n\lambda^{n-1}
(\int\psi\, d\nu +\int\psi\, d\eta )\,\delta (C)$$
for any $n\in\mathbb N$. The bound for the Kantorovich distance in 
\eqref{25A} now follows by taking the infimum over all couplings $\gamma\in\Pi (\nu ,\eta )$.
\end{proof}

From now on, we consider numerical HMC. Let $\pi$ denote the transition kernel
for a given step size $h>0$, and let $\rho$ be the metric defined by \eqref{rho},
\eqref{f}, \eqref{Ca} and \eqref{CR1}. To be able to apply Theorem \ref{thm:LYAP}, we first identify appropriate Lyapunov functions.

\begin{lemma}\label{lem:LYAP}
Let $T,h_1\in (0,\infty )$ such that \eqref{A0} holds. Then there exists $C_1\in (0,\infty )$ depending only on $L$, $M$ and $N$ such that for $C\in (1,\infty )$ and 
$h\in (0,h_1)$ with
\begin{equation}
\label{LTh} C_1Th^2\ \le\ \min\left( (\mathcal R+\sqrt{2/K}(\log C)^{3/4})^{-3},\, (1/3)^{3/2}\right) ,
\end{equation}
Conditions (C1) and (C2) in Assumption \ref{ALyap} are satisfied with 
\begin{eqnarray}
\label{phipsi}
\varphi (x)& =& |x|+2Td^{1/2},\qquad\psi (x)\ =\ \exp \left(U(x)^{2/3}\right),\\
\beta\ =\ 2,&& \lambda\ =\ E\left[ \exp(|\xi |^{4/3}+\frac 13|\xi |^2+1)\right]\text{ with } \xi\sim N(0,I_d).\label{betalambda}
\end{eqnarray}
\end{lemma}

%\begin{proof}
%We first remark that by Assumption \ref{A123}, $x\cdot\nabla U(x)\ge K|x|^2$ for $|x|\ge\mathcal R$. Therefore, for any $x\in\mathbb R^d$,
%\begin{equation}
%\label{Unorm}
%U(x)\ \ge\ \frac{K}{2}\min (|x|-\mathcal R,0)^2\quad\text{and}\quad |x|\ \le\ \mathcal R+\sqrt{2U(x)/K}.
%\end{equation}
%Furthermore, by \eqref{rho} and \eqref{f},
%$$\rho (x,y)\ \le\ |x-y|\ \le\ |x|+|y|\le\ \varphi (x)+\varphi (y)\qquad \text{for any }x,y\in\mathbb R^d.$$
%To verify the Lyapunov conditions recall that
%$$(\pi\varphi )(x)\ =\ E[\varphi (q_T(x,\xi ));A(x)]\, +\, \varphi (x)\, P[A(x)^C]\quad\text{with }\xi\sim N(0,I_d).$$
%By Lemma \ref{0A}, $|q_T(x,\xi )|\le 2(|x|+T|\xi |)$,
%and thus for any $x\in \mathbb R^d$,
%$$(\pi\varphi )(x)\ \le\ 2\, E\left[|x|+T|\xi |+2Td^{1/2}\right]\ \le\ 2|x|+4Td^{1/2}\ =\ 2 \varphi (x).$$
%Hence (C2) is satisfied.\smallskip\\
%Furthermore, by Lemma \ref{D}, there is a finite constant $C_1$ such that
%\begin{eqnarray}\label{Ubound}
%U(q_T(x,\xi ))& \le & H(\phi_T(x,\xi ))\ \le\ H(x,\xi )+C_1Th^2\max (|x|,|\xi |)^3\\
%\nonumber &\le & U(x)+\frac 12 |\xi |^2+C_1Th^2|\xi |^3+C_1Th^2|x|^3.
%\end{eqnarray}
%Suppose that $\psi (x)\le C$. Then $U(x)\le (\log C)^{3/2}$, and hence by 
%\eqref{Unorm},\\ $|x|\le \mathcal R+\sqrt{2/K}(\log C)^{3/4}$. Therefore, if
%\eqref{LTh} holds then by \eqref{Ubound}, we obtain
%$$(\pi\psi )(x)\ \le\ E\left[ \exp\left( U(x)^{2/3}+|\xi |^{4/3}+|\xi |^{2}/3+1\right)\right]
%\ =\ \lambda\psi (x)$$
%for any $x\in S$ such that $\psi (x)\le C$. Hence (C1) is satisfied as well.
%\end{proof}

\begin{proof}[Proof of Theorem \ref{thm:QBNHMC}]
Let $C\in [e,\infty )$, i.e., $\log C\ge 1$. Then by Lemma \ref{lem:LYAP},
Conditions (C1) and (C2) in Assumption \ref{ALyap} are satisfied for
$\varphi$, $\psi$, $\beta$ and $\lambda$ given by \eqref{phipsi} and 
\eqref{betalambda}, provided \eqref{LTh} holds. 
This is the case for $h\le h_0$ where $h_0>0$ can be chosen such that
$h_0^{-1}$ is of order $O(\mathcal R^{3/2}+(\log C)^{9/8})$ for fixed values
of $K$ and $L$. Furthermore, by \eqref{Unorm}, $\psi (x)\le C$ implies $|x|\le R_2$, where we set
\begin{equation}
\label{R2C} R_2\ =\ \mathcal R+\sqrt{2/K}\, (\log C)^{3/4} .
\end{equation}
Therefore, by Theorem \ref{thm:maincontraction}, the local contractivity condition (C3) is satisfied with $c$ given by \eqref{crate} provided
$h\le \min (h_\star ,h_1)$ where $h_\star$ can be chosen such that $h_{\star}^{-1}$ is of order
$O\left((1+T^{-1/2}+\mathcal R^{1/2})(d+\mathcal R^2+(\log C)^{3/2})\right)$. Hence for $h\le h_{\star\star}=\min (h_\star ,h_0,h_1)$, all parts of
Assumption \ref{ALyap} are satisfied, and thus we can apply Theorem
\ref{thm:LYAP}. By \eqref{25A}, and since $\mu\pi^n =\mu$, we obtain
$$\mathcal W_\rho (\nu\pi^n,\mu )\ \le\ e^{-cn}\mathcal W_\rho (\nu ,\mu )\, +\, 
\beta^n\lambda^{n-1}(\int\psi\, d\nu +\int\psi\, d\mu )\, \delta (C) ,$$
where $\delta (C)$ is given by \eqref{deltaC}. By \eqref{rhonorm},
\eqref{Ca} and \eqref{CR1}, this implies
\begin{eqnarray}
\label{W1bound}
\Delta (n) &=&\mathcal W^1(\nu\pi^n,\mu)\ \le\  \rn{1}+ \rn{2},\qquad \text{where}\\
\nonumber  \rn{1}&=& \exp (aR_1-cn)\Delta (0)\ =\  \exp (\frac 52(1+\mathcal R/T)-cn)\Delta (0),\qquad\text{and}\\
\nonumber  \rn{2}&=&  \exp (\frac 52(1+\mathcal R/T))\beta^n\lambda^{n-1}(\int\psi\, d\nu +\int\psi\, d\mu )\, \delta (C).
\end{eqnarray}
Choosing $n$ as in \eqref{QBNHMCn}, we obtain $\rn 1\le\epsilon /2$. Furthermore, we can ensure $\rn 2\le\epsilon /2$ and thus $\Delta (n)\le\epsilon$
by choosing $C$ sufficiently large. Indeed, by \eqref{deltaC},
\begin{equation}
\label{deltaC1}
\delta (C)\, \le\, 2\sup\left\{ \frac{\max (\varphi (x),\varphi (y))}{\max (\psi (x),\psi (y))} :
x,y\in S\text{ s.t.\ }\max (\psi (x),\psi (y))>C\right\} .
\end{equation}
Moreover, by \eqref{Unorm} and \eqref{phipsi}, for any $x\in S$,
$$\varphi (x)\, =\, |x|+2Td^{1/2}\, \le\,
\mathcal R+2Td^{1/2}+\sqrt{2/K}(\log\psi (x))^{3/4}.$$
Let $x,y\in S$ such that $\max (\psi (x),\psi (y))>C$. Without loss of generality,
we assume $\max (\psi (x),\psi (y))=\psi (x)$. Then $\log\psi (x)>\log C\ge 1$, and hence
$$ \frac{\max (\varphi (x),\varphi (y))}{\max (\psi (x),\psi (y))}\le 
\frac{\mathcal R+2T\sqrt d+\sqrt{2/K}\log\psi (x)}{\psi (x)}\le 
\frac{\mathcal R+2T\sqrt d+\sqrt{2/K}\log C}{C}.$$
Here we have used that $t\mapsto t^{-1}\log t$ is decreasing for $\log t\ge 1$.
By \eqref{deltaC1}, we see that
\begin{equation}
\label{deltaC2}\delta (C)\ \le\ 2\, \left( \mathcal R+2T\sqrt d+\sqrt{2/K}\log C\right)/C.
\end{equation}
Consequently, we have $\rn 2\le\epsilon /2$ if
\begin{equation}
\label{Cl1}
C/(u+v\log C)\ \ge\ w\, (\beta\lambda )^n,
\end{equation}
where $u:=\mathcal R+2T\sqrt d$, $v:=\sqrt{2/K}$, and
$$w\ :=\ 4\epsilon^{-1}\exp\left( \frac 52 (1+\mathcal R/T)\right)\cdot
 \left(\int\psi\, d\nu+\int\psi\, d\mu \right) .$$
 Condition \eqref{Cl1} holds if and only if
 \begin{equation}
 \label{Cl2}
 \log C\ \ge\ \log (u+v\log C)\, +\, \log w\,+\, n\log (\beta\lambda ).
 \end{equation}
 In particular, since
 $$\log (u+v\log C)\, \le\, \log^+(2u)+\log^+(2v\log C)\, \le\, \log^+u+\log^+v+2+\log\log C ,$$
 there is a universal finite constant $C_0$ such that \eqref{Cl1} is satisfied if
 $C\ge C_0$ and
  \begin{equation}
 \label{Cl3}
 \log C\ \ge\ \log^+ u\, +\, \log^+ v +\, \log w\,+\, n\log (\beta\lambda ).
 \end{equation}
 We have $\log u=\log (\mathcal R+2T\sqrt d)$, $\log v=\frac 12\log (2/K)$, and
 $$\log w\ =\ \frac 52 (1+\mathcal R/T)\,\log\left( 4\frac{\int\psi\, d\mu +\int\psi\, d\nu}{\epsilon}\right) .$$
 Furthermore, by \eqref{betalambda}, $\log (\beta\lambda )$ is of order $O(d)$, and $n$ satisfies \eqref{QBNHMCn}. Combining these observations, we see that we can ensure $\rn 2 \le\epsilon /2$ and thus $\Delta (n)\le\epsilon $ by
 choosing $\log C$ proportional to $\, dn+(1+\mathcal R/T)\log^+\left(
 \frac{\int\psi\, d\mu +\int\psi\, d\nu}{\epsilon}\right)$. The assertion
 follows since $\log^+\int\psi\, d\mu =O(d)$ and
 $$h_{\star\star}^{-1}\ =\ O\left( (1+T^{-1/2}+\mathcal R^{1/2})(d+\mathcal R^2+(\log C)^{3/2}\right) .$$ 
\end{proof}

\appendix
\section{Explicit Lyapunov functions} \label{app:explicit_lyap}
In this appendix, we prove the results on Lyapunov functions for HMC stated in Examples \ref{example:LYAP1} and \ref{example:LYAP2} and Lemma
\ref{lem:LYAP}.

\begin{proof}[Proof for Example \ref{example:LYAP1}]
Let $x_t=q_t(x,\xi )$ and $v_t=p_t(x,\xi )$ where $x\in\mathbb R^d$ and $\xi\sim N(0,I_d)$.
A simple computation shows that if \eqref{DC} holds then for $t\le T$,
\begin{eqnarray*}
\frac{d}{dt}|x_t|^2 &\le & 2x_t\cdot v_t+2hLx_T^{\star ,2},\qquad\text{and}\\
\frac{d}{dt}(x_t\cdot v_t) &\le & |v_t|^2-\kappa |x_t|^2 +C+\frac 32 hL x_T^{\star} v_T^{\star}+\frac 14 h^2L^2 x_T^{\star ,2}.
\end{eqnarray*}
Therefore,
\begin{eqnarray}
\nonumber |x_T|^2 &\le & |x|^2+2Tx\cdot \xi +2\int_0^T\int_0^t|v_s|^2ds\, dt-2\kappa \int_0^T\int_0^t|x_s|^2ds\, dt\\
\label{BXT1}&&\ +CT^2+2hLTx_T^{\star ,2}+\frac 32hLT^2 x_T^{\star} v_T^{\star}+\frac 14 h^2L^2 T^2x_T^{\star ,2}.
\end{eqnarray}
Now for $s\le T$, we can apply the a priori bounds
\begin{eqnarray}
\label{APV} |v_s| &\le & |\xi |+2Ls\max (|x|,|x+s\xi |),\\
\label{APX} |x_s-(x+s\xi )| &\le &  \max (|x|,|x+s\xi |)/10,
\end{eqnarray}
which follow from Lemma \ref{0A}. With a short computation these imply
\begin{eqnarray*}
2\int_0^T\int_0^t|v_s|^2ds\,dt &\le & (2T^2+\frac{16}{15}L^2T^6)|\xi |^2+\frac 83L^2T^4|x|^2,\\
2\int_0^T\int_0^t|x_s|^2ds\,dt &\ge & \frac{79}{200}T^2|x|^2+\frac{27}{100}T^3x\cdot\xi +\frac{27}{400} T^4|\xi |^2.
\end{eqnarray*}
By bounding the corresponding terms in \eqref{BXT1} and noting that
$x_T^\star\le \frac{11}{10}(|x|+T|\xi |)$, $v_T^\star\le \frac{6}{5}|\xi |+2LT|x|$ and $L(T^2+hT)\le \kappa /(10L)$, we conclude that
\begin{equation}
\label{BXT2}|x_T|^2\ \le\ (1-\frac{76}{600}\kappa T^2)|x|^2+(2T-\frac{27}{100}\kappa T^3)x\cdot\xi +(C+2|\xi |^2)T^2 
\end{equation}
provided that $T/h\ge n_0$ for an appropriate constant $n_0\in\mathbb N$ depending on $\kappa $ and $L$.
Taking expectations in \eqref{BXT2}, we conclude that \eqref{LyapPsi} holds for exact and unadjusted HMC with $\Psi (x)=|x|^2$. For adjusted numerical HMC, we obtain
\begin{eqnarray*}
(\pi\Psi )(x) &=& E[ |x_T|^2;A(x)]\, +\, |x|^2\,P[A(x)^C]\\
&\le & (1-\frac{76}{600}\kappa T^2+P[A(x)^C])\, \Psi (x)\, +\, (C+2d)T^2.
\end{eqnarray*}
Hence by the a priori bound on the rejection probability in Theorem \ref{AR}, we see
that \eqref{LyapPsi} holds for $|x|\le R_2$ if $R_2^3h^2$
is sufficiently small.
\end{proof}

\begin{proof}[Proof for Example \ref{example:LYAP2}]
We only sketch the proof for exact HMC.
Let $\delta >0$ and choose $\Psi (x)=g(|x|^2)$ where $g:\mathbb R_+\to\mathbb R_+$ is a smooth
increasing convex function such that $g(0)=0$, $g(s)=\exp (\delta \sqrt s)$ for $s \ge \delta^{-2}$, and $g''(s)\le 2\delta^4$ for $s\le\delta^{-2}$, and let $x_t$ and $v_t$ be defined as in the proof above. Assumption \eqref{DC1} implies that for all $t$,
\begin{equation}
\label{AP1}
|v_t-\xi |\le Qt\qquad\text{and}\qquad |x_t-(x+\xi t)|\le Qt^2/2.
\end{equation}
Noting that $\Psi (x)=\exp (\delta |x|)$ if $|x|\ge\delta^{-1}$, an explicit computation shows that if $|x_t|>\delta^{-1}$ then by \eqref{DC1} and \eqref{AP1},
\begin{eqnarray*}
\frac{d^2}{dt^2}\Psi (x_t) &=&
\left[   \delta |x_t|^{-1}\left(|v_t-\frac{x_t\cdot v_t}{|x_t|^2}x_t|^2-x_t\cdot\nabla U(x_t)\right)+\delta^2\frac{(x_t\cdot v_t)^2}{|x_t|^2}\right]\, \Psi (x_t)\\
&\le & \left[ -\kappa +\delta C+2\delta (|\xi |^2+Q^2t^2)\right]\, \delta\, \Psi (x_t)\\
&\le & \left[ -\kappa e^{-\delta t|\xi |-\delta Qt^2/2} +(C+2|\xi |^2+Q^2t^2)\delta e^{\delta t|\xi |+\delta Qt^2/2} \right]\, \delta\, \Psi (x).
\end{eqnarray*}
Similarly, one verifies that if $|x_t|\le \delta^{-1}$ then
$$\frac{d^2}{dt^2}\Psi (x_t) \ =\ \frac{d^2}{dt^2}g (|x_t|^2) \ \le\ \left( 16 |\xi |^2+16Q^2t^2+eC\right)\, \delta^2 .$$
We now choose $\delta :=\kappa /(4A+8d+Q^2T^2)$. Since by \eqref{DC1}, $\kappa \le Q$  and $\xi\sim N(0,I_d)$, this choice implies in particular that $E[\exp (\delta T|\xi |+\delta QT^2/2)]\le\sqrt 2$. Noting that 
$$\Psi (x_T)\ =\ \Psi (x)+Tg'(|x|^2)x\cdot \xi+\int_0^T\int_0^t\frac{d^2}{ds^2}\Psi (x_s)\,ds\, dt,$$
and combining the bounds above, we then obtain
$$
(\pi\Psi )(x) \ =\ \Psi (x) + \int_0^T\int_0^tE\left[\frac{d^2}{ds^2}\Psi (x_s) \right]\, ds\, dt\ \le \ \delta\kappa T^2\, \left( 5-\frac{1}{4\sqrt 2}\Psi (x)\right).
$$
\end{proof}

\begin{proof}[Proof of Lemma \ref{lem:LYAP}]
We first remark that by Assumption \ref{A123}, $x\cdot\nabla U(x)\ge K|x|^2$ for $|x|\ge\mathcal R$. Therefore, for any $x\in\mathbb R^d$,
\begin{equation}
\label{Unorm}
U(x)\ \ge\ \frac{K}{2}\min (|x|-\mathcal R,0)^2\quad\text{and}\quad |x|\ \le\ \mathcal R+\sqrt{2U(x)/K}.
\end{equation}
Furthermore, by \eqref{rho} and \eqref{f},
$$\rho (x,y)\ \le\ |x-y|\ \le\ |x|+|y|\le\ \varphi (x)+\varphi (y)\qquad \text{for any }x,y\in\mathbb R^d.$$
To verify the Lyapunov conditions recall that
$$(\pi\varphi )(x)\ =\ E[\varphi (q_T(x,\xi ));A(x)]\, +\, \varphi (x)\, P[A(x)^C]\quad\text{with }\xi\sim N(0,I_d).$$
By Lemma \ref{0A}, $|q_T(x,\xi )|\le 2(|x|+T|\xi |)$,
and thus for any $x\in \mathbb R^d$,
$$(\pi\varphi )(x)\ \le\ 2\, E\left[|x|+T|\xi |+2Td^{1/2}\right]\ \le\ 2|x|+4Td^{1/2}\ =\ 2 \varphi (x).$$
Hence (C2) is satisfied.\smallskip\\
Furthermore, by Lemma \ref{D}, there is a finite constant $C_1$ such that
\begin{eqnarray}\label{Ubound}
U(q_T(x,\xi ))& \le & H(\phi_T(x,\xi ))\ \le\ H(x,\xi )+C_1Th^2\max (|x|,|\xi |)^3\\
\nonumber &\le & U(x)+\frac 12 |\xi |^2+C_1Th^2|\xi |^3+C_1Th^2|x|^3.
\end{eqnarray}
Suppose that $\psi (x)\le C$. Then $U(x)\le (\log C)^{3/2}$, and hence by 
\eqref{Unorm},\\ $|x|\le \mathcal R+\sqrt{2/K}(\log C)^{3/4}$. Therefore, if
\eqref{LTh} holds then by \eqref{Ubound}, we obtain
$$(\pi\psi )(x)\ \le\ E\left[ \exp\left( U(x)^{2/3}+|\xi |^{4/3}+|\xi |^{2}/3+1\right)\right]
\ =\ \lambda\psi (x)$$
for any $x\in S$ such that $\psi (x)\le C$. Hence (C1) is satisfied as well.
\end{proof}

\bibliographystyle{amsplain}
\bibliography{bibhmc}

\end{document}